\documentclass[11pt,a4paper]{article}
\usepackage[english]{babel}
\usepackage[latin1]{inputenc}
\usepackage{fancyhdr}
\usepackage{amscd}
\usepackage{hyperref}
\usepackage{graphicx}
\usepackage{newlfont}
\usepackage{amssymb}
\usepackage{amsmath}
\usepackage{latexsym}
\usepackage{mathtools}
\usepackage{amsthm}
\usepackage{dsfont}
\usepackage{xcolor}
\usepackage{cancel} 
\usepackage{mathrsfs}
\newcommand{\R}{\mathbb{R}}
\newcommand{\C}{\mathbb{C}}
\newcommand{\Z}{\mathbb{Z}}
\newcommand{\N}{\mathbb{N}}
\newcommand{\Id}{\mathds{1}}
\newcommand{\ii}{\mathrm{i}}
\newcommand{\csi}{\xi}

\newcommand{\cH}{{\cal H}}

\newtheorem{lemma}{Lemma}[section]
\newtheorem{theorem}[lemma]{Theorem}
\newtheorem{proposition}[lemma]{Proposition}

\newcommand{\tn}{{\mathtt n}}
\newcommand{\tm}{{\mathtt m}}
\newcommand{\cV}{\mathcal{V}}
\newcommand{\dd}{\textnormal{d}}
\newcommand{\tc}{{\mathtt c}}

\newcommand{\HtHx}{H^{\frac 1 2 + \delta}_t \cH^{\frac 3 2 + \delta}_x}
\newcommand{\HtHeta}{H^{\frac 1 2 + \delta}_t \cH^{\frac 3 2 + \delta}_\eta}

\definecolor{darkgr}{rgb}{0.0, 0.62, 0.42}

\newcommand{\dbar}{d\hspace*{-0.08em}\bar{}\hspace*{0.1em}}
\newcommand{\omgen}{\omega}
\newcommand{\egen}{\mathrm{e}}
\numberwithin{equation}{section}

\theoremstyle{definition}

\newtheorem{definition}[lemma]{Definition}

\newtheorem{remark}[lemma]{Remark}

\title{\textbf{Time periodic solutions of completely resonant \ Klein-Gordon equations on $\mathbb{S}^3$}}

\begin{document}

 \author{Massimiliano Berti, Beatrice Langella, Diego Silimbani\footnote{
International School for Advanced Studies (SISSA), Via Bonomea 265, 34136, Trieste, Italy. 
 \textit{Emails: } \texttt{berti@sissa.it},  \texttt{beatrice.langella@sissa.it},  \texttt{dsilimba@sissa.it}
 }}

\date{}

\maketitle

\begin{abstract}
	We prove existence and multiplicity of Cantor families of small amplitude time periodic solutions of completely resonant Klein-Gordon equations on the sphere $\mathbb{S}^3$ with quadratic, cubic and quintic nonlinearity, regarded as toy models in General Relativity. The solutions are obtained by a variational Lyapunov-Schmidt decomposition, which reduces the problem to the search of mountain pass critical points of a restricted Euler-Lagrange action functional. 
	Compactness properties of its gradient are obtained by Strichartz-type estimates for the 
	solutions of the linear 
	Klein-Gordon equation on $\mathbb{S}^3$.
\end{abstract}

\tableofcontents
\section{Introduction}

Motivated by the stability problem of the anti-de Sitter space-time (AdS), the goal of this paper is to prove existence and multiplicity  of Cantor families of time-periodic solutions 
of nonlinear Klein-Gordon equations of the form
\begin{equation}\label{original.eq}
\begin{gathered}
\left(-\partial_{tt} + \Delta_{\mathbb{S}^{3}}-\mathds{1}\right)\phi= \begin{cases}
|\phi|^{p-1} \phi & \quad  \text{if } p \  \text{is  odd} \, ,\, \  p\geq 3\,,\\
\phi^p &\quad  \text{if } p \  \text{is  even}\,,
\end{cases}
\end{gathered}
\end{equation}
where $\phi: \R \times \mathbb{S}^3 \rightarrow \C$ and $\Delta_{\mathbb{S}^3}$ is the Laplace-Beltrami operator on the $3$-dimensional sphere 
$\mathbb{S}^3$. 
For $p=3$, time periodic solutions of \eqref{original.eq} have been 
very recently constructed  by Chatzikaleas and Smulevici in \cite{cha.smu}. 
A mathematical point of interest of 
 Theorems 
\ref{teo.spherical} and \ref{teo.hopf} below 
is that, jointly with \cite{cha.smu, Cha_Smu_2},  they are the only existence results 
of time periodic solutions for {\it completely resonant} Hamiltonian PDEs on 
a  manifold of dimension higher than one.
Their proof is based 
on a novel combination of {\it variational} methods and {\it Strichartz}-type estimates
 for free solutions of the Klein-Gordon equation on $\mathbb{S}^3$, 
that we find of theoretical interest in itself and
nowhere else available in literature. 

\smallskip

Let us first shortly outline the physical framework connecting \eqref{original.eq} with the stability problem of AdS space-time. AdS is the maximally symmetric solution to the vacuum Einstein equations $\textrm{Ric}(g) = -\Lambda g$ with negative cosmological constant $\Lambda$. Unlike de Sitter or Minkowski space-times, its stability properties are
nowadays 
still poorly  understood. 
In particular, the stability of AdS depends on the conformal  boundary conditions. 
While, for instance, it is expected that under dissipative boundary conditions 
AdS is stable, see \cite{Holzegel-Luk-Smulevici-Warnick}, 
it has been conjectured, by Dafermos and Holzegel in \cite{Dafermos-Holzegel} and by Anderson in  \cite{Anderson2006}, 
that AdS is unstable
under fully reflective boundary conditions.
The latter instability conjecture is supported by the numerical investigations 
of Bizon-Rostworowski   \cite{Bizon-Rostwo} for  
the spherically symmetric Einstein-massless-scalar field equations,
suggesting that AdS is unstable, against the formation of black holes under arbitrarily small perturbations. Notwithstanding, 
the work \cite{Bizon-Rostwo} also suggests the existence of small initial data leading to stable solutions, confirmed later by 
Maliborski-Rostworowski  \cite{Mali-Rostwo} 
who constructed formal  
time periodic solutions, 
supported by numerical evidences. The same existence conjecture of time periodic solutions --called geons-- has also been extended to the vacuum Einstein equations in \cite{Dias-Horowitz-Santos, Dias-Horowitz-Marolf-Santos}.

\smallskip

The nonlinear wave equation \eqref{original.eq} with $p=3$ has been introduced in \cite{Bizon-Rostwo, Mali-Rostwo, Bizon-Craps-altri} as a toy model of spherically symmetric Einstein-massless-scalar field equations close to the AdS solution.
In \cite{ch0} Chatzikaleas constructed formal power series expansions of small amplitude time periodic solutions of \eqref{original.eq} in the spherically symmetric case, which reduces to the 1d wave equation with singular nonlinearity
\begin{equation}\label{vista.in.1d}
-\partial_{tt} u + \partial_{xx} u  = \frac{u^3}{\sin^2(x)}\,, \quad u(t, 0) = u(t, \pi) = 0\,, \quad  x \in (0, \pi)\,.
\end{equation}
The absence of secular terms in the power series expansions is obtained using the method of Maliborski and Rostworowski \cite{Mali-Rostwo}, developed for the Einstein-Klein-Gordon equation. However, the presence of \emph{small divisors} prevents the convergence of such power series. This difficulty looks analogous to the convergence problem of ``Linstedt series"  of quasi-periodic solutions in Celestial Mechanics, devised since Poincar\'{e} \cite{Poincare}, and successfully overcome during the last century by the celebrated KAM theory. 
The first rigorous existence result of time periodic solutions of \eqref{original.eq} for $p=3$ with {strongly Diophantine} frequencies $\omega$  is given in the very recent paper \cite{cha.smu}. Such work constructs solutions of the following form:
\begin{itemize}
 	\item \textbf{spherically symmetric} functions, namely $\phi(t,x) = u(t, \cos(x))$, $x \in (0, \pi)$, see Definition \ref{def:spherical} below;
 	\item \textbf{plane waves in Hopf coordinates}, namely $\phi(t, \eta, \csi_1, \csi_2) = u(t, \eta) e^{\ii \mu_1 \csi_1} e^{\ii \mu_2 \csi_2}$, see Definition \ref{def:hopf.ansatz}, up to restricting to values of the momenta $\mu_1 = \mu_2 \in \{1, \dots, 5\}$, or $\mu_1 = \mu_2$ large enough.
\end{itemize}
The results in \cite{cha.smu}
rely on an abstract theorem by Bambusi and Paleari \cite{Bambusi-Paleari2001}, which uses a Lyapunov-Schmidt approach and whose main assumption is the existence of a non degenerate zero of the ``resonant system". 

\smallskip
The goal of this paper is to prove  existence and multiplicity of periodic solutions of \eqref{original.eq}, for more general values of the nonlinearity degree $p$ and of the momenta $\mu_1, \mu_2$. More precisely, for
\begin{itemize}
\item $p=2$ and $p=5$,  we find spherically symmetric solutions, see Theorem \ref{teo.spherical};
\item  $p=3$, we find plane waves in Hopf coordinates for \emph{any} value of the momenta  $\mu_1, \mu_2 \in \Z$, see Theorem \ref{teo.hopf}.
\end{itemize}
These generalizations require new methods, since the verification of the existence of a nondegenerate zero of the associated resonant system seems unapproachable, if ever true. In this work we combine variational methods of mountain pass type, inspired by the works of Berti and Bolle \cite{Berti-BolleCMP, Berti-BolleNA, Berti-BolleCantor} for 1-d semilinear wave equations, 
with Strichartz-type estimates for the linear Klein-Gordon equation on $\mathbb{S}^3$. 

We now present rigorously our results.

\subsection{Main Results}\label{subsec.main.results}

 Small amplitude time periodic solutions of \eqref{original.eq}  bifurcate from suitable solutions of the linear Klein-Gordon equation
\begin{equation}\label{lin.eq}
-\partial_{tt} v + \Delta_{\mathbb{S}^{3}} v - v = 0\,.
\end{equation}
Since the eigenvalues of $ - \Delta_{\mathbb{S}^{3}}  + \mathds{1}  $ are $ \{ j^2, j \in \N_* \} $, where $\N_* := \{1, 2, 3\dots\}$,
{\it all} the solutions of \eqref{lin.eq} are $2\pi$-periodic in time, i.e. have frequency $\omega = 1$. For this reason \eqref{original.eq} is called a \emph{completely resonant} equation
and
a major difficulty of the problem 
is to determine from which
 free solutions $ v $ of  \eqref{lin.eq} 
periodic solutions of the nonlinear Klein-Gordon equation 
\eqref{original.eq} branch off.

 We look for  time periodic solutions of \eqref{original.eq} with
 strongly Diophantine frequency $\omega \sim 1$ belonging to the Cantor set 
\begin{equation}\label{Omega.gamma}
 \Omega_\gamma := \left\{ \omega \in \left[\frac 1 2, 2\right] \ : \ |\omega \ell - j | \geq \frac{\gamma}{\ell} \quad \forall \ell \in \N_*,\ j \in \N\,, \quad \ell \neq j \right\}\,.
 \end{equation}
For $\gamma \in (0, \gamma_0)$ and $\gamma_0$ small enough, the set $\Omega_\gamma$ is uncountable, with zero measure, and accumulates to $\omega = 1$, as proved in \cite{Bambusi-Paleari2001}.
\label{sobsf}

We look for time periodic solutions of \eqref{original.eq} taking values into Sobolev spaces 
$H^{s}(\mathbb{S}^3, d \sigma)$  of scalar functions $\phi:\mathbb{S}^3 \rightarrow \C$ with $s > \frac 3 2$, where
\begin{equation}\label{Hs.std}
H^s(\mathbb{S}^3, d \sigma) := \Big\{ \phi \in L^2(\mathbb{S}^3, d \sigma)\ :\ (-\Delta_{\mathbb{S}^{3}} + \Id)^{\frac s 2} \phi \in  L^2(\mathbb{S}^3, d \sigma)\Big\}\,, \quad s \in \mathbb{R}\,,
\end{equation}
 and $d\sigma$ denotes the standard Lebesgue measure on the sphere. 
 Each $H^s(\mathbb{S}^3, d \sigma)$ is an Hilbert space endowed with the complex 
 scalar product $\langle \phi_1,\phi_2 \rangle_{H^s(\mathbb{S}^3,d \sigma)}
 :=
 \langle (-\Delta_{\mathbb{S}^3}+\mathds{1})^s\phi_1,\phi_2 
 \rangle_{L^{2}(\mathbb{S}^3,d \sigma)}$.
 
For any $s > \frac 3 2$  the spaces $H^s(\mathbb{S}^3, d \sigma)$ 
continuously embed into 
$ L^\infty (\mathbb{S}^3) $  and form an algebra with respect to the product of functions. 
 We also remind the continuous embedding of the Sobolev spaces
\begin{equation}\label{sob.sigh}
H^s(\mathbb{S}^3, d\sigma) \hookrightarrow L^{p}(\mathbb{S}^3, d\sigma)\,, \quad p \leq p^*(s,3):=\frac{6}{3-2s}\,,
\end{equation}
with compact embedding for $p < p^*(s,3)$. In particular $H^s(\mathbb{S}^3, d\sigma) \hookrightarrow L^{6}(\mathbb{S}^3, d\sigma)$ for any $s \geq 1$, with compact embedding if $s > 1$, and $H^s(\mathbb{S}^3, d\sigma) \hookrightarrow L^{4}(\mathbb{S}^3, d\sigma)$ for any $s \geq \frac 3 4$, with 
compact embedding if $s > \frac 3 4$.

 Our first existence result concerns spherically symmetric solutions, according to the following
 
 \begin{definition}[Spherically symmetric functions]\label{def:spherical}
 	Consider on $\mathbb{S}^3$ coordinates
 	\begin{equation}\label{sph.coord}
 	\scalebox{0.9}{$(0, \pi) \times (0, \pi) \times (0, 2\pi) \ni (x, \theta, \varphi) \mapsto  (\cos(x), \sin(x) \cos(\theta), \sin(x) \sin(\theta) \cos(\varphi), \sin(x) \sin(\theta) \sin(\varphi))\,.$}
 	\end{equation}
 	We say that $\phi: \mathbb{S}^3 \rightarrow \C$ is \emph{spherically symmetric} if 
 	\begin{equation}
 	\phi(x,\theta, \varphi) = u(x) \otimes 1_{\theta, \varphi}\,, \quad \forall ( x, \theta, \varphi) \in (0, \pi) \times (0, \pi) \times (0, 2\pi)\,, \quad u : (0, \pi) \rightarrow \C\,,
 	\end{equation}
 	where $1_{\theta, \varphi}$ is the function identically equal to $1$ for any $(\theta, \varphi)$. By \eqref{sph.coord}, $u$ has to be of the form $u(x) = U(\cos(x))$, for some $U:(-1, 1) \rightarrow \C$. We say that $\phi: \R \times \mathbb{S}^3 \rightarrow \C$ is \emph{spherically symmetric} if $\phi(t, \cdot)$ is spherically symmetric for any $t \in \R$.
 \end{definition}
	Our first result is the following:
\begin{theorem}[Spherically symmetric solutions]\label{teo.spherical}
	Let $p=2$ or $p=5$. Fix $\gamma \in (0, \gamma_0)$ and $\mathtt{d}\in (0, \frac{1}{4})$. For any $n \in \N$, $r> \frac 1 2$, and $s > \frac 3 2$, there exist $\varepsilon_0 := \varepsilon_0(n, r, s, \mathtt{d}, \gamma)>0$ and $C := C(n, r, s, \mathtt{d})>0$ such that for any $\varepsilon$ belonging to
	\begin{equation}\label{omega.ep}
	\mathcal{E} := \big\{\varepsilon \in (0, \varepsilon_0) \ :\ \omega_\varepsilon \in \Omega_\gamma \big\}\,, \quad \omega^2_\varepsilon := 1 + \varsigma\varepsilon\,, \quad \varsigma := \begin{cases}
	-1 & \text{if} \quad p= 2\,,\\
	1& \text{if} \quad p = 5 \, ,
	\end{cases} 
	\end{equation}
	with $\Omega_\gamma$ defined in \eqref{Omega.gamma}, there exist $n$ different real valued, non zero, $T_\varepsilon$-periodic solutions
	$$
	\{\phi^{(1)}_{\varepsilon}(t,x), \dots, \phi^{(n)}_{\varepsilon}(t,x)\}
	$$
	of equation \eqref{original.eq} with frequency $\omega_\varepsilon  := \frac{2\pi}{T_\varepsilon}$,
	even in time and with spherical symmetry. They are of the following form:
	\begin{enumerate}
	\item[(i)] There exist $n$
	different $2\pi$-periodic, even in time, non zero solutions  $\{v^{(1)}_{\varepsilon}(t,x), \dots, v^{(n)}_{\varepsilon}(t,x)\}$ of the linear equation \eqref{lin.eq},
	spherically symmetric, with size
	\begin{equation}\label{nero}
	C^{-1}\varepsilon^{\frac 1 q} \leq \| v^{(k)}_{\varepsilon}\|_{H^{r}\left([0, 2\pi], H^{s}(\mathbb{S}^3, d \sigma)\right)} \leq C {\varepsilon}^{\frac{1}{q} -\mathtt{d}}\,, \quad q := \begin{cases}
	2 &\text{if} \quad  p=2\,,\\
	4 &\text{if} \quad  p=5\,,
	\end{cases}
	\end{equation}
	such that, as $\varepsilon \rightarrow 0$,
	\begin{equation}\label{bianco}
	\| \phi^{(k)}_{\varepsilon} - v^{(k)}_{\varepsilon}(\omega_\varepsilon \cdot, \cdot )\|_{H^{r}\left([0, T_{\varepsilon}], H^{s}(\mathbb{S}^3, d \sigma)\right)} = o \left(\| v^{(k)}_{\varepsilon}(\omega_\varepsilon \cdot, \cdot)\|_{H^{r}\left([0, T_\varepsilon], H^{s}(\mathbb{S}^3, d \sigma)\right)}\right)\,.
	\end{equation}
	\item[(ii)] Each $\phi^{(k)}_{\varepsilon}$ 
	has minimal period $ T_{k, \varepsilon} := \frac{T_\varepsilon}{m_k}$
	 where $\{m_k\}_{k=1}^n$ is an increasing sequence of positive integers. Correspondingly, the functions $\{v^{(1)}_{\varepsilon}, \dots, v^{(n)}_{ \varepsilon}\}$ have minimal periods $T_k := \frac{2\pi}{m_k}$.
	\end{enumerate}
\end{theorem}
 We point out that Theorem \ref{teo.spherical} holds also in the case $p=3$. This is actually the result in \cite{cha.smu}.
 In this case the $v^{(j)}_{\varepsilon}$ are close to the ``one mode" functions
\begin{equation}\label{one.mode}
{\varepsilon}^{\frac 1 2} \underline{v}^{(j)}\,, \quad \underline{v}^{(j)} := \kappa_j \cos((j+1)t) e_j(x)\,, 
 \quad e_j(x) :=  \frac{\sin((j+1) x)}{\sin(x)}\,,
\end{equation}
for suitable $\kappa_j \in \R \setminus\{0\}$.
Note that $e_j$ are the spherically symmetric functions $e_j(x) = U_j(\cos(x))$, where $U_j: \R \rightarrow \R$ are Chebychev polynomials of second kind.
The functions $\underline{v}^{(j)}$ in \eqref{one.mode} are actually solutions of the 
``resonant system"
\begin{equation}\label{0.bif.3}
	(-\Delta_{\mathbb{S}^3} + \mathds{1}) v - \Pi_{V} (v^3) = 0\,,
\end{equation}
where $\Pi_V$ is the $L^2$-projector on the infinite dimensional linear space $ V $ formed by the solutions of \eqref{lin.eq} (see the definitions \eqref{def.V}, \eqref{def.Pi} below). On the other hand, for $p=5$ the functions $v^{(j)}_{ \varepsilon}$ in \eqref{nero} are not close to ``one modes" as in \eqref{one.mode}. 
Actually, the $v_\varepsilon^{(j)}$ are close to functions of the form $\varepsilon^{\frac{1}{4}} \underline{v}^{(j)}$,
where $\underline{v}^{(j)}$ are non zero solutions of the equation
\begin{equation}\label{0.bif.5}
	(-\Delta_{\mathbb{S}^3} + \mathds{1}) v - \Pi_V(v^5) = 0\,,
\end{equation}
which \emph{does not} possess one mode solutions.
We actually prove the existence of non trivial solutions of \eqref{0.bif.5}, 
exploiting that it is the Euler-Lagrange equation of the action functional
\begin{equation}\label{action.intro}
\frac{1}{2}\|v\|^2_{\cV^1_{t,z}} - \frac{1}{6}\int_{\mathbb{T} \times \mathbb{S}^3} v^{6}(t, z) \, dt d \sigma(z)\,, \quad \|v\|_{\cV^1_{t,z}} := \|v\|_{L^\infty(\mathbb{T}_t, H^1(\mathbb{S}^3, d \sigma))}\,, \quad \mathbb{T}:= \mathbb{R}/(2\pi \Z) \, , 
\end{equation}
which, thanks to the time-space Strichartz-type estimates proved 
in Section \ref{sec.Stri}, admits mountain pass critical points of class $C^\infty$. Strichartz estimates are required to imply compactness properties of the action functional, which are not a consequence of Sobolev embeddings  \eqref{sob.sigh} on $\mathbb{S}^3$, see Remark \ref{rmk.mountain.pass}.

The case $p=2$ is degenerate, since $\Pi_V(v^2) = 0$ (see Lemma \ref{lem.vsquare}), and the $v^{(j)}_\varepsilon$ are close to functions of the form $\varepsilon^{\frac 1 2} \underline{v}^{(j)}$, where $\underline{v}^{(j)}$ are non zero solutions of the equation
\begin{equation}\label{0.bif.2}
	(-\Delta_{\mathbb{S}^3} + \mathds{1}) v +\Pi_V \left( v \mathcal{L}_1^{-1} (v^2)\right) = 0\,, \quad \mathcal{L}_1 := -\partial_{tt} + \Delta_{\mathbb{S}^3} - \mathds{1}\,.
\end{equation}
It turns out that equation \eqref{0.bif.2} admits mountain pass critical points as well. Further comments are postponed after Theorem \ref{teo.hopf}.

 In the case $p=3$ we have new existence results of periodic Hopf plane waves solutions of \eqref{original.eq} for any value of the momenta $(\mu_1, \mu_2)$, which we now define:
\begin{definition}[Hopf plane waves]\label{def:hopf.ansatz}
	Consider on $\mathbb{S}^3$ Hopf coordinates
	\begin{equation}
	\left(0, \frac{\pi}{2}\right)
	\times \mathbb{T} \times \mathbb{T} \ni (\eta, \xi_1, \xi_2) \mapsto (\sin(\eta) \cos(\xi_1), \sin(\eta) \sin(\xi_1), \cos(\eta) \cos(\xi_2), \cos(\eta) \sin(\xi_2))\, . 
	\end{equation}
Given $(\mu_1, \mu_2) \in \Z^2$, we say that $\phi: \mathbb{S}^3 \rightarrow \C$ is a \emph{Hopf plane wave with momentum $(\mu_1, \mu_2)$} if 
	\begin{equation}\label{hopf.waves}
	\phi( \eta, \xi_1, \xi_2) = u( \eta) e^{\ii \mu_1 \xi_1} e^{\ii \mu_2 \xi_2}\,, \quad \forall (\eta, \csi_1, \csi_2) \in \left(0, \frac{\pi}{2}\right)
	\times \mathbb{T} \times \mathbb{T}\,, \quad u: \left(0, \frac{\pi}{2}\right) \rightarrow \R\,.
	\end{equation}
	We say that $\phi :  \R \times  \mathbb{S}^3 \rightarrow \C$ is a \emph{Hopf plane wave with momentum $(\mu_1, \mu_2)$} if $\phi (t, \cdot)$ is a Hopf plane wave with momentum $(\mu_1, \mu_2)$  for any $ t \in \R$.
\end{definition}
The following result extends \cite{cha.smu}, which holds for $\mu_1 = \mu_2$ either equal to $\{1, 2, 3, 4, 5\}$  or large enough:
\begin{theorem}[Hopf plane waves]\label{teo.hopf}
	Let $p=3$. Fix $\gamma \in (0, \gamma_0)$ and $\mathtt{d} \in (0, \frac{1}{2})$. For any $n \in \N$, $r > \frac 1 2$, $s > \frac 3 2$,  and any $(\mu_1, \mu_2) \in \Z^2$, there exist $\varepsilon_0 := \varepsilon_0(n, r, s, \mathtt{d}, \gamma, \mu_1, \mu_2)>0$ and $C := C(n, r, s, \mathtt{d}, \mu_1, \mu_2)>0$ such that for any $ \varepsilon$ belonging to the set
	\begin{equation}\label{omega.3}
	\mathcal{E}:= \{\varepsilon \in (0, \varepsilon_0) \ :\ \omega_\varepsilon \in \Omega_\gamma\}\,, \quad \omega^2_\varepsilon := 1 + \varepsilon\,,
	\end{equation}
	there exist $n$ different non zero, $T_\varepsilon$-periodic Hopf wave solutions (see \eqref{hopf.waves}),
	$$
		\{\phi^{(1)}_{\varepsilon}, \dots, \phi^{(n)}_{\varepsilon}\}:= \{\phi^{(1)}_{\varepsilon, \mu_1, \mu_2}, \dots, \phi^{(n)}_{\varepsilon, \mu_1, \mu_2}\}
	$$
	of equation \eqref{original.eq}, with frequency $\omega_\varepsilon := \frac{2\pi}{T_\varepsilon}$,  even in time. They are of the following form:
	\begin{enumerate}
	\item[(i)] There exist $n$ different $2\pi$-periodic non zero Hopf wave solutions $\{v^{(1)}_{ \varepsilon, \mu_1, \mu_2}, \dots, v^{(n)}_{\varepsilon, \mu_1, \mu_2}\}$ of the linear equation \eqref{lin.eq}, with size
	\begin{equation}\label{nero.hopf}
	C^{-1} \varepsilon^{\frac 1 2} 	\leq \| v^{(k)}_{\varepsilon, \mu_1, \mu_2}\|_{H^{r}\left([0, 2\pi], H^{s}(\mathbb{S}^3, d \sigma)\right)} \leq C \varepsilon^{\frac 1 2 - \mathtt{d}}\,,
	\end{equation}
	such that $\phi^{(k)}_{\varepsilon, \mu_1, \mu_2} = v^{(j)}_{\varepsilon, \mu_1, \mu_2} + o(v^{(k)}_{\varepsilon, \mu_1, \mu_2})$ as in \eqref{bianco}.
	\item[(ii)]	Each $\phi^{(k)}_{\varepsilon, \mu_1, \mu_2}$ has minimal period $T_{k, \varepsilon} := \frac{T_\varepsilon}{m_k}$, where $\{m_k\}_{k=1}^n$ is an increasing sequence of positive integers, and each $v^{(k)}_{\varepsilon, \mu_1, \mu_2}$ has minimal period $T_{k} := \frac{2\pi}{m_k}$.
	\end{enumerate} 
\end{theorem}

It is proved in \cite{cha.smu} that for any $\mu_1, \mu_2$ there exist one mode Hopf plane wave solutions of the resonant system \eqref{0.bif.3}, of the form
\begin{equation}\label{one.mode.hopf}
{\varepsilon}^{\frac 1 2} \underline{v}^{(j)}\,, \quad \underline{v}^{(j)} := \kappa_j \cos(\omega^{(\mu_1, \mu_2)}_j t) \egen^{(\mu_1, \mu_2)}_j(\eta) e^{\ii \mu_1 \csi_1} e^{\ii \mu_2 \csi_2}\,, \quad \kappa_j \in \R \setminus\{0\}\,,
\end{equation}
where $\egen^{(\mu_1, \mu_2)}_j$ are eigenfunctions of $-\Delta_{\mathbb{S}^3} + \mathds{1}$ with eigenvalue $\omega^{(\mu_1, \mu_2)}_j := 2 j + |\mu_1| + |\mu_2| + 1$. However, the proof that they are non degenerate is obtained only for $\mu_1 = \mu_2$ and either $\mu_1 \in \{0, \dots, 5\}$ or $\mu_1$ large enough, with explicit computations performed with Mathematica code.  
The proof of the more general Theorem \ref{teo.hopf} is obtained exploiting variational methods.

We make the following comments, common to both Theorems \ref{teo.spherical} and \ref{teo.hopf}:
\begin{enumerate}
	\item (Regularity and multiplicity) If $r > \frac{5}{2}$ and $s > \frac{7}{2}$ the solutions $\{\phi^{(j)}_{ \varepsilon}\}$ of \eqref{original.eq} proved in Theorems \ref{teo.spherical} and \ref{teo.hopf} are \emph{classical}. Actually, the smoother we require the solutions to be in time and in space (i.e., the larger $r, s$ are), the smaller $\varepsilon_0(r, s, n)$ has to be. Analogously, the larger is the number of solutions $n$, the smaller $\varepsilon_0(r, s, n)$ has to be.
	\item (Minimal periods) The solutions $\phi^{(k)}_{\varepsilon}$ of \eqref{original.eq} whose existence is stated in Theorems \ref{teo.spherical}, \ref{teo.hopf} are \emph{geometrically distinct}, having  different minimal time periods $\frac{T_{\varepsilon}}{m_k}$, as stated in Items (ii).
	\item (Increasing norms) The  functions $v^{(k)}_\varepsilon$ turn out to have increasing norms in $k = 1, \dots, n$, although, for simplicity, we have stated estimates \eqref{nero} and \eqref{nero.hopf} uniformly in $k$. 
	\item (Critical exponent) The restrictions on the exponents $p \leq 5$ and $p \neq 4$ are {\it not} technical.
	In the critical case $p=5$, the functional
	\begin{equation}\label{I.p}
	\mathcal{G}_{p+1}(v):= \frac{1}{p+1}\int_{\mathbb{T} \times \mathbb{S}^3} v^{p+1}(t, z) \, dt d \sigma(z)\,
	\end{equation}
	associated to the nonlinearity is finite for any $v$ in the space  $L^\infty(\mathbb{T}_t, H^1(\mathbb{S}^3, d \sigma))$ (which appears in \eqref{action.intro}) by Sobolev embedding $H^1(\mathbb{S}^3, d \sigma) \hookrightarrow L^6(\mathbb{S}^3, d \sigma)$. However, it follows to have \emph{compact} gradient by the Strichartz estimates in Proposition \ref{lem.Strichartz.1} (see Remark \ref{rmk.mountain.pass}). For the supercritical exponents $p \geq 7$, the functional $\mathcal{G}_{p+1}(v)$ is not expected to be well defined for any 
	$v $ in $ L^\infty(\mathbb{T}_t, H^1(\mathbb{S}^3, d \sigma))$. If $p=4$, then $\mathcal{G}_5(v) \equiv 0$, as well as for all even values of $p$.
		Then the leading term in the action functional of the corresponding resonant system turns out to have degree $8$, which is supercritical. 
\end{enumerate}
  As already mentioned, Theorems \ref{teo.spherical} and \ref{teo.hopf} are inspired by the variational approach of \cite{Berti-BolleCMP, Berti-BolleNA, Berti-BolleCantor, Berti_book}, developed for $1$-d semilinear completely resonant wave equations
  $-\partial_{tt} u + \partial_{xx} u = u^{p} + \dots $ with Dirichlet boundary conditions. Major difficulties with respect to these works arise because of the $3$-dimensional manifold 
  $\mathbb{S}^3$. 
  This becomes evident for instance in the search of spherically symmetric solutions of \eqref{original.eq}, that reduces to solve the wave equation
  \begin{equation}\label{eq.singularity}
  -\partial_{tt} u + \partial_{xx} u = \frac{u^p}{\sin^{p-1}(x)}\,, \quad x \in (0, \pi)\,, \quad
  u(t, 0) = u(t, \pi) = 0\,,
  \end{equation}
  which has a singular nonlinearity at $x=0, \pi$.
  Before explaining the main difficulties and  ideas of our proof, we present a few related results.
  
\smallskip
  
\noindent \textbf{Related literature.}
The first existence results of $2\pi$-periodic solutions for completely resonant wave equations $\partial_{tt} u - \partial_{xx} u = |u|^{p-2} u$, $p>2$, have been proved by Rabinowitz starting with \cite{Rabi2}, via global variational methods. These techniques, as well as those in \cite{Brezis-Coron-Nirenberg, Brezis-Nirenberg},  enable to find periodic orbits with rational frequency, the reason being that other periods give rise to a small denominator problem.

Independently of these global results, the local bifurcation theory of periodic and quasi-periodic solutions was initiated for \emph{non resonant} 1-d Klein-Gordon equations by Wayne \cite{Wayne1990}, Kuksin \cite{Kuksin2000}, Craig and Wayne \cite{Craig-Wayne1993}, Poschel \cite{Poschel1996}, Chierchia and You \cite{Chierchia-You2000}, with KAM methods. For semilinear Klein-Gordon equations on $\mathbb{T}^d$ with convolution potentials, the first result is due to Bourgain in \cite{Bourgain2005}, later extended by \cite{Berti-Bolle2012, Berti-Bolle_book} for multiplicative potentials. Bifurcation for periodic and quasi-periodic solutions of non resonant Klein-Gordon equations 
was obtained in \cite{Berti-Procesi2011, Berti-Corsi-Procesi} for Lie Groups and homogeneous manifolds, in \cite{Berti-Bolle-Procesi} for Zoll manifolds, and in \cite{Grebert-Paturel} for the sphere $\mathbb{S}^d$. These results do not cover the completely resonant case \eqref{original.eq}, where all the linear frequencies of oscillations are integers. 

The first existence results of Cantor families of small amplitude 
time periodic solutions of 1-d \emph{completely resonant} wave equations $-\partial_{tt} u + \partial_{xx} u = u^p$, $p=3$, was proved in \cite{Lid-Shulman} under periodic boundary conditions and in \cite{Bambusi-Paleari2001} for Dirichlet boundary conditions, for frequencies belonging to the zero measure set \eqref{Omega.gamma}. 
The latter result was then generalized 
in \cite{Berti-BolleCMP, Berti-BolleNA} to arbitrary exponents $p$, using variational methods. Existence of periodic solutions for a set of frequencies $\omega \sim 1$ of density one was proved in \cite{Berti-BolleCantor, Berti-BolleNodea} via  Nash-Moser implicit function 
techniques, and in \cite{Gentile-Mastropietro-Procesi} via trees resummation arguments. Existence of time quasi-periodic solutions with two frequencies of completely resonant nonlinear wave equations on the circle were obtained in \cite{Procesi2005} and \cite{Berti-Procesi2005}. 

For completely resonant wave equations, or even more general Hamiltonian PDEs in dimension higher than one, not much is known about time periodic solutions besides the aforementioned paper \cite{cha.smu} and the present work.

\subsection{Ideas of proof}\label{sec.ideas}

In order to 
look for bifurcation of small amplitude 
time periodic solutions of \eqref{original.eq} 
with frequency $ \omega \sim 1 $ 
a 
natural approach is to implement a Lyapunov-Schmidt decomposition
in the spirit of 
\cite{Berti-BolleCMP, Berti-BolleNA,Berti-BolleCantor} for
$ 1d$ semilinear wave equations.
Major difficulties 
arise due to the higher dimension of the space domain, here the sphere $\mathbb{S}^3$, as we now explain. 
After a time rescaling, 
we look for $ 2 \pi $-periodic in time real solutions $u (t, z) $
of  $ - \omega^2 \partial_{tt} u + \Delta_{\mathbb{S}^{3}} u -u = u^p $.
By splitting  
$$ 
u  = v + w \, , \quad v := \Pi_V u \, , \quad w := \Pi_W u \, , \quad \Pi_W := \mathds{1} - \Pi_V \, , 
$$ 
where $V$ is the kernel of the operator $-\partial_{tt} + \Delta_{\mathbb{S}^3} - \mathds{1}$
(namely the space of solutions of the free Klein-Gordon equation \eqref{lin.eq})
and   $\Pi_V$ the corresponding orthogonal projector, 
it amounts to the 
system
\begin{gather}
\tag{Bif. eq} \label{Bif.eq}
(\omega^2 - 1) (-\Delta_{\mathbb{S}^3} + \mathds{1}) v = \Pi_V ((v +w)^p)\,,\\
\tag{Range eq} \label{Range.eq}
\left(-\omega^{2} \partial_{tt} + \Delta_{\mathbb{S}^3} - \mathds{1}  \right) w = \Pi_{W} ((v + w)^p)
\, . 
\end{gather}
For any $\omega \in \Omega_\gamma $ 
the operator $ \mathcal{L}_\omega:= -\omega^{2} \partial_{tt} + \Delta_{\mathbb{S}^3} - \mathds{1}   $
is invertible on the range $ W := V^\bot$ and, for any fixed $v \in V$ small enough (in some suitable norm),   
 one may solve first 
the range equation, obtaining $w=w(v) = o( v)$, by a contraction argument. 
Here, in order to control the nonlinearity $ (v + w)^p $, 
it is natural to
close the contraction in Sobolev spaces   
 $$ 
 H^r_t H^s_z := H^r( \mathbb{T}_t , H^s(\mathbb{S}^3, d \sigma)) \, , \quad 
 r > \frac 1 2 \, , \quad s > \frac 3 2 \, , 
 $$
 which are an algebra with respect to the product of functions, and where $\left.\mathcal{L}_\omega^{-1}\right|_W$ is bounded. This requires to take
 $ v $ small enough in 
$ H^r_t H^s_z $ 
as well, 
  which  amounts, for functions in the kernel $ V $,  
	 to require that 
	\begin{equation}\label{vdove}
	 \| v \|_{\cV^{r+s}_{t,z}} \ll 1 \, , \qquad 
	 \cV^{r+s}_{t,z}:=L^\infty(\mathbb{T}_t, H^{r+s}(\mathbb{S}^3, d \sigma))\cap V \, , \quad r + s > 2 
	 \, . 
	\end{equation}
 On the other hand, one needs then to solve the bifurcation equation \eqref{Bif.eq} with $w = w(v)$. 
 As observed in \cite{Berti-BolleCMP, Berti-BolleNA},  this
  turns out to be the Euler-Lagrange equation of the reduced action functional
 \begin{equation}\label{big.Pharma}
	 \Phi(v) := \frac{(\omega^2 -1)}{2} \| v\|_{\cV^{1}_{t,z}}^2 - \frac{1}{p+1} \int_{\mathbb{T} \times \mathbb{S}^3} (v + w(v))^{p+1} \,\dd t \dd \sigma\,.
 \end{equation}
A serious problem which arises is thus the following:  
\begin{itemize}
	\item[\ ]{\bf Problem:} The natural space to find mountain pass 
	critical points for the functional $\Phi$ in \eqref{big.Pharma}  is 
	(a small ball in) the space $\cV^1_{t,z}$ (modeled with an $ H^1 $-norm), 
	associated to its quadratic part. This is clearly in contradiction with solving 
	 the  \eqref{Range.eq} on the much smaller domain $ \{  \| v \|_{\cV^{2+}_{t,z}} \ll 1 \} $ \footnote{Here, if $a \in \R$, by $a+$ we mean a number greater than $a$.}
	 in \eqref{vdove}.
	 How to fill this regularity gap? 
\end{itemize}

We remark that the previous difficulty \emph{does not} disappear 
restricting to search solutions which depend on only one space variable, as spherically symmetric functions or Hopf waves. This is evident for instance in the spherically symmetric case, where the reduced equation \eqref{eq.singularity}
has a {\it singular} nonlinearity. If $p=3$, this issue is overcome (cfr. \cite{cha.smu}) noting that the functional $\Phi$  in \eqref{big.Pharma} possesses non degenerate critical points of the explicit form $v= \varepsilon^{\frac 1 2} (\underline{v} + \dots)\,,$
where $\underline{v}$ is a one mode function as in \eqref{one.mode}, which belong to 
 $ \{  \| v \|_{\cV^{r+s}_{t,z}} \ll 1 \} $ for any $ r+s > 2 $.

\smallskip

We now describe our strategy. For simplicity, we focus on the case $p=5$
and we restrict on spherically symmetric functions. The seminal idea is to note that, neglecting $w(v)$,
the functional $ \Phi (v) $ in \eqref{big.Pharma} is a 
perturbation of  the  ``resonant system" functional 
 \begin{equation}\label{big.Pharma0}
\Phi_0(v) :=  \frac{\varepsilon}{2} \| v \|_{\cV^{1}_{t,z}}^2 
  - \mathcal{G}_{p+1}(v)\,, \quad \mathcal{G}_{p+1}(v):= \frac{1}{p+1} \int_{\mathbb{T} \times \mathbb{S}^3} v^{p+1} \,\dd t \dd \sigma\,, \quad \varepsilon:= \omega^2 -1>0\,.
 \end{equation}
 The Strichartz estimate \eqref{item1st1} implies that $\mathcal{G}_6$ is well defined on $\cV^1_{t,z}$ and its gradient $\nabla_{\cV^{1}_{t,z}} \mathcal{G}_6$ is a bounded map from $\cV^{s}_{t,z}$ to $ \cV^1_{t,z}$ for any $s > \frac 5 6$, thus compact on $\cV^1_{t,z}$.
Thus $\Phi_0$ possesses a mountain pass critical point $v \in \cV^1_{t,z}$ (see \cite{Ambro-Rabi}), which by homogeneity has the form $v = \varepsilon^{\frac 1 4} \underline{v}$, where 
$\underline{v}$ solves the rescaled equation $v  = (-\Delta_{\mathbb{S}^3} + \Id)^{-1} \Pi_V v^5$. Such $\underline{v}$ is not a one mode function, but  it is $C^\infty$ by the following bootstrap argument.
By the Strichartz estimate \eqref{item1st1}, one has
$$
\| \underline{v}\|_{\cV^{\frac 7 6 - \delta}_{t,z}} = \| \Pi_V {\underline{v}}^5
\|_{\cV^{-(\frac 5 6 + \delta)}_{t,z}} 	 = \sup_{h \in \cV^{\frac 5 6 + \delta}_{t,z}\,,\ \|h\|_{\cV^{\frac 5 6 + \delta}_{t,z}} \leq 1} \Big|\int_{\mathbb{T} \times \mathbb{S}^3} \underline{v}^5 h d t d \sigma \Big| \leq C_\delta  \|\underline{v}\|^{5}_{\cV^{\frac 5 6 + \delta}_{t,z}}\,.
$$
Then, to increase further the regularity of $\underline{v}$, we observe that the Strichartz estimate \eqref{strich.2delta} implies
$$
\| \underline{v}\|_{\cV^{2+\delta'}_{t,z}} = \| \Pi_{V} \underline{v}^5\|_{\cV^{\delta'}_{t,z}} = \sup_{h \in \cV^{\infty}_{t,z}\,,\ \|h\|_{\cV^{- \delta'}_{t,z}} \leq 1} \Big|\int_{\mathbb{T} \times \mathbb{S}^3} \underline{v}^5 h d t d \sigma \Big|  \leq C_\delta \|\underline{v}\|_{\cV^{1+\delta'}_{t,z}}^5\,.
$$
Iterating this procedure with increasing values of $\delta'$, one deduces that 
$\underline{v} $ is in $ C^\infty$.

In order to adapt the previous arguments to deal with the whole functional $\Phi$
in \eqref{big.Pharma}, we
split the bifurcation equation into low and high frequencies.
For any  $N \in \N$ (to be determined later large enough) 
 the bifurcation equation is equivalent to the system 
\begin{gather}
\label{low.p}
\varepsilon (-\Delta_{\mathbb{S}^3} + \mathds{1}) v_1 = \Pi_{V_{\leq N}} ((v_1 + v_2 +w)^p)\,,\\
\label{high.p}
\varepsilon (-\Delta_{\mathbb{S}^3} + \mathds{1}) v_2 = \Pi_{V_{> N}} ((v_1 + v_2 +w)^p)\,,
\end{gather}
where 
$$
v(t,z) = \sum_{j \in \N} v_j \cos(\omega_j t) e_j(z)\,, \quad v = v_1 + v_2\, , \quad v_1 := \Pi_{V_{\leq N}} v \, ,  \quad v_2 := \Pi_{> N}v\,, 
$$ 
$\omega_j := j +1$ are the frequencies associated to the eigenfunctions $e_j$ defined in \eqref{one.mode}, and  $\Pi_{\leq N}$, resp. $\Pi_{> N}$,  is the projector on the  time-space  Fourier frequencies smaller than $ N $, resp. $ > N $.  

 Then we solve both  the high frequency 
 bifurcation equation \eqref{high.p} and the range equation \eqref{Range.eq} 
   arguing by contraction:
 \begin{itemize}
 \item In Section \ref{sec.v2}, we solve first the high frequency bifurcation equation \eqref{high.p} for $v_2$ in a small ball of $\cV^{2+}_{t,z}$, for any  $\|v_1\|_{\cV^1_{t,z}} \leq R \varepsilon^{\frac 1 4}$ and $\|w\|_{H_t^{\frac 1 2 +} H^{\frac 3 2 + }_z} \lesssim \varepsilon^{\frac 5 4} N^{5 +}$. Here we use the Strichartz-type estimates \eqref{item1st1}-\eqref{strich.2delta}. 
  \item In Section \ref{sec.w}, we solve 
  the range equation \eqref{Range.eq} for $w$ in a small ball of $H^{\frac 1 2 +}_t H^{\frac 3 2 + }_z $. We exploit algebra properties
   since $v_2 \in \cV^{2+}_{t,z} \subseteq H^{\frac 1 2 +}_t H^{\frac 3 2 + }_z $
  and $v_1$ belongs to a finite dimensional space.
  \item In Section \ref{sec.bif} we solve the finite dimensional bifurcation equation \eqref{low.p}, which has a variational structure, applying  mountain pass arguments. 
  Finally in Section \ref{sec.tante} we prove multiplicity of critical points, distinguished by their minimal period, and in Section \ref{sec.classical} we prove their higher regularity.
\end{itemize} 
 In all these points, we use smallness conditions of the form $N^a \varepsilon^b \ll 1$, for $a, b >0$.
 
 In the case $p=3$ we follow an analogous variational procedure. Remark that in this case the ``resonant system" functional \eqref{big.Pharma0} possesses one mode Hopf plane wave solutions for any value of the momenta $(\mu_1, \mu_2)\in \Z \times \Z$, but in general their non-degeneracy is not known, except for the particular values considered in \cite{cha.smu}. This is because for $\mu_1 \neq \mu_2$ an explicit formula for the product between the eigenfunctions $\{e^{(\mu_1, \mu_2)}_j\}_j$ in \eqref{one.mode.hopf} is not available.
 Then we split our equation \eqref{original.eq} into the range equation \eqref{Range.eq} and the high and low bifurcation equations \eqref{high.p}, \eqref{low.p}. 
 We solve the low frequency bifurcation equation \eqref{low.p} using duality arguments, H\"older inequality and the Sobolev embedding \eqref{sob.sigh}, without Strichartz-type estimates.

In the degenerate case $p=2$, one has $\Pi_V(v^2)=0$ and the leading nonlinear term in the bifurcation equations \eqref{low.p}-\eqref{high.p} turns out to be the cubic term $\Pi_V \left(v \mathcal{L}_\omega^{-1} v^2\right)$. The Strichartz-type estimates \eqref{in.treno}-\eqref{v.che.vuoi} are used to solve the high frequency equation \eqref{high.p}, avoiding to prove if $\left.\mathcal{L}_\omega^{-1} \right|_{W}$ is bounded on $L^q(\mathbb{T}_t, L^q(\mathbb{S}^3, d \sigma))$ spaces.
\\[1mm]
\emph{Notation:} We denote by $\N := \{0, 1, 2, 3, \dots\}$ the set of integer numbers and  $\N_* := \{1, 2, 3, \dots\}$. Given $a \in \R$, we  denote $\langle a \rangle := \max\{1, |a|\}$. Given $a, b$ real valued functions, $a \lesssim b$ means that there exists $C>0$ such that $a \leq C b$. If $C$ depends on parameters $\alpha_1, \dots, \alpha_r$, we write $a \lesssim_{\alpha_1, \dots, \alpha_r} b$. If $a \lesssim b$ and $b \lesssim a$, we write $a \asymp b$.

\smallskip

\noindent{\bf Acknowledgments.}
Research supported by PRIN 2020 (2020XB3EFL001) 
``Hamiltonian and dispersive PDEs''.

\section{Functional Setting}

We describe the Laplace-Beltrami operator in spherical and Hopf coordinates,
we describe its spectrum and eigenfunctions
and we define Sobolev space of  spherically symmetric functions and  
and Hopf plane waves. 


\subsection{Functions with spherical symmetry}


According to Definition \ref{def:spherical}, in spherical coordinates the metric tensor is represented with respect to the basis of the tangent space $\big\{ \frac{\partial}{\partial x}, \frac{\partial}{\partial \theta},  \frac{\partial}{\partial \varphi} \big\} $ as
$$
g(x,\theta,\varphi)=\begin{pmatrix}
1 & 0 & 0\\
0 & \sin^2 (x) &0 \\
0 & 0 & \sin^2 (x) \sin^2 (\theta)
\end{pmatrix}\,.
$$
Hence the volume form is 
$d\sigma=\sin^2(x)\sin(\theta)dxd\theta d\varphi$, 
and the Laplace-Beltrami operator reads
\begin{equation}\label{LBsphe}
\Delta_{\mathbb{S}^3}=\partial_x^2+2\frac{\cos (x)}{\sin (x) }\partial_x+\frac{1}{\sin^2 (x)}\partial_{\theta}^{2}+\frac{\cos (\theta )}{\sin^2 (x) \sin (\theta)}\partial_\theta+\frac{1}{\sin^2(x)\sin^2 (\theta)}\partial_{\varphi}^2\,.
\end{equation}
For convenience, 
we introduce the  normalized measures
$$
	\dbar\sigma:=\frac{1}{2\pi^2}d\sigma\,, \quad \dbar x:= \frac{2}{\pi}dx\,, \quad \dbar\theta:=\frac{1}{2}d\theta\,, \quad \dbar\varphi:=\frac{1}{2\pi}d\varphi\,,
$$
chosen in such a way that the measure of the sphere $ \mathbb{S}^3 $ is $ 1$.
We denote $L^p(\mathbb{S}^3, \dbar \sigma) := L^p(\mathbb{S}^3)$.

The Laplace-Beltrami operator \eqref{LBsphe} leaves invariant the subspace of spherically symmetric functions (cfr. Definition \ref{def:spherical}), acting as
$$
\Delta_{\mathbb{S}^3} (u \otimes 1_{\theta, \varphi}) = \left(\Delta_{\mathbb{S}^{3}}^{ss} u\right) \otimes 1_{\theta, \varphi}\,, \quad \Delta_{\mathbb{S}^{3}}^{ss}:=\partial_x^2+2\frac{\cos (x)}{\sin (x)}\partial_x\,.
$$
As a consequence, the subspaces of spherically symmetric functions in $H^s(\mathbb{S}^3, \dbar \sigma)$ coincide with
\begin{equation}\label{Hsx.storti}
\mathcal{H}^s_x:=\left\lbrace u \in L^{2}([0,\pi],\sin^2(x)\dbar x)\, :\, (-\Delta_{\mathbb{S}^{3}}^{ss}+\mathds{1})^{\frac{s}{2}}u \in L^{2}([0,\pi],\sin^2(x)\dbar x)\right\rbrace\,,
\end{equation}
equipped with inner product  $\langle u_1,u_2 \rangle_{\mathcal{H}^s_x}:=\langle (-\Delta_{\mathbb{S}^{3}}^{ss}+\mathds{1})^{s}u_1,u_2\rangle_{L^2([0,\pi],\sin^2(x)\dbar x)}$, in the sense that
$u  \in \cH^s_x$ if and only if $u \otimes 1_{\theta, \varphi} \in H^s(\mathbb{S}^3, \dbar \sigma)$, with 
\begin{equation}\label{Hs.cHs}
\|u\|_{\cH^s_x} = \| u \otimes 1_{\theta, \varphi} \|_{H^s(\mathbb{S}^3, \dbar \sigma)} \quad \forall u \in \cH^s_x\,.
\end{equation}
We now exhibit a  basis of eigenfunctions and eigenvalues for  the operator $\Delta^{ss}_{\mathbb{S}^{3}}$, see \cite{cha.smu}:
\begin{lemma}[Spectral decomposition of $\Delta^{ss}_{\mathbb{S}^3}$]\label{teo.eigencouples}
	The set of functions $\{e_n\}_{n \in \N} $ defined by
	\begin{equation}\label{def.ej}
	e_n(x) := \frac{\sin((n+1)x)}{\sin(x)} \quad \forall n \in \N \, , 
	\end{equation} 
	is an orthonormal basis for $\mathcal{H}_x^{0}$ of 
	eigenfunctions of  $-\Delta^{ss}_{\mathbb{S}^{3}} + \Id$, with eigenvalues
	\begin{equation}\label{omega.n}
	\omega_n^{2}\,,
	\quad \omega_n := n+1 \, , \quad \forall n \in \N \,.
	\end{equation}
\end{lemma}

As a consequence, the Sobolev spaces $\cH^s_x$ in \eqref{Hsx.storti} are spectrally characterized as
	\begin{equation}\label{Hs.x.Fourier}
	\mathcal{H}_x^{s}=\Big\lbrace u(x)=\sum_{j\in \N}{u}_je_j(x)\, :\, \| u\|^2_{\cH^s_x} := \sum_{j\in \N} |{u}_j|^2\omega_j^{2s} <\infty \Big\rbrace 
	\,,
	\end{equation}
	where ${u}_j := \langle u, e_j \rangle_{\cH^0_x}$ are the Fourier coefficients with respect to the basis $\{e_j\}$, with scalar product $\langle u, v\rangle_{\mathcal{H}^s_x}=\sum_{j\in \N} {u}_j \overline{{v}_j} \omega_j^{2s}$.
	
The eigenfunctions $\{e_n\}_{n \in \N}$ 
satisfy the following product rule: for any integer $ n\geq m $,
	\begin{equation}\label{productrule}
	e_n(x)e_m(x)=\sum\limits_{k=0}^{m}e_{n-m+2k}(x)\,.
	\end{equation}
We will use property \eqref{productrule} to prove the Strichartz type Propositions \ref{lem.Strichartz.1} and \ref{lem.strich.Lom}. It can also be used to prove that the spaces $\cH^s_x$ with $s > \frac 3 2$ enjoy algebra property.

\subsection{Hopf symmetry}

 According to Definition \ref{def:hopf.ansatz}, in Hopf coordinates the metric tensor is represented with respect to the basis of the tangent space
  $\big\{ \frac{\partial}{\partial \eta}, \frac{\partial}{\partial \xi_1}, \frac{\partial}{\partial \xi_2}\big\}$ as 
$$
g(\eta,\xi_1,\xi_2)=\begin{pmatrix}
1 & 0 & 0 \\
0 & \sin^2 (\eta) & 0\\
0 & 0 & \cos^2 (\eta)
\end{pmatrix} \, . 
$$
Hence the volume form is $d\sigma=\frac{1}{2}\sin(2\eta)d\eta d\xi_1 d\xi_2$ and 
the Laplace-Beltrami operator reads
\begin{equation}\label{laplace}
\Delta_{\mathbb{S}^3}=\partial_{\eta}^2+2\frac{\cos(2\eta)}{\sin(2\eta)}\partial_{\eta}+\frac{1}{\sin^2 (\eta)}\partial_{\xi_1}^2+\frac{1}{\cos^2 (\eta)}\partial_{\xi_2}^2\,.
\end{equation}
We introduce the normalized  measure  
\begin{equation}\label{norm.meas}
\dbar \sigma := \frac{1}{2\pi^2} d \sigma\,, \quad \quad \dbar \eta := d \eta\,, \quad \dbar \xi_1 := \frac{1}{2\pi} d \csi_1\,, \quad \dbar \xi_2 := \frac{1}{2\pi} d \csi_2\, , 
\end{equation}
so  that the measure of the sphere $ \mathbb{S}^3 $ is $ 1$.

Representing a function $\phi$ in Hopf coordinates (see Definition \ref{def:hopf.ansatz}) and expanding in Fourier series with respect to the variables $\xi_1,\, \xi_2$, we have
\begin{equation}\label{envelop1}
\phi(\eta,\xi_1,\xi_2)=\sum\limits_{\mu_1,\mu_2\in \mathbb{Z}}
\hat{\phi}_{\mu_1,\mu_2}(\eta) e^{\ii\mu_1\xi_1}e^{\ii\mu_2\xi_2} \,.
\end{equation}
In these coordinates, the Laplace-Beltrami operator  \eqref{laplace} reads
\begin{equation}\label{laplace.hopf}
\Delta_{\mathbb{S}^3}\phi(\eta,\xi_1,\xi_2)=\sum\limits_{\mu_1,\mu_2\in \mathbb{Z}}e^{\ii\mu_1\xi_1}e^{\ii\mu_2\xi_2} \Delta_{\mu_1, \mu_2} \hat{\phi}_{\mu_1, \mu_2}(\eta)
\end{equation}
with
\begin{equation}\label{Lmu.op}
\Delta_{\mu_1,\mu_2} :=\partial_\eta^2+2\frac{\cos(2\eta)}{\sin(2\eta)}\partial_{\eta}-\frac{\mu_1^2}{\sin^2(\eta)}-\frac{\mu_2^2}{\cos^2(\eta)}\,.
\end{equation}
As a consequence, the space of Hopf plane waves is left invariant by $\Delta_{\mathbb{S}^3}$ and recalling \eqref{laplace}, \eqref{norm.meas}, 
 the subspaces of Hopf plane waves in $H^s(\mathbb{S}^3, \dbar \sigma)$ 
  coincide with
\begin{equation}\label{SobolHopf}
	\cH_{\eta}^{s}:=\left\lbrace u(\eta) \in L^2\left( ( 0,\tfrac{\pi}{2}),\sin(2\eta)d\eta \right)\, :\, (-\Delta_{\mu_1,\mu_2}  + \Id)^{\frac{s}{2}}u \in L^2\left(( 0,\tfrac{\pi}{2}),\sin(2\eta)d\eta \right)  \right\rbrace 
\end{equation}
for any $ s \in \R $,
equipped with inner product $\langle u_1, u_2 \rangle_{\cH^s_{\eta}}=\langle  \left(-\Delta_{\mu_1,\mu_2} + \Id \right)^s u_1,u_2\rangle_{L^2\left(( 0,\tfrac{\pi}{2}), \sin(2\eta)d\eta\right)}$, in the sense that 
 $u(\eta) \in \cH^s_{\eta} $ if and only if  
 $ u(\eta)e^{\ii\mu_1\xi_1}e^{\ii\mu_2\xi_2} \in H^{s}(\mathbb{S}^3,\dbar\sigma)$ 
 (cfr.  
\eqref{Hs.std}), with
	\begin{equation}\label{eta.sfera}
	\|u(\cdot)\|_{\cH^s_{\eta}}=\|u(\cdot)e^{\ii\mu_1\xi_1}e^{\ii\mu_2\xi_2}\|_{H^s(\mathbb{S}^3,\dbar \sigma)}\,.
	\end{equation} 
We now exhibit a basis of eigenfunctions and eigenvalues of $-\Delta_{\mu_1, \mu_2}$, see  \cite[Section 3.2]{cha.smu}. 

\begin{lemma}[Spectral decomposition of $\Delta_{\mu_1, \mu_2}$] \label{lemma.Lmu.diag}
	There exists an $L^2$-orthonormal basis of eigenfunctions $\{e^{(\mu_1, \mu_2)}_j\}_{j \in \N}$ of  $-\Delta_{\mu_1, \mu_2}  + \Id $, with eigenvalues $ ( \omega_j^{(\mu_1,\mu_2)})^2 $ where
	$$
	\omega_j^{(\mu_1,\mu_2)}:=2j+1+|\mu_1|+|\mu_2|\,, \quad j \in \N\,.
	$$
	The $e_j^{(\mu_1,\mu_2)}$ are the real functions 
	\begin{equation}\label{ejmu}
	e_j^{(\mu_1,\mu_2)}(\eta)=N_j^{(|\mu_1|,|\mu_2|)}\left(1-\cos(2\eta)\right)^{\frac{|\mu_1|}{2}}\left(1+\cos(2\eta)\right)^{\frac{|\mu_2|}{2}} P_j^{(|\mu_1|,|\mu_2|)}(\cos(2\eta))
	\end{equation}
	where $\{P_j^{(|\mu_1|,|\mu_2|)}\}_{j \in \N}$ are the Jacobi polynomials
	 and $N_j^{(|\mu_1|,|\mu_2|)}$ are suitable normalization constants.
\end{lemma}

By Lemma \ref{lemma.Lmu.diag}, the Sobolev spaces $\cH^s_{\eta}$ in \eqref{SobolHopf} are spectrally characterized as 
\begin{equation}\label{Hs.eta}
	\cH^s_\eta := \Big\lbrace u(\eta) := \sum_{j \in \N} {u}_j e_j^{(\mu_1, \mu_2)}(\eta)\ :\ \|u\|_{\cH^s_\eta}^2 := \sum_{j\in \N} |{u}_j|^2 \big(\omega_j^{(\mu_1, \mu_2)}\big)^{2s} < \infty\Big \rbrace 
\end{equation}
where ${u}_j := \langle u, e^{(\mu_1, \mu_2)}_j \rangle_{\cH^0_\eta}$ are the Fourier coefficients with respect to the basis $\{e^{(\mu_1, \mu_2)}_j\}_{j \in \N}$ equipped with scalar product
$ \langle u, v \rangle_{\cH^s_\eta} = \sum_{j\in \N} {u}_j \overline{{v}_j} (\omega_j^{(\mu_1, \mu_2)})^{2s} $. 

\subsection{Sobolev spaces in time-space}

Since equation \eqref{original.eq} is time reversible, 
we  look for functions which are even in time. For this reason, we consider the Sobolev spaces of time periodic even real functions
\begin{equation}\label{Hps.storti}
H^r_t \cH^s_z := \left\lbrace u(t, z) = \sum_{\ell, j \in \N} u_{\ell, j} \cos(\ell t) \egen_j(z)\ :\ \|u\|^2_{{H}^r_t\cH^s_z}:= \sum_{\ell \in \N} \langle \ell \rangle^{2 r} \sum_{j \in \N} \omega_j^{2s} |u_{\ell, j}|^2 < \infty\right\rbrace\,,
\end{equation}
taking values in
\begin{equation}\label{Hs.zeta}
\cH^s_z := \begin{cases}
\cH^s_x \text{ defined in \eqref{Hs.x.Fourier}} & \text{for spherically symmetric functions} \, ,  \\ \cH^s_\eta \text{ defined in \eqref{Hs.eta}} & \text{for Hopf plane waves}\, . 
\end{cases}
\end{equation}
In \eqref{Hps.storti} the 
 $\{\egen_j\}_{j \in \N}$ are respectively the eigenfunctions of $-\Delta_{\mathbb{S}^3}^{ss} - \Id$ and $-\Delta_{\mu_1, \mu_2} + \Id $, namely
\begin{equation}
\egen_j := \begin{cases}
e_j \text{ as in } \eqref{def.ej} &  \text{ for spherically symmetric functions} \, ,  \\
e_j^{(\mu_1, \mu_2)} \text{ as in Lemma \ref{lemma.Lmu.diag}} &\ \text{for Hopf plane waves}\,,
\end{cases}
\end{equation}
corresponding to the  eigenvalues $\{\omega_j^2 \}_{j \in \N}$ where 
\begin{equation}\label{omgen}
\omgen_j := \begin{cases}
j+1 & \ \text{for spherically symmetric functions} \, , \\ 
2 j + 1 + |\mu_1| + |\mu_2| &\ \text{for Hopf plane waves}\,,
\end{cases}
\end{equation}
and $u_{\ell, j}$ are the time-space 
Fourier coefficients of $ u $,  
\begin{equation}\label{fou.coef}
	u_{\ell, j} :=
	\begin{cases}
	\int_{\mathbb{T}} \int_0^\pi u(t, x) \cos(\ell t) e_j(x)\sin^2(x) d x \dbar t & \text{for spherically symmetric functions} \, , \\
	\int_{\mathbb{T}} \int_0^{\frac{\pi}{2}} u(t, \eta) \cos(\ell t) e^{(\mu_1, \mu_2)}_j(\eta)\sin (2\eta) d \eta \dbar t & \text{for Hopf plane waves}\,,
	\end{cases}
\end{equation}
where 
\begin{equation}\label{def.dbart}
\dbar t:=\frac{1}{\pi}dt\,.
\end{equation}
For any $r >\frac 1 2$ and for any $s \in \R$ the space $H^{r}_t \cH^s_z$ is embedded continuously into 
$L^\infty(\mathbb{T}_t, \cH^s_z)$, with
\begin{equation}\label{cor.12delta}
	\| u\|_{L^\infty (\mathbb{T}_t, \cH^s_z)} \leq C_{r} \|u\|_{H^{r}_t \cH^s_z}\,
\end{equation}
for some $C_r >0$. Moreover, since the spaces $\cH^s_z$ for $s > \frac 3 2 $ are an algebra, 
also the spaces $ H^{r}_t \cH^s_z  $ are an algebra 
for any $r > \frac 1 2 $ and $s > \frac 3 2$:  there exists a constant $C_{s,r}>0$ such that 
\begin{equation}\label{algebra.prop}
\|u_1\, u_2\|_{{H}^{r}_t \mathcal{H}^{s}_z}\leq C_{s,r} \|u_1\|_{{H}^{r}_t \mathcal{H}^{s}_z}\|u_2\|_{{H}^{r}_t \mathcal{H}^{s}_z}\,.
\end{equation}

\begin{lemma}\label{lem.lp.hs.easy}
There exists $C>0$ such that for any $ u^{(1)}, u^{(2)}, u^{(3)}, u^{(4)} \in L^{\infty}(\mathbb{T}_t,\cH^{\frac{3}{4}}_\eta)$
		\begin{equation}\label{int4estimate.0}
		\Big| \int_{\mathbb{T}}\int_{0}^{\frac{\pi}{2}} u^{(1)}u^{(2)}u^{(3)}u^{(4)} \sin(2\eta) d \eta \dbar t \Big| \leq C \prod_{l=1}^4 
		\|u^{(l)}\|_{L^{\infty}(\mathbb{T}_t, \cH^{\frac 3 4}_\eta)} \, . 
		\end{equation}
For any $ u^{(1)}, u^{(2)}, u^{(3)} \in L^{\infty}(\mathbb{T}_t ,\cH^1_\eta)$, and any $ u^{(4)} \in L^{\infty}(\mathbb{T}_t, \cH^0_\eta)$,
		\begin{equation}\label{int4estimate}
		\Big| \int_{\mathbb{T}}\int_{0}^{\frac \pi 2} u^{(1)} u^{(2)} u^{(3)} u^{(4)} \sin(2\eta) d \eta \dbar t \Big| \leq C \Big(\prod_{l=1}^3 \|u^{(l)}\|_{L^{\infty}(\mathbb{T}_t, \cH_\eta^1)} \Big) \|u^{(4)}\|_{L^{\infty}(\mathbb{T}_t, \cH^0_\eta)} \,.
		\end{equation}
\end{lemma}

\begin{proof}
	Defining $U^{(1)} := u^{(1)} e^{\ii \mu_1 \csi_1 + \ii \mu_2 \csi_2}$, $U^{(2)} := u^{(2)} e^{-\ii \mu_1 \csi_1 - \ii \mu_2 \csi_2}$, $U^{(3)} := u^{(3)} e^{\ii \mu_1 \csi_1 + \ii \mu_2 \csi_2}$, $U^{(4)} := u^{(4)} e^{-\ii \mu_1 \csi_1 - \ii \mu_2 \csi_2}$, 
	and recalling the definition of $\dbar \sigma$, $\dbar \csi_1$, $\dbar \csi_2$  in \eqref{norm.meas}, one has
	$$
	\begin{aligned}
	\int_{\mathbb{T}} \int_{0}^{\frac \pi 2} \prod_{l=1}^4 u^{(l)}(t, \eta) \sin(2\eta) d \eta \dbar t 
= \int_{\mathbb{T}} \int_{\mathbb{S}^3} \prod_{l=1}^4 U^{(l)}(t, z) \dbar \sigma(z) \dbar t\,.
	\end{aligned}
	$$ 
	Then 
	applying at any time $t$ 
	 the generalized H\"older inequality 
	 with $p_1=p_2=p_3=p_4 = 4$ for functions on $\mathbb{S}^3$, and the Sobolev embedding $H^{\frac{3}{4}}(\mathbb{S}^3, \dbar \sigma) \hookrightarrow L^4(\mathbb{S}^3, \dbar \sigma)$ we get 
	$$
	\begin{aligned}
	\Bigg| 	\int_{\mathbb{T}} \int_{0}^{\frac \pi 2} \prod_{l=1}^4 u^{(l)}(t, \eta) \sin(2\eta) d \eta \dbar t\Bigg| & \leq \int_{\mathbb{T}} \prod_{l=1}^4 \left\| U^{(l)}(t, \cdot) \right\|_{L^4(\mathbb{S}^3, \dbar \sigma)} \dbar t 
	\lesssim \int_{\mathbb{T}} \prod_{l=1}^4\Big\| U^{(l)}(t, \cdot) \Big\|_{H^{\frac 3 4 }(\mathbb{S}^3, \dbar \sigma)} \dbar t\,.
	\end{aligned}
	$$
	Then \eqref{int4estimate.0} follows because $\|U^{(l)}(t, \cdot)\|_{H^{\frac 3 4}(\mathbb{S}^3)} = \|u^{(l)}(t)\|_{\cH^{\frac 3 4}_{\eta}}$ by \eqref{eta.sfera}.
The bound	
\eqref{int4estimate} follows similarly applying the generalized H\"older inequality with $p_1=p_2=p_3=6$, and $p_4=2$ for functions on $\mathbb{S}^3$, using the embedding $H^{1}(\mathbb{S}^3, \dbar \sigma) \hookrightarrow L^6(\mathbb{S}^3, \dbar \sigma)$ and \eqref{eta.sfera}.
\end{proof}

\section{Variational Lyapunov-Schmidt decomposition}

We look for time periodic solutions of \eqref{original.eq} with time frequency $\omega$ close to $1$, via a Lyapunov-Schmidt decomposition. 
More specifically we look for a $\frac{2\pi}{\omega}$-time periodic real valued spherically symmetric solution $u(t,x)$ of \eqref{original.eq} which solves
\begin{equation}\label{main.sph.sym}
\begin{gathered}
-\omega^2 \partial_{tt} u(t, x) + (\Delta^{ss}_{\mathbb{S}^3} - \Id) u(t, x) = u^p(t, x)\,, \\
(t,x) \in \mathbb{T} \times (0, \pi)\,, \quad \partial_x u(t, 0) = \partial_x u(t, 2\pi) = 0 \,.
\end{gathered}
\end{equation}
We consider the cases $p=2,5$ only, because the case $p=3$ is covered in \cite{cha.smu}.

If $p=3$ we look for a $\frac{2\pi}{\omega}$-time periodic Hopf plane wave solution $\phi(t,\eta,\xi_1,\xi_2)=u(t,\eta)e^{i\mu_1\xi_1}e^{i\mu_2\xi_2}$ of \eqref{original.eq}, with $u(t,\eta)$ real. The function $u(t,\eta)$ solves
\begin{equation}\label{main.hopf.1}
\begin{gathered}
-\omega^2 \partial_{tt} u(t, \eta) + (\Delta_{\mu_1, \mu_2} -\Id) u(t, \eta) = u^3(t, \eta)\,,\\
(t, \eta) \in \mathbb{T} \times (0, \frac\pi 2)\,, \quad \partial_{\eta} u(t, 0) = \partial_{\eta} u(t, \frac \pi 2) = 0\,,
\end{gathered}
\end{equation}
with $\Delta_{\mu_1, \mu_2}$ defined in \eqref{Lmu.op}.
Both the equations in \eqref{main.sph.sym} and \eqref{main.hopf.1} are of the form
\begin{equation}\label{main.common}
\mathcal{L}_\omega u = u^p\,, \quad \mathcal{L}_\omega := -\omega^2\partial_{tt} - A\,,
\end{equation}
where $A$ denotes the unbounded, self-adjoint, positive operator
\begin{equation}\label{def.A}
A := \begin{cases}
-\Delta^{ss}_{\mathbb{S}^3} + \Id & \text{for spherically symmetric functions}\\
-\Delta_{\mu_1, \mu_2} + \Id & \text{for Hopf waves}\,.
\end{cases}
\end{equation}
Equation \eqref{main.common} admits a variational formulation. It is the formal Euler Lagrange equation of the action functional 
\begin{equation}\label{az.funct}
\Psi(u) := \frac{1}{2} \int_{\mathbb{T}} \langle \mathcal{L}_{\omega} u(t) , u(t) \rangle_{\cH^0_z} \dbar t - \mathcal{G}_{p+1}(u)\,,
\end{equation}
	with
\begin{equation}\label{big.G}
\mathcal{G}_{p+1}(u) :=
\begin{cases}
\begin{aligned}
\frac{1}{p+1} \int_{\mathbb{T}} \int_0^\pi u^{p+1}(t, x) \sin^2(x) \dbar x \dbar t & \quad \text{for spherically symmetric functions} \\
\frac{1}{p+1} \int_{\mathbb{T}} \int_0^{\frac \pi 2} u^{p+1}(t, \eta) \sin(2\eta) d\eta \dbar t & \quad \text{for Hopf plane waves}\,.
\end{aligned}
\end{cases}
\end{equation}
We shall exploit the variational structure of \eqref{main.common} in Section \ref{sec.bif}, after a suitable finite dimensional reduction. We perform a Lyapunov-Schmidt  decomposition of equation \eqref{main.common}. 
We define
\begin{gather}
	\label{def.V}
	\begin{aligned}
	V :=	\ker (-\partial_{tt} - A) &=\Big\lbrace u(t,z)=\sum\limits_{j,\ell\in \N} {u}_{\ell, j}\cos (\ell t)\egen_j(z)\ :\ {u}_{\ell, j}=0,\, \forall \ell\neq \omgen_j \Big\rbrace\\
	&= \Big\lbrace v(t,z)=\sum\limits_{j \in \N} {v}_{j}\cos (\omega_j t)\egen_j(z) \Big\rbrace\,,
	\end{aligned}
	\\
	\label{def.W}
		W:=
		\textrm{Rg}(-\partial_{tt} -A)=\Big\lbrace u(t,z)=\sum\limits_{j,\ell\in \N}{u}_{\ell, j}\cos (\ell t)\egen_j(z) \ : \,{u}_{\ell, j}=0,\, \forall \ell= \omgen_j \Big\rbrace\,.
\end{gather}
Note that $W=V^{\perp}$ in any ${H}^{r}_t\cH^s_z$.

We decompose the space $V$ into low and high frequencies: given $N \in \N$,
 we define
	\begin{gather}
	\label{def.V1}
	V_1 :=V_{\leq N}:=\Big\lbrace v(t,z)=\sum\limits_{0\leq \omega_j\leq N} {v}_j\cos(\omgen_j t) \egen_j(z)\,\Big\rbrace\,,\\
	\label{def.V2}
	V_2:=V_{>N}:=\Big\lbrace v(t,z)=\sum\limits_{\omega_j>N}{v}_j\cos (\omgen_j t)\egen_j(z) \Big\rbrace\,.
	\end{gather}
	We denote by $\Pi_V$, $\Pi_W$, $\Pi_{V_{\leq N}} \equiv\Pi_{V_1}$, $\Pi_{V>N} \equiv \Pi_{V_2}$, the orthogonal projectors on $V$, $W$, $V_{1}$ and $V_{2}$ respectively, so that any $u$ can be decomposed as
	\begin{equation}\label{def.Pi}
	\begin{gathered}
	u = v + w\,, \quad v := \Pi_V u = \sum_{j\in \N} u_{\omega_j, j} \cos(\omega_j t) \egen_j \in V\,, \quad w := \Pi_W u \in W\,,\\
	u = v_1 + v_2 + w\,, \quad v_1 := \Pi_{V_1} u \in V_1\,, \quad v_2:= \Pi_{V_2} u \in V_2\,, \quad w := \Pi_W u \in W\,.
	\end{gathered}
	\end{equation}
	We then observe that a function $u$ satisfies \eqref{main.common} if and only if it is a solution of the system
\begin{gather}
\label{v1.eq}
(\omega^2-1) A v_1 - 	\Pi_{V_{1}} (v_1 + v_2 +w)^p = 0 \,,\\
\label{v2.eq}
(\omega^2-1) A v_2 - \Pi_{V_{2}} (v_1 + v_2 +w)^p = 0 \,,\\
\label{w.eq}
\mathcal{L}_{\omega} w -  \Pi_W (v_1 + v_2 +w)^p = 0\,.
\end{gather}
We shall solve the equation \eqref{v2.eq} for $v_2$ by a contraction argument in Section \ref{sec.v2}. Then in Section \ref{sec.w} we shall solve the range equation \eqref{w.eq}, arguing again by a contraction argument and using the following lemma. 

\begin{lemma}\label{lemma.sono.pochissimi}
	Assume $\omega \in \Omega_\gamma$ with $\Omega_\gamma$ defined  in \eqref{Omega.gamma}. Then the linear operator $\mathcal{L}_\omega$ defined  in \eqref{main.common} is invertible on $W$, with
	\begin{equation}
	\left\| \mathcal{L}_\omega^{-1} \right\|_{\mathcal{B}(W \cap H^{r}_t \cH^s_{z}; W \cap H^{r}_{t} \cH^s_{z})} \leq \frac 2 \gamma \quad  \forall r, s \in \R\,.
	\end{equation}
	Furthermore, if $\omega = 1$, one has $\|\mathcal{L}_1^{-1}\|_{\mathcal{B}(W \cap H^r_t \cH^s_z; W \cap H^r_t \cH^s_z)} \leq 1$.
\end{lemma}

\begin{proof}
Let $w(t,z) = \sum_{\ell, j\,, \ell \neq \omgen_j } w_{\ell, j} \cos(\ell t) \egen_j(z)$. Then
\begin{equation}\label{L.om.fourier}
	\mathcal{L}_\omega^{-1} w(t, z) = \sum_{\ell \neq \omgen_j} \frac{w_{\ell, j}}{\omega^2 \ell^2 - \omgen_j^2} \cos(\ell t) \egen_j(z) \,.
\end{equation}
Then it is sufficient to observe that, if $\omega \in \Omega_\gamma$, then
\begin{equation}\label{approx.bded}
\forall \ell \neq \omgen_j \quad 
|\omega^2 \ell^2 - \omgen_j^2| \geq \frac{\gamma}{2}\,,
\end{equation}
because for any $\ell$ and $j$ such that $\ell \neq \omega_j$ and $\ell \neq 0$, one has
	${| (\omega \ell + \omega_j)(\omega \ell - \omega_j) | \geq |\omega \ell | \frac{\gamma}{|\ell|} \geq \frac{\gamma}{2}}$.
Finally, if $\omega = 1$, the estimate immediately follows observing that $|\ell^2 - \omega_j^2| \geq 1$, for any  $  \ell \neq \omega_j$.
\end{proof}

\section{Properties of functions in $V$}\label{sec.Stri}

In this section we prove some properties of functions in the kernel $V$ which will be used to solve the system \eqref{v1.eq}-\eqref{w.eq}.
Given $s \in \R$, we shall denote $\mathcal{V}^s_{t,z} := V \cap H^0_t \cH^s_z$, equipped with norm
\begin{equation}\label{hs.ker}
	\| v\|^2_{\cV^{s}_{t,z}} := \sum_{j \in \N} |v_j|^2 \omgen_j^{2s} \,.
\end{equation}
Furthermore we denote $\cV^{\infty}_{t,z} := \bigcap_{s \geq 0} \cV^{s}_{t,z}$ and 
 $\| \cdot \|_{L^p(\mathbb{T}_t, E)} := \left(\int_{\mathbb{T}} \| \cdot\|_E^p \dbar t\right)^{\frac 1 p}$.

\begin{lemma}\label{lemma.come.vuoi}
	Let $r,r',s,s' \in \mathbb{R}$ such that $r+s=r'+s'$.
	Then for any $  v \in V$
	\begin{align}\label{distrib.come.vuoi}
	& \|v\|_{{H}^r_t\cH^s_z}=\|v\|_{{H}^{r'}_t\mathcal{H}^{s'}_z}=\|v\|_{L^{2}(\mathbb{T}_t,\mathcal{H}^{r+s}_z)}\,, \\
\label{hs.linf}
&	\|v\|^2_{{H}^0_t\mathcal{H}^s_z}=\|v(0,\cdot)\|^2_{\mathcal{H}^s_z}=\|v\|^2_{L^{\infty} (\mathbb{T}_t,\mathcal{H}^s_z)} = \| v\|^2_{\cV^{s}_{t,z}}\,.
	\end{align}

	\begin{proof}
		In order to prove \eqref{distrib.come.vuoi} it is sufficient to observe that
		$$
		\|v\|^2_{{H}^r_t\mathcal{H}^s_z}=\big\|\sum_{j \in \N} {v}_{j}\cos(\omgen_jt)\egen_j(z)\big\|_{{H}^r_t\mathcal{H}^s_z}^2=\sum_{j \in \N} | {v}_{j}|^2\omgen_j^{2r}\omgen_j^{2s}=\sum_{j \in \N} |{v}_{j}|^2\omgen_j^{2(r+s)}\,.
		$$
		The identities 
		\eqref{hs.linf} follow because, for any $ t \in \mathbb{T} $, 
		$$
		\|v(t,\cdot)\|^2_{\mathcal{H}^s_z}=\big\|\sum_{j \in \N}{v}_{j}\cos (\omgen_jt)\egen_j(z)\big\|^2_{\mathcal{H}^s_z}=\sum_{j \in \N}|{v}_{j}|^2|\cos(\omgen_jt)|^2\omgen_j^{2s}\,,
		$$
		with $\| \cdot\|_{\cH^s_z}$ defined according to \eqref{Hs.zeta}
		and since  $|\cos(\cdot)|\leq 1$.
	\end{proof}
\end{lemma}

By \eqref{hs.linf} and algebra property of the spaces $\cH^s_z$, 
for any $  v^{(1)}, v^{(2)} \in \cV^{s}_{t,z}$ and $s > \frac 3 2$
\begin{equation}\label{algebra.cheap}
\|v^{(1)} v^{(2)}\|_{\cV^s_{t,z}} \lesssim_{s} \|v^{(1)}\|_{\cV^s_{t,z}} \|v^{(2)}\|_{\cV^s_{t,z}}\,.
\end{equation}
For any $ s < s' $ the following smoothing properties
hold (cfr.  \eqref{def.V1}, \eqref{def.V2}): for any $ v \in V $
\begin{equation}\label{lem.s.sprime}
\| \Pi_{V_1} v\|_{\cV^{s'}_{t,z}}\leq N^{s'-s}\|v\|_{\cV^{s}_{t,z}}\,, \quad \|\Pi_{V_2} v\|_{\cV^{s}_{t,z}}\leq N^{-	(s'-s)}\|v\|_{\cV^{s'}_{t,z}}\,.
\end{equation}
Since  $A \egen_j(z) = \omega_j^2 \egen_j(z)$ for any $j$ (see Lemmas \ref{teo.eigencouples} and \ref{lemma.Lmu.diag}) 
and recalling \eqref{hs.ker}  it results
\begin{equation}\label{Delta.smooth}
\|A^{-1} v\|_{\cV^{s-2}_{t,z}} \leq \|v\|_{\cV^{s}_{t,z}}\, , \quad \forall v \in \cV^{s}_{t,z}\, .
\end{equation}
We will also use that by Lemma \ref{lemma.come.vuoi} and the Sobolev embedding \eqref{sob.sigh}, for any $v \in \cV^{s}_{t,x}$
one has
\begin{equation}\label{sob.sigh.v}
	\|v\|_{L^p_{t,x}} \lesssim_p \|v\|_{\cH^s_{t,x}}\,, \quad p\leq \frac{6}{3-2s}\,,
\end{equation}
with $L^p_{t,x} := L^p(\mathbb{T}_t, L^p((0, \pi), \sin^2(x) \dbar x)$.
\begin{lemma}\label{pesce.rosso}
	For any  $u \in L^2(\mathbb{T}_t,\cH_z^{s})$ it results 
	$\left \| \Pi_{V} u \right\|_{\cV^{s}_{t,z}} \leq \|u\|_{H^0_t \cH^{s}_z}\,.$
	The same holds if $\Pi_V$ is replaced by $\Pi_{V_{1}}$ or $\Pi_{V_2}$.
\end{lemma}

\begin{proof}
	By Lemma \ref{lemma.come.vuoi} we have
	$\left \| \Pi_{V} u \right\|_{\cV^{s}_{t,z}}^2 = \left \| \Pi_{V} u \right\|_{H^0_t \cH_z^{s}}^2\leq \|u\|^2_{H^0_t \cH_z^{s}}$.
\end{proof}

\begin{lemma}\label{algebraKer}
	Let $s>\frac{3}{2}$, $r>\frac{1}{2}$ and $q \in \N$, then there exists a positive constant $C = C(s, r, q)$ such that for any $j< q$ and any $v^{(1)},\dots,v^{(j)} \in \cV^{s}_{t,z}$, $u^{(j+1)},\dots u^{(q)} \in H^r_t \cH^s_z$,
	\begin{equation}\label{eq.algebraker}
	\left\|\Pi_{V} \left({v^{(1)} \cdots v^{(j)} u^{(j+1)}\cdots u^{(q)}}\right)\right\|_{\cV^{s}_{t,z}}\leq C \|v^{(1)}\|_{\cV^s_{t,z}}\cdots \|v^{(j)}\|_{\cV^s_{t,z}}\|u^{(j+1)}\|_{H^r_t \cH^s_z}\cdots \|u^{(q)}\|_{H^r_t \cH^s_z}\,.
	\end{equation}
	The same statement holds if $\Pi_V$ is replaced by $\Pi_{V_1}$ or $\Pi_{V_2}$.
\end{lemma}

\begin{proof}
	Estimate \eqref{eq.algebraker} follows from Lemma \ref{pesce.rosso}, since 
	$\| u\|_{H^0_t \cH^s_z} \lesssim \| u\|_{L^\infty_t \cH^s_z}$, \eqref{algebra.prop} and   \eqref{cor.12delta}.
\end{proof}

 \begin{lemma}\label{lem.vsquare}
	For any $q \in \N$ odd and any $v^{(1)}, \dots, v^{(q)} \in V$ one has
	\begin{equation}\label{q.odd}
	\int_{\mathbb{T}} \int_{0}^\pi v^{(1)} (t,x) \cdots v^{(q)} (t,x)  \sin^2(x) \dbar x \dbar t = 0\,.
	\end{equation}
	In particular,  if $n $ is even then 
	$\Pi_V (v^{(1)} \cdots v^{(n)}) = 0$  for any $v^{(1)}, \dots, v^{(n)} \in V$.  
\end{lemma}

\begin{proof}
	Performing in the integral in \eqref{q.odd} the change of variables $(t,x) \mapsto (t', x')$, with $t := t' + \pi$, $x := \pi - x'$, one has 
	$
	v^{(l)}(t + \pi, \pi - x) = -v^{(l)}(t,x) $ and, 
	thus, since $q$ is odd,
	$$
	\begin{aligned}
	\mathcal{I} &:= \int_{\mathbb{T}} \int_0^\pi v^{(1)}(t,x)  \cdots v^{(q)}(t,x) \sin^2(x) \,\dbar x \dbar t\\
	&= \int_{\mathbb{T}} \int_0^\pi v^{(1)}(t'+ \pi, \pi -x')  \cdots v^{(q)}(t'+\pi, \pi-x') \sin^2(\pi-x') \,\dbar x' \dbar t'\\
	& = \int_{\mathbb{T}} \int_0^\pi (-1)^q v^{(1)}(t', x')  \cdots v^{(q)}(t', x') \sin^2(x') \,\dbar x' \dbar t' = -\mathcal{I}\,, 
	\end{aligned}
	$$
	namely $\mathcal{I} = 0$.
\end{proof}

\subsection{Strichartz-type estimates for $p=5$}

The aim of this section is to prove a set of Strichartz-type estimates for solutions of \eqref{lin.eq} in the case of spherical symmetry.
We shall use the following duality property: for any $s \in \R$,
\begin{equation}\label{duality}
\| v\|_{\cV^s_{t,x}} = \sup_{h \in \cV_{t,x}^\infty \atop \|h\|_{\cV^{-s}_{t,x}} \leq 1} \int_{\mathbb{T}} \int_{0}^\pi  v(t,x) h(t,x) \sin^2(x) \dbar x \dbar t\,.
\end{equation}
\begin{proposition}
{\bf (Generalized Strichartz-type estimates)}\label{lem.Strichartz.1}
	The following estimates hold:
\\[1mm]
1.  For any $\delta >0$ there exists a constant $C_\delta>0$ such that, for any $ v^{(1)}, \dots, v^{(6)} \in \cV^{\frac 5 6 + \delta}_{t,x}$,
		\begin{equation}\label{item1st1}
		\left| \int_{\mathbb{T}} \int_{0}^\pi v^{(1)}(t,x) \cdots v^{(6)}(t,x) \sin^2(x) \dbar x \dbar t\right| \leq C_\delta \prod_{n= 1}^6 \|v^{(n)}\|_{\cV^{\frac{5}{6} + \delta}_{t,x}}\,;
		\end{equation}
2. For any $\delta>0$ there exists a  constant $C_\delta > 0 $ such that, for any $ v^{(1)}, \dots, v^{(5)} \in \cV^{1 + \delta}_{t,x}$ and $v^{(6)} \in \cV^{1}_{t,x}$, 
		\begin{equation}\label{strich.2delta}
		\left| \int_{\mathbb{T}} \int_{0}^\pi {v^{(1)}(t,x) \cdots v^{(6)}(t,x)} \sin^2(x) \dbar x \dbar t\right| \leq C_\delta \Big(\prod_{n = 1}^5 \|v^{(n)}\|_{\cV^{1 +\delta}_{t,x}}\Big) \|v^{(6)}\|_{\cV^{-\delta}_{t,x}}\,.
		\end{equation}
\end{proposition}
\begin{remark}\label{rmk.mountain.pass}
	From \eqref{item1st1} the functional $\mathcal{G}_{6}(v) := \tfrac 1 6 \int_{\mathbb{T}} \int_{0}^\pi v^{6} \sin^2(x) \dbar x \dbar t$ is well defined on $\cV^1_{t,x}$
	with compact gradient.
\end{remark}

\begin{remark}
	By \eqref{item1st1} with $v^{(1)} = \cdots = v^{(6)} =: v$, using \eqref{hs.linf} and \eqref{Hs.cHs}, any solution $v$ of the Cauchy problem
$
	\partial_{tt} v + (-\Delta_{\mathbb{S}^3}  + \Id) v = 0 $, 
	 $ \partial_t v(0, \cdot) = 0 $, 
	 $  v(0, \cdot) = v_0 \in \cH^{\frac 5 6 + \delta}_{x} $, 
satisfies the Strichartz estimate
	$
		\| v\|_{L^6(\mathbb{T}_t\times\mathbb{S}^3, \dbar t\dbar \sigma)} \lesssim_{\delta}\|v\|_{\cV^{\frac 5 6 + \delta}_{t,x}}= \|v_0\|_{\cH^{\frac 5 6 + \delta}_x} $. 
\end{remark}

Proposition \ref{lem.Strichartz.1} enables us to deduce the following:
\begin{lemma}\label{lem-int.pv2}
	For any $\delta >0$ there exists $C_\delta>0$ such that for any $  v^{(1)},v^{(2)},v^{(3)},v^{(4)},v^{(5)} \in \cV^{1 + 2\delta}_{t,x}$
	\begin{equation}
	\left\|\Pi_{V_2}\left({v^{(1)}v^{(2)}v^{(3)}v^{(4)}v^{(5)}}\right)\right\|_{\cV^{2\delta}_{t,x}}\leq C_\delta \prod\limits_{n=1}^{5}\|v^{(n)}\|_{\cV^{1+2\delta}_{t,x}}\,.
	\end{equation}
\end{lemma}

\begin{proof}
	By \eqref{duality} we have:
	\begin{align*}
	&\left\|\Pi_{V_2}\big({v^{(1)}v^{(2)}v^{(3)}v^{(4)}v^{(5)}}\big) \right\|_{\cV^{2\delta}_{t,x}}=\sup\limits_{h \in V_2, \|h\|_{\cV^{-2\delta}_{t,x}}\leq 1}\int_\mathbb{T} \int_{0}^{\pi}\Pi_{V_2}\big({v^{(1)}v^{(2)}v^{(3)}v^{(4)}v^{(5)}}\big)h\sin^2(x) \dbar x \dbar t\\
	&\lesssim_\delta \sup\limits_{h \in V_2, \|h\|_{\cV^{-2\delta}_{t,x}}\leq 1} \prod\limits_{n=1}^{5}\|v^{(n)}\|_{\cV^{1+2\delta}_{t,x}}\|h\|_{\cV^{-2\delta}_{t,x}} \lesssim_\delta \prod\limits_{n=1}^{5}\|v^{(n)}\|_{\cV^{1+2\delta}_{t,x}}
	\end{align*}
	where in the second passage we have used Item 2 of Proposition \ref{lem.Strichartz.1}.
\end{proof}

The rest of this section is devoted to the proof of Proposition \ref{lem.Strichartz.1}.
We use the following definition. 
\begin{definition}\label{def.jmin}
	Given $j_1,j_2,j_3,j_4,j_5,j_6 \in \mathbb{N}$, we define $j_{\min_1}, \dots, j_{\min_6} \in \N$ by the property that $\{j_1, \dots, j_6\} = \{j_{\min_1}, \dots, j_{\min_6}\}$, and
	\begin{equation}
	\min\{j_1, \dots, j_6\}=: j_{\min} := j_{\min_1} \leq j_{\min_2} \leq j_{\min_3} \leq \cdots \leq j_{\min_6}:= \max\{j_1, \dots, j_6\} \,.
	\end{equation}
\end{definition}
Furthermore we  denote $\delta (a = b) := \delta_{a,b} $ for any
$ a, b \in \N $, the Kronecker delta.

The following lemma is a direct computation, recalling \eqref{def.dbart}:

\begin{lemma}[Integral in time]\label{lem.int.t}
	Given $\vec{\omega}:=(\omega_{j_1},\omega_{j_2},\omega_{j_3},\omega_{j_4},\omega_{j_5},\omega_{j_6}) \in \mathbb{N}^6_{*}$
	and $\vec{\sigma}\in \lbrace \pm 1\rbrace^{6}$,
	\begin{equation}\label{cosines}
	\int_{\mathbb{T}} \prod_{k=1}^6 \cos(\omega_{j_k}t)\dbar t= 2^{-5}\sum\limits_{\vec{\sigma} \in \{\pm 1\}^6 }\delta ( \vec{\sigma} \cdot \vec{\omega} = 0 )\,.
	\end{equation}
\end{lemma}

The next lemma exploits  properties of the eigenfunctions $\{e_{n}\}_{n \in \N}$ of $-\Delta_{\mathbb{S}^3}^{ss} + \Id$, defined in \eqref{def.ej}. 

\begin{lemma}[Integral in space]\label{lem.le.piu.piccole}
	For any $j_1, j_2, j_3, j_4, j_5, j_6 \in \N,$ the integral
	\begin{equation}\label{omeghine.int}
	{\cal I}_{j_1, \dots, j_6}:= \int_{0}^{\pi} e_{j_1}(x) \cdots e_{j_6}(x) \sin^2(x) \dbar x
	\end{equation}
	satisfies
	\begin{equation}\label{space_sharp_c}
	0\leq {\cal I}_{j_1, \dots, j_6} \leq \omega_{j_{\min}}\omega_{j_{\min_2}}\omega_{j_{\min_3}}\,.
	\end{equation}
\end{lemma}

\begin{proof}
	With no loss of generality, we suppose that $j_1 \leq j_2 \leq  \cdots \leq j_6$. 
	By the  product rule \eqref{productrule}, 
	\begin{equation*}
	e_{j_1} e_{j_4} = \sum_{k= 0}^{j_1} e_{j_4 - j_1  + 2 k}\,, \quad 	e_{j_3} e_{j_5} = \sum_{l = 0}^{j_3} e_{j_5 - j_3  + 2 l}\,,
	\end{equation*}
	and
	\begin{equation*}
	\begin{gathered}
	e_{j_1} e_{j_4} e_{j_2} = \sum_{k= 0}^{j_1} \sum_{h= 0}^{\min\{j_2, j_4 - j_1  + 2 k\}} e_{|j_2 - (j_4 - j_1  + 2 k)| + 2 h}\,, \ \
		e_{j_3} e_{j_5} e_{j_6} = \sum_{l= 0}^{j_3} \sum_{m= 0}^{\min\{j_6, j_5 - j_3  + 2 l\}} e_{|j_6 - (j_5 - j_3  + 2 l)| + 2 m}\,.
	\end{gathered}
	\end{equation*}
	Since by Lemma \ref{teo.eigencouples} the $\{e_j\}_{j \in \N}$ are orthonormal on $L^2([0,\pi], \sin^2(x)\dbar x)$, one has
	\begin{align}
	\label{prodotto}
	& {\cal I}_{j_1, \dots, j_6} =  \int_{0}^{\pi} e_{j_1}(x) \cdots e_{j_6}(x) \sin^2(x)\,\dbar x\\
	\nonumber
	&= 
	\sum_{k= 0}^{j_1} \sum_{h= 0}^{\min\{j_2, j_4 - j_1  + 2 k\}}  \sum_{l= 0}^{j_3} \sum_{m= 0}^{\min\{j_6, j_5 - j_3  + 2 l\}} \hspace{-9pt}\delta (|j_2 - (j_4 - j_1  + 2 k)| + 2 h = |j_6 - (j_5 - j_3  + 2 l)| + 2 m)\,.
	\end{align}
	Now for each fixed value of $j_1, \dots, j_6, k, h$ and $l$, there is at most one value of $m$ such that ${\delta (|j_2 - (j_4 - j_1  + 2 k)| + 2 h = |j_6 - (j_5 - j_3  + 2 l)| + 2 m) \neq 0.}$ Moreover the sum over $k$ runs over $j_1 + 1 = \omega_{j_1}$ elements, the sum over $h$ runs over $\leq j_2 + 1 = \omega_{j_2}$ elements, and the sum over $l$ runs over $\leq j_3 + 1 = \omega_{j_3}$ elements. 
	This proves that the integral 
	\eqref{prodotto} satisfies  ${\cal I}_{j_1, \dots, j_6} \leq \omega_{j_1} \omega_{j_2}\omega_{j_3}$.
	Also  the lower bound 
	${\cal I}_{j_1, \dots, j_6} \geq 0$ directly follows because \eqref{prodotto}
	is the sum of non-negative integers.
\end{proof}

We now prove  Proposition \ref{lem.Strichartz.1}. 

\begin{proof}[Proof of Proposition \ref{lem.Strichartz.1}, Item 1]
	We  show that
	for any $ v \in V $ and $  N \in \N$, the function   
	$ v_N := \Pi_{V_{\leq N}} v $, satisfies 
	\begin{equation}\label{frequenze.finite}
	\int_{\mathbb{T}} \int_{0}^\pi v_N^6 (t,x)\sin^2(x) 
	\dbar x \dbar t \lesssim_{ \delta} \|v_N\|^{6}_{\cV^{\frac 5 6 + \delta}_{t,x}}\,. 
	\end{equation}
	Then \eqref{item1st1}  follows since, using generalized H\"older inequality with $p_1 = \cdots = p_6 = \frac 1 6$, Fatou's Lemma
	and the fact that $v^{(n)}_N \to v^{(n)} $ in $ \cV^{\frac 5 6 + \delta}_{t,x} $ and, up to subsequence, $v^{(n)}_N(t,x) \to v^{(n)} (t,x) $ a.e, 
	\begin{equation}\label{intcomple}
	\begin{aligned}
	\left|\int_{\mathbb{T}} \int_{0}^\pi \prod_{n=1}^6 v^{(n)}(t,x) \sin^2(x)  \dbar x \dbar t\right| & \leq \prod_{n=1}^{6} \big\| v^{(n)}\big\|_{L^6_{t,x}} \leq \liminf_{N \rightarrow \infty }\prod_{n=1}^{6} \big\| v^{(n)}_N\big\|_{L^6_{t,x}}\\
	&\stackrel{\eqref{frequenze.finite}}{\lesssim_{ \delta}} \lim_{N \rightarrow \infty }\prod_{n=1}^{6} \|v^{(n)}_N\|_{\cV^{\frac 5 6 + \delta}_{t,x}} \lesssim_\delta \prod_{n=1}^{6} \|v^{(n)}\|_{\cV^{\frac 5 6 + \delta}_{t,x}}\,. 
	\end{aligned}
	\end{equation}
	We now prove \eqref{frequenze.finite}. 
	By Lemma \ref{lem.int.t} and recalling 
	\eqref{omeghine.int}, one has
	\begin{align}
	\mathcal{I}&:= \int_{\mathbb{T}} \int_{0}^{\pi} {v_N^6(t,x)} \sin^2(x) \dbar x \dbar t
	= \int_{\mathbb{T}} \int_{0}^{\pi} \sum_{j_1, \dots, j_6 \in \N \atop j_k  \leq N \ \forall k} \prod_{k=1}^6 v_{ j_k} \cos(\omega_{j_k}t) e_{j_k}(x) \sin^2(x) \,\dbar x \dbar t \notag 
	\\
	\label{symmetric}
	&= 2^{-5} \sum_{\sigma_1, \dots, \sigma_6 \in \{\pm 1\}}  \sum_{j_1, \dots, j_6 \in \N \atop j_k  \leq N \ \forall k} \delta\left(\sigma_1\omega_{j_1}  + \dots +  \sigma_6\omega_{j_6} = 0\right) v_{ j_1} \cdots v_{j_6} \mathcal{I}_{j_1, \dots, j_6}\\
	&\stackrel{\eqref{space_sharp_c}}{\lesssim} \sum_{\sigma_1, \dots, \sigma_6 \in \{\pm 1\}}  \sum_{j_1  \leq \cdots \leq j_6 \leq N \atop \sigma_1 \omega_{j_1}  + \dots +  \sigma_6\omega_{j_6} = 0}  |v_{ j_1} \cdots v_{j_6}| \omega_{j_1} \omega_{j_2} \omega_{j_3}\,, \label{stima.sei}
	\end{align}	
	by the symmetry of \eqref{symmetric} with respect to $j_1, \dots, j_6$.
	Since $\omega_{j_1} \leq \cdots \leq\omega_{j_6}$, one has 	$
	\omega_{j_{1}} \leq \omega_{j_{1}}^{\frac 5 6} \omega_{j_{2}}^{\frac 1 6} $, $ \omega_{j_{2}} \leq \omega_{j_{2}}^{\frac 4 6} \omega_{j_{3}}^{\frac 2 6}$
	and $ \omega_{j_{3}}\leq \omega_{j_{4}}^{\frac 1 3} \omega_{j_{5}}^{\frac 1 3} \omega_{j_{6}}^{\frac 1 3} $, thus
	\begin{equation}\label{omega.mega.10}
	\omega_{j_1} \omega_{j_2} \omega_{j_3} \leq \omega_{j_{1}}^{\frac 1 3} \omega_{j_{2}}^{\frac 1 3} \omega_{j_{3}}^{\frac 1 3} \omega_{j_{4}}^{\frac 1 3} \omega_{j_{5}}^{\frac 5 6} \omega_{j_6}^{\frac 5 6}\,.
	\end{equation}
	Moreover, recalling $\omega_j = j + 1$, the sum in \eqref{stima.sei} is restricted to 
	\begin{equation}\label{zucchero}
	{j_6} = \sigma_6^{-1}\left(\sigma_1 \omega_{j_1} + \dots + \sigma_{5} \omega_{j_5}\right) -1 =: f(\sigma_1, \dots, \sigma_6, j_1, \dots, j_5) =: f(\vec{\sigma}, \hat{\jmath})\,,
	\end{equation}
	where we set  $\vec{\sigma} := (\sigma_1, \dots, \sigma_6)$ and 
	$ \hat{\jmath} := (j_1, \dots, j_5)$. 
	By \eqref{stima.sei}, \eqref{omega.mega.10} and Cauchy-Schwarz inequality, one has
	\begin{align*}
	{\cal I}&\lesssim  \sum_{j_1 \leq N, \dots, j_4 \leq N \atop \sigma_1, \dots, \sigma_6 \in \{\pm 1\}} \left|v_{j_1} \cdots v_{j_4}\right|  \omega_{j_1}^{\frac{1}{3}}  \cdots \omega_{j_4}^{\frac{1}{3}} \sum_{ j_5 \leq N  \atop \textrm{s.t.} f(\vec{\sigma}, \hat \jmath) \leq N} |v_{j_5}| \omega_{j_5}^{\frac{5}{6}} |v_{ f(\vec{\sigma}, \hat \jmath)}| \omega_{f(\vec{\sigma}, \hat \jmath)}^{\frac{5}{6}}\\
	& \lesssim \sum_{j_1 \leq N, \dots, j_4 \leq N \atop 
	\sigma_1, \dots, \sigma_6 \in \{\pm 1\}}  |v_{j_1}|  \omega_{j_1}^{\frac{1}{3}} \cdots |v_{j_4}| \omega_{j_4}^{\frac{1}{3}}  \Big(\sum_{j_5 \leq N } |v_{j_5}|^2 \omega^{\frac{5}{6} 2 }_{j_5}\Big)^{\frac{1}{2}}  \Big(\sum_{j_5\,:\,f(\vec{\sigma}, \hat \jmath) \leq N} |v_{f(\vec{\sigma}, \hat \jmath)}|^2 \omega^{\frac{5}{6} 2 }_{f(\vec{\sigma}, \hat \jmath)}\Big)^{\frac{1}{2}}\\
	& \lesssim \|v_N\|_{\cV^{\frac{5}{6}}_{t,x}}^2 \Big(\sum_{j \leq N} |v_{j}| \omega^{\frac{1}{3}}_{j}\Big)^4 \lesssim_\delta \|v_N\|_{\cV^{\frac{5}{6} + \delta}_{t,x}}^6 \,,
	\end{align*}
	where in the last passage we have used Cauchy-Schwarz inequality to bound 
	$$
	\sum_{j \leq N} |v_{j}| \omega^{\frac{1}{3}}_{j} \leq \Big(\sum_{j \leq N} |v_{j}|^2 \omega^{2(\frac{5}{6} + \delta)}_{j}\Big)^{\frac 1 2} \Big(\sum_{j \leq N} \omega_j^{-(1 + 2 \delta)}\Big)^{\frac 1 2} \lesssim_\delta \| v_N\|_{\cV^{\frac 5 6 + \delta}_{t,x}}\,.
	$$
	This proves \eqref{frequenze.finite}.
\end{proof}

\begin{proof}[Proof of Proposition \ref{lem.Strichartz.1}- Item 2]
	We are going to show that,
	for any $n =1, \dots , 6$ and $N \in \N$, the functions 
	$ v^{(n)}_N := \Pi_{\leq N} v^{(n)} $ satisfy 
	\begin{equation}\label{buona.domenica}
	\left|\int_{\mathbb{T}}\int_{0}^\pi 
	v^{(1)}_{N} (t,x) \ldots v^{(6)}_{N}(t,x)
	\sin^2(x)\dbar x \dbar t \right|
	\lesssim_{ \delta} \prod_{n=1}^5 
	\| v_N^{(n)} \|_{\cV^{1 + \delta}_{t,x}} \| v_N^{(6)}\|_{\cV^{-\delta}_{t,x}}\,.
	\end{equation}
This implies \eqref{strich.2delta}. 
Indeed $ v^{(n)}_N \to v^{(n)} $ in $ \cV^{1}_{t,x} $ and, by \eqref{sob.sigh.v}, in $L^6_{t,x}$ and a.e. 
for any $n = 1, \ldots, 6 $. Then passing to the limit one obtains \eqref{strich.2delta}.

We now prove \eqref{buona.domenica}. By Lemma \ref{lem.int.t} and recalling \eqref{omeghine.int}, one has
	\begin{align}
	\nonumber
	\mathcal{I}_N & := \left|\int_{\mathbb{T}}\int_{0}^\pi 
	v^{(1)}_{N} (t,x) \ldots v^{(6)}_{N}(t,x)
	\sin^2(x)\dbar x \dbar t \right|
	\leq \sum_{\sigma_1, \dots, \sigma_6 \in \{\pm 1\}} 
	\!\!\!\! \sum_{j_1 \leq N, \dots, j_6\leq N \atop \sigma_1 \omega_{j_1} + \cdots + \sigma_j \omega_{j_6} = 0} \!\!\!\! |v^{(1)}_{j_1} \cdots v^{(6)}_{j_6}| \mathcal{I}_{j_1, \dots, j_6}\\
	\label{cannella}
	& \stackrel{\eqref{space_sharp_c}}{\lesssim} \hspace{-20pt} \sum_{\sigma_1,\dots, \sigma_6 \in \{\pm 1\}} \sum_{j_1 \leq \cdots \leq j_5 \leq N, \atop
	j_6 = f(\vec{\sigma}, \hat{\jmath}) \leq N} \hspace{-10pt}
	|v^{(1)}_{ j_1}| \cdots |v^{(5)}_{j_5}|
	|v^{(6)}_{j_6}| \omega_{j_1} \omega_{j_2} \omega_{j_6}\,,
	\end{align}   
	by the symmetry of \eqref{cannella} in $j_1, \dots, j_5$. Since
	$ \omega_{j_6} = \sigma_6^{-1} \big(\sigma_1 \omega_{j_1} + \dots + \sigma_5 \omega_{j_5}\big) \leq 5 \omega_{j_5} $ and
	$$
	\omega_{j_6} \leq 5^{1 + \delta} \omega_{j_6}^{-\delta} \omega_{j_5}^{1 + \delta}\,, \quad \omega_{j_{1}}\omega_{j_{2}} \leq 
	\omega_{j_1}^{\tfrac12} \omega_{j_2}^{\tfrac12}
	\omega_{j_3}^{\tfrac12} \omega_{j_4}^{\tfrac12} \, ,
	$$
	by \eqref{cannella} and using Cauchy-Schwarz, we have
	\begin{align*}
	\mathcal{I}_N &	\lesssim_\delta \sum_{\sigma_1,\dots, \sigma_6 \in \{\pm 1\}} \sum_{j_1 \leq \cdots \leq j_5 \leq N, \atop
		j_6 = f(\vec{\sigma}, \hat{\jmath}) \leq N} \hspace{-10pt}
	|v^{(1)}_{ j_1}| \cdots |v^{(5)}_{j_5}|
	|v^{(6)}_{j_6}| \omega^{\tfrac12}_{j_1} \cdots \omega^{\tfrac12}_{j_4} \omega^{1+\delta}_{j_5} \omega^{-\delta}_{f_{(\vec{\sigma}, \hat{\jmath})}}\\
	&\lesssim_\delta
	 \sum_{j_1 \leq N, \dots, j_4 \leq N
	 \atop \sigma_1, \dots, \sigma_6 \in\{\pm 1\}} 
	 |v^{(1)}_{j_1}|   \omega^{\tfrac12}_{j_1}
	 \cdots |v^{(4)}_{j_4}|\omega^{\tfrac12}_{j_4}   \sum_{j_5 \leq N \atop
	 s.t. f_{(\vec{\sigma}, \hat{\jmath})} \leq N } 
	 |v^{(5)}_{ j_5}| \omega^{1+\delta}_{j_5} 
	 |v^{(6)}_{ f(\vec{\sigma}, \hat \jmath)}| 
	  \omega^{-\delta}_{f(\vec{\sigma}, \hat{\jmath})}  \\
	&\lesssim_\delta 
	 \sum_{j_1 \leq N, \dots, j_4 \leq N} 
	 |v^{(1)}_{j_1}|   \omega^{1+\delta}_{j_1} \omega^{-(\tfrac12 + \delta)}_{j_1}
	 \cdots |v^{(4)}_{j_4}| \omega^{1+\delta}_{j_4} \omega^{-(\tfrac12 + \delta)}_{j_4} 
	 \| v_N^{(5)} \|_{\cV^{1+\delta}_{t,x}} \|v_N^{(6)} \|_{\cV^{-\delta}_{t,x}}\\
	& \lesssim_\delta
	 \|v_N^{(1)}\|_{\cV^{1+\delta}_{t,x}} \ldots \|v_N^{(4)}\|_{\cV^{1+\delta}_{t,x}}
	   \|v_N^{(5)}\|_{\cV^{1+\delta}_{t,x}} \|v_N^{(6)}\|_{\cV^{-\delta}_{t,x}} 
	\end{align*}
	proving \eqref{buona.domenica}.
\end{proof}

\subsection{Strichartz-type estimates for $p=2$}

The following result will play a central role in studying \eqref{original.eq} in the case $p=2$:
\begin{proposition} {\bf (Generalized Strichartz-type estimates)}\label{lem.strich.Lom}
	For any  $\delta>0$ there exists $C_\delta>0$ such that
\\[1mm]	
	\noindent$1.$ For any $v^{(1)}, \dots, v^{(4)} \in \cV^{\frac 1 2 + \delta}_{t,x}$, 
	\begin{equation}\label{in.treno}
	\left|\int_{\mathbb{T}}\int_0^\pi v^{(1)} v^{(2)} \mathcal{L}_{\omega}^{-1} (v^{(3)}v^{(4)}) \sin^2(x) \dbar x \dbar t\right| \leq {C_{\delta}}{\gamma}^{-1} \prod_{n=1}^4\|v^{(n)}\|_{\cV^{\frac 1 2 + \delta}_{t,x}}\,.
	\end{equation}
	$2.$ 
	For any $v^{(1)}, \dots, v^{(4)} \in V \cap \cV^{\frac 2 3 + \delta}_{t,x}$, 
	for any $ l = 1, \dots, 4 $, 
	\begin{equation}\label{v.che.vuoi}
	\left|\int_{\mathbb{T}} \int_{0}^\pi v^{(1)} v^{(2)} \mathcal{L}_\omega^{-1}(v^{(3)} v^{(4)}) \sin^2(x) \dbar x \dbar t\right| \leq C_{\delta} \gamma^{-1} \prod_{n=1 \atop n \neq l}^4 \|v^{(n)}\|_{\cV^{\frac 2 3 + \delta}_{t,x}} \|v^{(l)}\|_{\cV^{-\delta}_{t,x}} \, . 
	\end{equation}
Furthermore, if $\omega = 1$, 
estimates \eqref{in.treno} and \eqref{v.che.vuoi} 
hold with the factor $\gamma^{-1}$ at right-hand-side replaced by $1$.
\end{proposition}

	Note that, due to Lemma \ref{lem.vsquare}, one has $v^{(3)} v^{(4)} \in W$,  thus $\mathcal{L}_\omega^{-1}(v^{(3)}v^{(4)})$ is well defined.
The rest of this section is devoted to the proof of Proposition \ref{lem.strich.Lom}.

\begin{lemma}[Integral in space]\label{lem.int.space.4}
	For any $j_1, j_2, j_3, j_4\in \N,$ the integral 
	$$
	{\cal I}_{j_1, \dots, j_4}:= \int_{0}^{\pi} e_{j_1}(x) \cdots e_{j_4}(x) \sin^2(x) \dbar x
	$$
	satisfies $	0\leq {\cal I}_{j_1, \dots, j_4} \leq \omega_{j_{\min}} $.
\end{lemma}

\begin{proof}
	With no loss of generality, we suppose that $j_1 \leq j_2 \leq j_3 \leq j_4$. By the product rule \eqref{productrule} one has
	\begin{equation}\label{due.a.due}
	e_{j_1} e_{j_2} e_{j_3} e_{j_4} = \sum_{k = 0}^{j_1} \sum_{h=0}^{j_3} e_{j_2 - j_1 + 2 k} e_{j_4 - j_3 + 2h}\,,
	\end{equation} 
	thus since by Lemma \ref{teo.eigencouples} the $\{e_j\}_{j \in \N}$ are orthonormal on $L^2([0,\pi], \sin^2(x)\dbar x)$, one has
	\begin{equation}\label{prodotto4}
	\mathcal{I}_{j_1, \dots, j_4} =  \sum_{k = 0}^{j_1} \sum_{h=0}^{j_3} \delta(j_2 - j_1 + 2 k = j_4 - j_3 + 2h)\,.
	\end{equation}
	Now for each fixed $j_1,j_2,j_3,j_4,k$ there is at most one value of $h$ such that $j_2 - j_1 + 2 k = j_4 - j_3 = 2h$. Moreover the sum over $k$ runs over $j_1+1=\omega_1$ elements. This proves 
	$\mathcal{I}_{j_1j_2j_3j_4}\leq \omega_{j_1}$. The lower bound $\mathcal{I}_{j_1j_2j_3j_4}\geq 0$ directly follows because \eqref{prodotto4} is the sum of 
	non-negative integers.
\end{proof}

\begin{proof}[Proof of Proposition \ref{lem.strich.Lom}, Item 1]
	For any $v \in V $ and $N \in \N_*$  we 
	$v_N:= \Pi_{V_{\leq N}}v$. We are going to show that 
	\begin{equation}\label{L4.H12}
	\int_{\mathbb{T}} \int_{0}^{\pi} v_N^4 (t,x) \sin^2(x) \dbar x \dbar t \lesssim_\delta \|v_N\|_{\cV^{\frac 1 2 + \delta}_{t,x}}^4\,.
	\end{equation}
	Then estimate \eqref{in.treno} follows.
	Indeed, by \eqref{L4.H12}, one has, arguing as for \eqref{intcomple}, 
	\begin{equation}\label{stri.l4}
	\|v\|_{L^4_{t,x}} 
	\lesssim_{ \delta} \|v\|_{\cV^{\frac 1 2 + \delta}_{t,x}}\,.
	\end{equation}
	Furthermore, for any $\omega \in \Omega_\gamma$, using 
	Cauchy-Schwarz inequality and Lemma \ref{lemma.sono.pochissimi}, we get 
	\begin{align}\label{tutto.bene}
	\left|\int_{\mathbb{T}} \int_{0}^\pi v^{(1)}v^{(2)} \mathcal{L}_\omega^{-1}(v^{(3)} v^{(4)}) \sin^2(x) \dbar x \dbar t \right| &\leq \big\|v^{(1)} v^{(2)} \big\|_{L^2_{t,x}} \big\| \mathcal{L}_\omega^{-1} \big(v^{(3)} v^{(4)} \big) \big\|_{L^2_{t,x}} \notag \\
	&\leq 2 \gamma^{-1} \prod_{n=1}^{4} \|v^{(n)}\|_{L^4_{t,x}} \stackrel{\eqref{stri.l4}} 
	\leq C_\delta \gamma^{-1} \prod_{n=1}^{4} \|v^{(n)}\|_{\cV^{\frac 1 2 + \delta}_{t,x}}\,, \notag
 	\end{align}
proving \eqref{in.treno}.
	If $\omega = 1$ the thesis follows  since, 
	by Lemma \ref{lemma.sono.pochissimi}, 
	$\|\mathcal{L}_1^{-1} \|_{\mathcal{B}(\cH^0_{t,x}, \cH^{0}_{t,x})} \leq 1$.
	
	We now prove \eqref{L4.H12}. Arguing as for \eqref{stima.sei}, we have
	\begin{equation}\label{stima.4}
	\begin{aligned}
	\mathcal{I}_N &:= \int_{\mathbb{T}} \int_{0}^{\pi} v_N^4(t,x) \sin^2(x) \dbar x \dbar t
	\lesssim \sum_{\sigma_1, \dots, \sigma_4 \in \{\pm 1\}} \sum_{j_1 \leq \cdots \leq j_4 \leq N \atop \sigma_1 \omega_{j_1} + \cdots + \sigma_4 \omega_{j_4} = 0} |v_{j_1} \cdots v_{j_4}| \mathcal{I}_{j_1, \dots, j_4}\,.
	\end{aligned}
	\end{equation}
	By Lemma \ref{lem.int.space.4} and estimating $\mathcal{I}_{j_1, \dots, j_4} \leq \omega_{j_1} \leq \omega_{j_3}^{\frac 12} \omega_{j_4}^{\frac 1 2}$, and defining 
	\begin{equation}\label{j4.4}
	{j_4} = \sigma_4^{-1}\left(\sigma_1 \omega_{j_1} + \dots + \sigma_{3} \omega_{j_3}\right) -1 =: g(\sigma_1, \dots, \sigma_4, j_1, \dots, j_3) =: g(\vec{\sigma}, \hat{\jmath})\,,
	\end{equation}
	we have, using Cauchy-Schwarz inequality and \eqref{stima.4},
	\begin{align*}
	\mathcal{I}_N & \lesssim \sum\limits_{j_1, j_2\leq N \atop \sigma_1, \dots, \sigma_4 \in \{\pm 1\}} |v_{j_1}||v_{j_2}| \sum_{j_3 \leq N \atop g(\vec{\sigma}, \hat{\jmath}) \leq N} |v_{j_3}||v_{g(\vec{\sigma}, \hat{\jmath})}|\omega_{j_3}^{\frac{1}{2}}\omega_{g(\vec{\sigma}, \hat{\jmath})}^{\frac{1}{2}}\\
	&\lesssim \sum\limits_{j_1 \leq N}|v_{j_1}| \omega_{j_1}^{\frac 1 2 + \delta} \omega_{j_1}^{-(\frac 1 2 + \delta)}  \sum\limits_{j_2 \leq N}|v_{j_2}| \omega_{j_2}^{\frac 1 2 + \delta}  \omega_{j_2}^{-(\frac 1 2 + \delta)} \Big(\sum\limits_{j_3 \leq N}|v_{j_3}|^2 \omega_{j_3}\Big)^{\frac 1 2} \Big(\sum_{j_3\leq N \atop \text{s.t.} g(\vec{\sigma}, \hat{\jmath}) \leq N} |v_{g(\vec{\sigma}, \hat{\jmath})}|^2 \omega_{g(\vec{\sigma}, \hat{\jmath})}\Big)^{\frac 1 2}\\
	&\lesssim_{\delta} \|v_N\|^2_{\cV^{\frac{1}{2}+\delta}_{t,x}}\|v_N\|^2_{\cV^{\frac{1}{2}}_{t,x}}\,,
	\end{align*}
	which gives \eqref{L4.H12}.
\end{proof}

\begin{proof}[Proof of Proposiition \ref{lem.strich.Lom}, Item 2]
	First we prove that, defining  
	$v^{(k)}_N := \Pi_{\leq N} v^{(k)}$, $N \in \N $, 
	one has
	\begin{equation}\label{ancora.N}
	\!\! \left|\int_{\mathbb{T}} \int_{0}^\pi v^{(1)}_N v^{(2)}_N \mathcal{L}_\omega^{-1}\big(v^{(3)}_N v^{(4)}_N\big) \sin^2(x) \dbar x \dbar t\right| \lesssim_{\delta} 
	\frac{1}{\gamma} \|v^{(1)}\|_{\cV^{\frac 2 3 + \delta}_{t,x}} \|v^{(2)}\|_{\cV^{\frac 2 3 + \delta}_{t,x}} \|v^{(3)}\|_{\cV^{\frac 2 3 + \delta}_{t,x}} \|v^{(4)}\|_{\cV^{-\delta}_{t,x}}\,,
	\end{equation}
	with the factor $\gamma^{-1}$ in \eqref{ancora.N} replaced by $1$ if $\omega = 1$. 
	Once that \eqref{ancora.N} has been proved, Item 2 follows by the following claim: for any $  v^{(1)},v^{(2)},v^{(3)},v^{(4)}\in \cV^{\frac{1}{2}+\delta}_{t,x}$ it results 
	\begin{equation}\label{stri.stri}
	v^{(1)}_{N}v^{(2)}_{ N} \stackrel{L^2_{t,x}} \longrightarrow v^{(1)}v^{(2)} \, ,\quad
	\mathcal{L}^{-1}_{\omega}(v^{(3)}_{N}v^{(4)}_{N})
	\stackrel{L^2_{t,x}} \longrightarrow  \mathcal{L}^{-1}_{\omega} (v^{(3)}v^{(4)}) \, , \quad 
	\text{as} \ \  N \rightarrow \infty \, , 
	\end{equation}
	so that 
	$ v^{(1)}_{N}v^{(2)}_{N} \mathcal{L}^{-1}_{\omega}(v^{(3)}_{N}v^{(4)}_{N}) $ converges to 
	$ v^{(1)} v^{(2)} \mathcal{L}^{-1}_{\omega}(v^{(3)} v^{(4)} ) $  
	in $ L^1 $.
The claim 	\eqref{stri.stri} follows because 
	$$
	\begin{aligned}
	\| v^{(1)}_N v^{(2)}_N - v^{(1)} v^{(2)}\|_{L^2_{t,x}}
	&\leq \|v^{(1)}_N- v^{(1)}\|_{L^4_{t,x}} \|v^{(2)}_N\|_{L^4_{t,x}} + \| v^{(2)} - v_N^{(2)}\|_{L^4_{t,x}} \|v^{(1)}\|_{L^4_{t,x}}\\
	&\stackrel{\eqref{stri.l4}}{\lesssim_\delta} \|v^{(1)}_N- v^{(1)}\|_{\cV^{\frac 1 2 + \delta}_{t,x}} \|v^{(2)}_N\|_{\cV^{\frac 1 2 + \delta}_{t,x}} + \| v^{(2)} - v_N^{(2)}\|_{\cV^{\frac 1 2 + \delta}_{t,x}} \|v^{(1)}\|_{\cV^{\frac 1 2 + \delta}_{t,x}} \rightarrow 0
	\end{aligned}
	$$
	as $N \rightarrow \infty$.
	Similarly 
	$\mathcal{L}_\omega^{-1} v^{(3)}_N v^{(4)}_N \rightarrow \mathcal{L}_\omega^{-1}(v^{(3)} v^{(4)})$ in $L^2_{t,x}$ using also Lemma \ref{lemma.sono.pochissimi}. Moreover \eqref{v.che.vuoi} for a general index $l$ follows by self-adjointness of $\mathcal{L}_\omega^{-1}$.

	The first step in the proof of \eqref{ancora.N} is the following:
	\begin{lemma}
	\begin{equation}\label{bianconiglio}
	\begin{split}
	\left|\int_{\mathbb{T}} \int_{0}^\pi v^{(1)}_N v^{(2)}_N \mathcal{L}_\omega^{-1}\big(v^{(3)}_N v^{(4)}_N\big) \sin^2(x) \dbar x \dbar t\right| \lesssim
	\begin{cases}
	\gamma^{-1} \mathcal{J} & \text{if} \quad \omega \neq 1\,,\\
	\mathcal{J} & \text{if} \quad \omega = 1\,,
	\end{cases}\\
	\mathcal{J}:= \sum_{\sigma, \sigma', \sigma_1, \sigma_2 \in \{\pm 1\}} \sum_{j_1, j_2, j_3 \leq N} \big|v^{(1)}_{j_1} v^{(2)}_{j_2} v^{(3)}_{j_3} v^{(4)}_{h(\vec{\sigma}, \hat{\jmath})} \big| \omega_{\min\{j_1, j_2, j_3, h(\vec{\sigma}, \hat{\jmath})\}}\,,
	\end{split}
	\end{equation}
	with
	\begin{equation}\label{cappellaio.matto}
	h(\vec{\sigma}, \hat{\jmath}):=  -\sigma' \sigma_2  \sigma_1 \omega_{j_1} - \sigma' \sigma_2 \sigma_1 \sigma \omega_{j_2} - \sigma'  \omega_{j_3} -1\,, \quad \vec{\sigma}:= (\sigma, \sigma', \sigma_1, \sigma_2)\,, \quad \hat{\jmath}:= (j_1, j_2, j_3)\,.
	\end{equation} 
	\end{lemma}
	\begin{proof}
	By \eqref{productrule} and Lemma \ref{lemma.sono.pochissimi}, we compute
	\begin{equation}\label{tanti.A}
	\int_{\mathbb{T}}\int_0^\pi v^{(1)}_N v^{(2)}_N \mathcal{L}_{\omega}^{-1} ( v^{(3)}_N v^{(4)}_N ) = \sum_{\sigma, \sigma' \in \{\pm 1\}} \sum_{j_1\leq N, \dots, j_4\leq N} v^{(1)}_{j_1} v^{(2)}_{j_2} v^{(3)}_{j_3} v^{(4)}_{j_4} A^{(\sigma,\sigma ')}_{j_1, j_2, j_3, j_4}\,,
	\end{equation}
	where
	$	A^{(\sigma, \sigma')}_{j_1, j_2, j_3, j_4}$, 
	$\sigma, \sigma' \in \{-1, 1\}$, are equal to
	$$
	\sum_{h = 0}^{\min\{j_1, j_2\}} \sum_{k=0
	}^{\min\{j_3, j_4\}} \int_{\mathbb{T}} \frac{  \cos((\omega_{j_1} + \sigma \omega_{j_2})t) \cos((\omega_{j_3} + \sigma' \omega_{j_4})t)}{4(\omega^2(\omega_{j_3} + \sigma' \omega_{j_4})^2 -\omega_{|j_4 - j_3| + 2k}^2)} \dbar t \int_{0}^\pi  e_{|j_4-j_3| + 2k} e_{|j_2-j_1| + 2h} \sin^2(x) \dbar x\, . 
	$$
	Using $\omega \in \Omega_{\gamma}$ and \eqref{approx.bded}, we have the lower bound
	\begin{equation}\label{gamma.non.1}
	\left|\omega^2(\omega_{j_3} + \sigma' \omega_{j_4})^2 -(\omega_{|j_4 - j_3| + 2k})^2\right| \geq \frac{\gamma}{2}\,.
	\end{equation}
	Moreover
	$\int_{\mathbb{T}} \cos(\alpha_1 t) \cos(\alpha_2 t)\dbar t= \frac{1}{2}\sum\limits_{\sigma_1=\pm 1, \sigma_2=\pm 1}\delta\left(\sigma_1 \alpha_1 + \sigma_2 \alpha_2 \right)$,  and therefore
	 $A^{(\sigma, \sigma')}_{j_1, j_2, j_3, j_4} \neq 0$ only if
	$\sigma_1 (\omega_{j_1} + \sigma \omega_{j_2}) + \sigma_2 (\omega_{j_3} + \sigma' \omega_{j_4}) = 0$,
	which gives $j_4 = h(\vec{\sigma}, \hat{\jmath})$ with $ \vec{\sigma}, \hat{\jmath}$ and $h(\vec{\sigma}, \hat{\jmath})$ as in \eqref{cappellaio.matto}.
	Furthermore, by orthogonality of $\{e_j\}_j$ as in Lemma \ref{teo.eigencouples}, by \eqref{due.a.due} and by Lemma \ref{lem.int.space.4}, one has
	\begin{equation}\label{in.spazio}
	\sum_{h = 0}^{\min\{j_1, j_2\}} \sum_{k=0}^{\min\{j_3, j_4\}} \left| \int_{0}^\pi  e_{|j_4-j_3| + 2k} e_{|j_2-j_1| + 2h} \sin^2(x) \dbar x \right| = \mathcal{I}_{j_1, j_2, j_3, j_4}
	\leq \omega_{\min\{j_1, j_2, j_3, j_4\}}\,. 
	\end{equation}
	Note that, if $\omega = 1$, the factor $\gamma^{-1}$ in \eqref{gamma.non.1} can be replaced by $1$. Thus, combining \eqref{tanti.A}, \eqref{gamma.non.1}, \eqref{in.spazio}, one gets \eqref{bianconiglio}.
	\end{proof}
	The sum $\mathcal{J}$ in \eqref{bianconiglio}, using its symmetry in the indexes $j_1, j_2, j_3$, is bounded by
	\begin{equation}\label{cosi.e.se.vi.pare}
	\begin{aligned}
	\mathcal{J}& \lesssim \sum_{\vec{\sigma} \in \{\pm 1\}^4} \sum_{j_1 \leq j_2 \leq j_3 \leq N \atop h(\vec{\sigma}, \hat{\jmath}) \leq N} |v^{(1)}_{j_1} v^{(2)}_{j_2} v^{(3)}_{j_3}| |v^{(4)}_{h(\vec{\sigma}, \hat{\jmath})}| \omega_{j_{\min}}\,.
	\end{aligned}
	\end{equation} 
	Since $\omega_{h(\vec{\sigma}, \hat{\jmath})} \leq 4 \omega_{j_3}$, 
	\begin{equation}\label{giuoco.delle.parti}
	\omega_{j_{\min}} \leq \omega_{j_1}^{\frac 1 6} \omega_{j_2}^{\frac 1 6} \omega_{j_3}^{\frac 2 3} \lesssim_\delta  \omega_{j_1}^{\frac 1 6} \omega_{j_2}^{\frac 1 6} \omega_{j_3}^{\frac 2 3 + \delta}\omega_{h(\vec{\sigma}, \hat{\jmath})}^{-\delta}\,.
	\end{equation}
	Then using Cauchy-Schwarz inequality, \eqref{cosi.e.se.vi.pare} and \eqref{giuoco.delle.parti}, one has
	\begin{align*}
	\mathcal{J} &\lesssim_\delta \sum_{\vec{\sigma} \in \{\pm 1\}^4} \sum_{j_1 \leq N, j_2\leq N} |v^{(1)}_{j_1} v^{(2)}_{j_2}| \omega_{j_1}^{\frac 1 6} \omega_{j_2}^{\frac 1 6} \sum_{j_3 \leq N \atop \text{s.t.} h(\vec{\sigma}, \hat{\jmath}) \leq N} |v^{(3)}_{j_3}|  \omega_{j_3}^{\frac 2 3 + \delta}  |v^{(4)}_{h(\vec{\sigma}, \hat{\jmath}) } |\omega_{h(\vec{\sigma}, \hat{\jmath})}^{-\delta}\\
	&\lesssim_{\delta} \|v^{(1)}_N\|_{\cV^{\frac 2 3 + \delta}_{t,x}} \|v^{(2)}_N\|_{\cV^{\frac 2 3  + \delta}_{t,x}} \|v^{(3)}_N\|_{\cV^{\frac 2 3+\delta}_{t,x}} \|v^{(4)}_N\|_{\cV^{-\delta}_{t,x}}\,,
	\end{align*}
	from which \eqref{ancora.N} follows. 
\end{proof}

\section{Solution of the $v_2$ equation}\label{sec.v2}
In this section we solve the equation \eqref{v2.eq} for the high frequency components $v_2$ in the kernel. We argue separately for the cases $p=5, p=3$, and 
for the degenerate case $p=2$.
Given $\rho_1\in (0, 1)$, $\rho_2 \in (0, 1)$, $\rho_3 \in (0, 1)$, we define
\begin{equation}
\label{v1.Domain}
	\mathcal{D}_{\rho_1}:= \big\{ v_1 \in V_{1}\ :\ \|v_1\|_{\cV^1_{t,z}}\leq \rho_1 \big\}\,,
\end{equation}
and for some $\delta>0$
\begin{gather}
	\label{v2.Domain}
	{\mathcal{D}}^{V_2}_{\rho_2} :=\left\lbrace v_2 \in V_2 \cap \cV^{2 +2 \delta}_{t,z} \ : \|v_2\|_{\cV^{2 +2 \delta}_{t,z}} \leq \rho_2\right\rbrace\,,\\
	\label{w.Domain}
	{\mathcal{D}}^W_{\rho_3} :=\Big\{ w \in H^{\frac 1 2 + \delta}_t \cH^{\frac 3 2 + \delta}_z \cap W\ : \|w\|_{H^{\frac 1 2 + \delta}_t \cH^{\frac 3 2 + \delta}_z} \leq \rho_3\Big\}\,.
\end{gather}
In the sequel $\delta$ will always denote a positive small constant.
\subsection{Case $p=5$}
For any $R>0$, $\delta >0$ and $\gamma \in (0, \gamma_0)$ let
\begin{equation}\label{parametri.p5}
\rho_1:=\varepsilon^{\frac{1}{4}}R,\quad \rho_2:=  \tc_2 N^{10\delta}R^5 \varepsilon^{\frac{1}{4}}, \quad \rho_3:= \tc_3 \gamma^{-1} N^{5+10\delta}R^5\varepsilon^{\frac{5}{4}}\,, \quad N := \varepsilon^{-\frac{1}{\beta}}\,,
\end{equation}
where $\varepsilon = \omega^2 - 1 >0$  according to  \eqref{omega.ep}, $\tc_2, \tc_3>0$ and $\beta >1$. 
\begin{proposition}[Solution of $v_2$ equation for $p=5$]\label{prop.v2}
	For any $\delta \in (0, \frac{1}{8})$, $R>0$ and $\gamma \in (0, \gamma_0)$, let $\rho_1, \rho_2, \rho_3, N$ be as in \eqref{parametri.p3}. There exist $\tc_2 :=\tc_2(\delta) >0$, $\beta_\delta>1$, $\zeta:=\zeta(\delta) >0$, $\epsilon_{\delta, R}>0$ and $C_{1,\delta}, C_{2,\delta}>0$ such that, for any $\beta>\beta_\delta$ and any $\varepsilon>0$ such that
	\begin{equation}\label{little.gamma}
	\varepsilon N^{\zeta} \gamma^{-1} \leq \epsilon_{\delta, R}\,,
	\end{equation}
	 there exists a $C^1$ function ${v}_2: \mathcal{D}_{\rho_1} \times {\mathcal{D}}^W_{\rho_3} \rightarrow V_2 \cap \cV^{2+2\delta}_{t,x}$, $(v_1, w) \mapsto {v}_2(v_1, w),$ where $\mathcal{D}_{\rho_1}$ and $\mathcal{D}^W_{\rho_3}$ are defined as in \eqref{v1.Domain} and \eqref{w.Domain}, satisfying $v_2(0, 0) = 0$, and
	\begin{gather}
	\label{in.che.bolla}
	\|v_2(v_1, w)\|_{\cV^{2+2\delta}_{t,x}} \leq \rho_2\,,\\
	\label{de.v2.de.v1}
	\|\partial_{v_1}v_2(v_1,w)\|_{\mathcal{B}(V_1 \cap \cV^1_{t,x}, \cV^{2+2\delta}_{t,x})}\leq C_{1,\delta} N^{10\delta} R^4\,,\\
	\label{de.v2.de.w}
	\|\partial_{w}v_2(v_1,w)\|_{\mathcal{B}(W \cap H^{\frac{1}{2}+\delta}_t \cH^{\frac{3}{2}+\delta}_x,\cV^{2+2\delta}_{t,x})}\leq C_{2,\delta} N^{\frac{1}{2}+5\delta} R^4\,,
	\end{gather}
	 such that
	$v_2(v_1, w)$ solves
	\begin{equation}\label{v2.di.v1.e.w}
	\varepsilon (-\Delta_{\mathbb{S}^3}^{ss} + \Id) v_2(v_1, w) -  \Pi_{V_2} \left({(v_1 + v_2(v_1, w) +w)^5}\right) = 0 \,.
	\end{equation}
\end{proposition}
In the rest of this section we prove Proposition \ref{prop.v2}. For any $(v_1, w) \in \mathcal{D}_{\rho_1} \times {\mathcal{D}}^W_{\rho_3}$, we look for a solution of \eqref{v2.di.v1.e.w} as a fixed point of the map
\begin{equation}\label{tv2}
v_2 \mapsto \mathcal{T}_{v_1,w}(v_2):=\varepsilon^{-1} A^{-1}\Pi_{V_2}{(v_1+v_2+w)^{5}}\,,
\end{equation}
where, according to \eqref{def.A}, we  set $A := -\Delta_{\mathbb{S}^3}^{ss} + \Id$. The next lemma is based on the Strichartz estimates of Section \ref{sec.Stri}.

\begin{lemma}[Contraction]\label{lem.t2.bolla}
	Let $\zeta \geq \frac{11}{2} + 5 \delta$. There exist $C_\delta>0$ and $\epsilon_{\delta, R}>0$ 
	such that, if \eqref{little.gamma} holds, the map $\mathcal{T}_{v_1, w}$ maps  $\mathcal{D}^{V_2}_{\rho_2}$ into itself, with
	\begin{equation}\label{pizza}
	\left\|\mathcal{T}_{v_1,w}(v_2) -\mathcal{T}_{v_1,w}(v'_2) \right\|_{\cV^{2+2\delta}_{t,x}}\leq C_\delta R^4 N^{-1+8\delta} \|v_2 - v_2'\|_{\cV^{2+2\delta}_{t,x}} \quad \forall v_2, v'_2 \in \mathcal{D}^{V_2}_{\rho_2}\,.
	\end{equation}
	As a consequence, for any $ (v_1, w) \in \mathcal{D}_{\rho_1} \times \mathcal{D}^W_{\rho_3}$ there exists a unique $v_2(v_1, w) \in \mathcal{D}_{\rho_2}^{V_2}$ solving \eqref{v2.di.v1.e.w} and such that $v_2(0, 0) =0$.
	\end{lemma}
	\begin{proof}
	We write $\mathcal{T}_{v_1,w}(v_2)=\varepsilon^{-1} A^{-1}\Pi_{V_2} \sum\limits_{j_1+j_2+j_3=5}c_{j_1 j_2 j_3} v_1^{j_1}v_2^{j_2}w^{j_3}$. 
	We estimate the terms where $w$ does not appear using the Strichartz-type estimate in Lemma \ref{lem-int.pv2}, and the terms with $w$  using the 
	algebra property \eqref{algebra.prop}.
	If $j_3 = 0$, by  \eqref{Delta.smooth},  \eqref{lem.s.sprime}, and Lemma \ref{lem-int.pv2},
	for any $ j_1, j_2$ one has
	\begin{equation}\label{cor_ker}
	\left\|A^{-1}\Pi_{V_2}\big({v_1^{j_1}v_2^{j_2}}\big)\right\|_{\cV^{2+2\delta}_{t,x}}\lesssim_\delta (N^{2\delta}\|v_1\|_{\cV^1_{t,x}})^{j_1}(N^{-1}\|v_2\|_{\cV^{2+2\delta}_{t,x}})^{j_2} \lesssim_{\delta} (N^{2\delta} \rho_1)^{j_1} (N^{-1} \rho_2)^{j_2}\,
	\end{equation}
	for any $\|v_1\|_{\cV^1_{t,x}} \leq \rho_1$ and $\| v_2\|_{\cV^{2+2\delta}_{t,x}} \leq \rho_2$. Recalling the definitions of $\rho_1, \rho_2, N$ in \eqref{parametri.p5}, one then gets $N^{-1} \rho_2 \leq N^{2\delta} \rho_1$ and thus for any $j_1 + j_2 = 5$
	\begin{equation}\label{luglio}
		\big\| A^{-1} \Pi_{V_2} \big({v_1^{j_1}v_2^{j_2}}\big)\big\|_{\cV^{2+2\delta}_{t,x}}
		\lesssim_\delta N^{10\delta} \rho_1^5
		\,.
	\end{equation}
	On the other hand, if $j_3 \neq 0$, 
	by estimates \eqref{Delta.smooth}, \eqref{lem.s.sprime} and Lemma \ref{algebraKer} one has
	\begin{align}
	\nonumber
	\big\| A^{-1}\Pi_{V_2}\big( {v_1^{j_1}v_2^{j_2}w^{j_3}}\big) \big\|_{\cV^{2+2\delta}_{t,x}}&\lesssim_\delta N^{-\frac{3}{2}+\delta}(N^{\frac{1}{2}+\delta }\|v_1\|_{\cV^1_{t,x}})^{j_1} (N^{-\frac{1}{2}-\delta}\|v_2\|_{\cV ^{2+2\delta}_{t,x}})^{j_2}\|w\|_{H^{\frac{1}{2}+\delta}_t \cH^{\frac{3}{2}+\delta}_x}^{j_3}\\
	\label{lem.f.pv2}
	&\lesssim_\delta N^{-\frac{3}{2}+\delta}( N^{\frac{1}{2}+\delta}  \rho_1)^{j_1}(\rho_2N^{-\frac{1}{2}-\delta})^{j_2} \rho_3^{j_3}\,
	\end{align}
	for any $\| v_1\|_{\cV^1_{t,x}} \leq \rho_1$, $\|v_2\|_{\cV^{2+2\delta}_{t,x}} \leq \rho_2$, and ${\|w\|_{H^{\frac 1 2 + \delta}_t \cH^{\frac 3 2 + \delta}_x} \leq \rho_3}$. Assuming \eqref{little.gamma} with $\zeta =\frac{11}{2}+5\delta$ and $\epsilon_{\delta, R}$ small enough, one has $\rho_3 \leq N^{\frac 1 2 + \delta} \rho_1$ and $N^{-\frac 1 2 - \delta}\rho_2 \leq N^{\frac 1 2 +\delta} \rho_1$, and recalling $j_3 \geq 1$, one gets
	\begin{equation}\label{o.non.luglio}
	\big\| A^{-1}\Pi_{V_2}\big( {v_1^{j_1}v_2^{j_2}w^{j_3}}\big) \big\|_{\cV^{2+2\delta}_{t,x}}
	\lesssim_{\delta}N^{\frac{1}{2}+5\delta}\rho_1^{4}\rho_3 \lesssim_{ \delta} N^{10\delta} \rho_1^5\,.
	\end{equation}
	Thus, combining \eqref{luglio} and \eqref{o.non.luglio}, there exists a constant $C_\delta>0$ such that we have
		$$
		\left\|\mathcal{T}_{v_1,w}(v_2)\right\|_{\cV^{2+2\delta}_{t,x}} \leq C_\delta  \varepsilon^{-1}  N^{10\delta} \rho_1^{5}
		\stackrel{\eqref{parametri.p5}}{=}C_\delta N^{10\delta }R^{5}\varepsilon^{\frac{1}{4}}< \tc_2 N^{10\delta }R^{5}\varepsilon^{\frac{1}{4}} = \rho_2\,,
		$$
		provided $\tc_2 > C_\delta$. Thus $\mathcal{T}_{v_1, w}$ maps $\mathcal{D}^{V_2}_{\rho_2}$ into itself. We now prove that it is a contraction.
	For any $h_2 \in \cV^{2+2\delta}_{t,x} \cap V_2$, we have
	\begin{equation}\label{diff.v2.j}
	\begin{aligned}
	\partial_{v_2}\mathcal{T}_{v_1,w}(v_2)[h_2]&
	=5\varepsilon^{-1}A^{-1}\Pi_{V_2}\left((v_1+v_2+w)^4h_2\right) \\
	& =5\varepsilon^{-1}\sum\limits_{j_1+j_2+j_3=4}{c}_{j_1j_2j_3}
	A^{-1} \Pi_{V_2}\big(v_1^{j_1}v_2^{j_2}w^{j_3}h_2\big)\,.
	\end{aligned}
	\end{equation}
	If $j_3 = 0$, we argue as in \eqref{cor_ker} to get
	\begin{align}\label{mal.di.testa}
	\big\| A^{-1} \Pi_{V_2} \big({v_1^{j_1}v_2^{j_2}h_2}\big)\big\|_{\cV^{2+2\delta}_{t,x}} &\lesssim_{\delta } (N^{2\delta} \rho_1)^{j_1} (N^{-1} \rho_2)^{j_2} N^{-1}  \|h_2\|_{\cV^{2+2\delta}_{t,x}} \lesssim_\delta N^{-1 + 8\delta} \rho_1^4 \|h_2\|_{\cV^{2+2\delta}_{t,x}}\,,
	\end{align}
	whereas if $j_3 \neq 0$ we argue as in \eqref{lem.f.pv2} to get
	\begin{align}
	\nonumber
	\big\| A^{-1} \Pi_{V_2} \big({v_1^{j_1}v_2^{j_2} w^{j_3} h_2}\big)\big\|_{\cV^{2+2\delta}_{t,x}}&\lesssim_{\delta} N^{-2} (\rho_1 N^{\frac 1 2 + \delta})^{j_1} (\rho_2 N^{-\frac 1 2 -\delta})^{j_2} \rho_3^{j_3} \|h_2\|_{\cV^{2+2\delta}_{t,x}}\\
	\label{brufen}
	& \lesssim_\delta N^{-\frac 1 2 + 3\delta} \rho_1^3 \rho_3 \|h_2\|_{\cV^{2+2\delta}_{t,x}} \lesssim_{ \delta}  N^{-1 + 8\delta} \rho_1^4 \|h_2\|_{\cV^{2+2\delta}_{t,x}}\,.
	\end{align}
	Thus by \eqref{diff.v2.j}, \eqref{mal.di.testa} and \eqref{brufen} and since $\rho_1 = \varepsilon^{\frac 1 4} R$ we deduce
	\begin{equation}\label{diff.t}
	\left\|\partial_{v_2}\mathcal{T}_{v_1,w}(v_2)[h_2]\right\|_{\cV^{2+2\delta}_{t,x}}\leq C_\delta \varepsilon^{-1}N^{-1+8\delta}\varepsilon R^4\|h_2\|_{\cV^{2+2\delta}_{t,x}}= C_\delta R^4 N^{-1+8\delta}\|h_2\|_{\cV^{2+2\delta}_{t,x}}\,,
	\end{equation} 
	from which \eqref{pizza} follows. Thus $\mathcal{T}_{v_1, w}$ is a contraction on $\mathcal{D}_{\rho_2}^{V_2}$. 
\end{proof}

\begin{lemma}[Differentiability of $v_2$]\label{lem.v2.diff}
	The function $v_2(v_1,w)$ is differentiable in $v_1,\, w$, and it satisfies estimates \eqref{de.v2.de.v1}, \eqref{de.v2.de.w}.
\end{lemma}

\begin{proof}
	By \eqref{diff.t}, the operator $ \Id - \partial_{v_2} \mathcal{T}_{v_1, w}(v_2)$ is invertible on $\cV^{2+2\delta}_{t,x}$, with operator norm bounded by $2$. Furthermore, arguing as for \eqref{cor_ker} and \eqref{lem.f.pv2}, one obtains
	$$
	\begin{gathered}
	\left\|\partial_{v_1}\mathcal{T}_{v_1, w}(v_2)[h_1]\right\|_{\cV^{2+2\delta}_{t,x}} \lesssim_{\delta} N^{10\delta} R^4 \|h_1\|_{\cV^1_{t,x}}\,,\\
	\left\|\partial_{w}\mathcal{T}_{v_1, w}(v_2)[h_3]\right\|_{\cV^{2+2\delta}_{t,x}} \lesssim_{\delta } \varepsilon^{-1} N^{\frac 1 2 + 5\delta} \rho_1^4 \|h_3\|_{H^{\frac 12 + \delta}_t \cH^{\frac 3 2 + \delta}_{x}} \lesssim_{\delta} N^{\frac 1 2 + 5\delta} R^4 \|h_3\|_{H^{\frac 12 + \delta}_t \cH^{\frac 3 2 + \delta}_{x}}\,.
	\end{gathered}
	$$
	Then
	$\partial_{v_1} v_2(v_1, w) = \left( \Id - \partial_{v_2} \mathcal{T}_{v_1, w}(v_2)\right)^{-1} \partial_{v_1} \mathcal{T}_{v_1, w} (v_2) $ satisfies \eqref{de.v2.de.v1}, as well as
	$\partial_{w} v_2(v_1, w) = \left( \Id - \partial_{v_2} \mathcal{T}_{v_1, w}(v_2)\right)^{-1} \partial_{w} \mathcal{T}_{v_1, w} (v_2)$ satisfies \eqref{de.v2.de.w}.
\end{proof}

\subsection{Case $p=3$}
For any $R>0$, $\delta>0$ and $\gamma \in (0, \gamma_0)$ let
\begin{equation}\label{parametri.p3}
\rho_1:=\varepsilon^{\frac{1}{2}}R,\quad \rho_2:= \tc_2 R^3 N^{4\delta} \varepsilon^{\frac{1}{2}}, \quad \rho_3:= \tc_3 \gamma^{-1} N^{3+6\delta}R^3\varepsilon^{\frac{3}{2}}\,, \quad N := \varepsilon^{-\frac{1}{\beta}}\,, 
\end{equation}
where we recall that $\varepsilon = \omega^2 - 1 >0$,  according to  \eqref{omega.3}, $\beta>1$, and $\tc_2, \tc_3>0$.
\begin{proposition}[Solution of $v_2$ equation for $p=3$]\label{prop.v2.p3}
	For any $\delta \in (0, \frac 1 8)$, $R>0$ and $\gamma \in (0, \gamma_0)$, let $\rho_1, \rho_2, \rho_3, N$ as in \eqref{parametri.p3}. There exist $\tc_2 := \tc_2(\delta)>0$, $\beta_\delta>1$, $\zeta := \zeta(\delta)>1$, $\epsilon_{\delta, R}>0$ and $C_{1,\delta}, C_{2,\delta}>0$ such that, for any $\beta > \beta_\delta$ and any $\varepsilon$ such that
	\begin{equation}\label{little.gamma.p3}
	\varepsilon N^{\zeta} \gamma^{-1} \leq \epsilon_{\delta, R}\,,
	\end{equation}
	there exists a $C^1$ function $v_2 : \mathcal{D}_{\rho_1} \times \mathcal{D}^W_{\rho_3} \rightarrow V_2 \cap \cV^{2+2\delta}_{t,\eta}$, $(v_1, w) \mapsto v_2(v_1, w)$, where $\mathcal{D}_{\rho_1}$ and $\mathcal{D}^W_{\rho_3}$ are defined as in \eqref{v1.Domain} and \eqref{w.Domain}, satisfying $v_2(0,0) = 0$ and
	\begin{gather}
	\|v_2(v_1, w) \|_{\cV^{2+2\delta}_{t,\eta}} \leq \rho_2\,,\\
	\label{de.v2.de.v1.p3}
	\|\partial_{v_1}v_2(v_1,w)\|_{\mathcal{B}(V_1 \cap \cV^1_{t,\eta}, \cV^{2+2\delta}_{t,\eta})}\leq C_{1,\delta} R^2 N^{4\delta}\,,\\
	\label{de.v2.de.w.p3}
	\|\partial_{w}v_2(v_1,w)\|_{\mathcal{B}(W \cap H^{\frac{1}{2}+\delta}_t \cH^{\frac{3}{2}+\delta}_\eta, \cV^{2+2\delta}_{t,\eta})}\leq C_{2,\delta} R^2 N^{-\frac{1}{2}+3\delta}\,,
	\end{gather}
	such that $v_2(v_1, w)$ solves
	\begin{equation}\label{v2.di.v1.e.w.p3}
	\varepsilon A v_2(v_1, w) -  \Pi_{V_2} \left({(v_1 + v_2(v_1, w) +w)^3}\right) = 0 \,,
	\end{equation}
	where $A := -\Delta_{\mu_1, \mu_2} + \Id$ according to \eqref{def.A}.
\end{proposition}
 We now prove Proposition \ref{prop.v2.p3}. We define the map
 \begin{equation}\label{J.v2.p3}
 	v_2 \mapsto \mathcal{T}_{v_1, w} (v_2) := \varepsilon^{-1} A^{-1} \Pi_{V_2} \left((v_1 + v_2 + w)^3\right)
 \end{equation}
 and show that it is a contraction. 
\begin{lemma}[Contraction for $p=3$]\label{lem.contraz.p3}
	Let $\zeta\geq \frac5 2 + 5\delta$.
	There exist $\epsilon_{\delta, R}>0$ and $C_\delta>0$ such that, if \eqref{little.gamma.p3} holds, then $\mathcal{T}_{v_1, w}$ defined as in \eqref{J.v2.p3} maps $\mathcal{D}^{V_2}_{{\rho}_2}$ into itself, with
	\begin{equation}\label{v2.contrae.p3}
	\| \mathcal{T}_{v_1, w}(v_2) - \mathcal{T}_{v_1, w} (v_2') \|_{\cV^{2+2\delta}_{t,\eta}} \leq C_{\delta } R^2 N^{-1} \|v_2 - v'_2\|_{\cV^{2+2\delta}_{t,\eta}} \quad \forall v_2, v'_2 \in \mathcal{D}^{V_2}_{{\rho}_2}\,.
	\end{equation}
	As a consequence, for any $(v_1, w) \in \mathcal{D}_{\rho_1} \times \mathcal{D}^W_{\rho_3}$ there exists a unique solution $v_2(v_1, w) \in 	\mathcal{D}^{V_2}_{{\rho}_2}$ satisfying \eqref{v2.di.v1.e.w.p3} and $v_2(0,0) = 0$.
\end{lemma}

\begin{proof}
	We start  expanding
	$\mathcal{T}_{v_1,w}(v_2)= \varepsilon^{-1}\sum\limits_{j_1+j_2+j_3=3}c_{j_1,j_2,j_3} A^{-1}\Pi_{V_{2}}\big(v_1^{j_1}v_2^{j_2}w^{j_3}\big)$.
	By \eqref{lem.s.sprime} and Lemma \ref{algebraKer}, one has
	\begin{align*}
	\big\|A^{-1} \Pi_{V_{2}}\big(v_1^{j_1}v_2^{j_2}w^{j_3}\big)\big\|_{\cV^{2+2\delta}_{t,\eta}} & \leq N^{-\frac 3 2 + \delta}  \big\| \Pi_{V_{2}}\big(v_1^{j_1}v_2^{j_2}w^{j_3}\big)\big\|_{\cV^{\frac 3 2 +\delta}_{t,\eta}}\\
	&\lesssim_{\delta} N^{-\frac 3 2 + \delta} (N^{\frac 1 2 + \delta} \rho_1)^{j_1} (N^{-\frac 1 2 -\delta} \rho_2)^{j_2} \rho_3^{j_3}
	\end{align*}
	for any $v_1 \in \mathcal{D}_{\rho_1}$, $v_2 \in \mathcal{D}^{V_2}_{\rho_2}$, $w \in \mathcal{D}^{W}_{\rho_3}$. By \eqref{parametri.p3} and the smallness assumption \eqref{little.gamma.p3}, if $\zeta \geq \frac 5 2 + 5\delta$ one has $N^{-\frac 1 2 -\delta}\rho_2 \leq N^{\frac 1 2 + \delta} \rho_1$ and $\rho_3 \leq N^{\frac 1 2 + \delta} \rho_1$. Then recalling $j_1 + j_2 + j_3 = 3$, there exists $C_\delta >0$ such that
	$$
	\left\|\mathcal{T}_{v_1, w}(v_2)\right\|_{\cV^{2+2\delta}_{t,\eta}} \leq C_{ \delta} \varepsilon^{-1} N^{4\delta} \rho_1^3 = C_\delta \varepsilon^{\frac 1 2} N^{4 \delta} R^3 \leq \tc_2 \varepsilon^{\frac 1 2} N^{4 \delta} R^3 = \rho_2\,,
	$$
	provided $\tc_2 \geq C_\delta$. Thus $\mathcal{T}_{v_1, w}$ maps $\mathcal{D}_{\rho_2}^{V_2}$ into itself. We now prove that it is a contraction. One has
	$$
	\partial_{v_2}\mathcal{T}_{v_1,w}(v_2)[h_2]=3A^{-1}\Pi_{V_{2}}\left( (v_1+v_2+w)^2 h_2\right) \quad \forall h_2 \in \cV^{2+2\delta}_{t,\eta} \cap V_2\,.
	$$
	Applying Lemma \ref{algebraKer}, \eqref{lem.s.sprime}, using \eqref{parametri.p3} and the smallness condition \eqref{little.gamma.p3}, one obtains
	\begin{align*}
	\big\|A^{-1}\Pi_{V_{2}}\big( v_1^{j_1}v_2^{j_2}w^{j_3}h_2\big)\big\|_{\cV^{2+2\delta}_{t,\eta}}&
	\leq N^{-\frac 3 2 + \delta} (N^{\frac 1 2 + \delta} \rho_1)^{j_1} (N^{-\frac 1 2 -\delta} \rho_2)^{j_2} \rho_3^{j_3} N^{-\frac 12 - \delta} \| h_2\|_{\cV^{2+2\delta}_{t,\eta}}\\
	&\lesssim_{ \delta} N^{-1 + 2\delta} \rho_1^{2}  \| h_2\|_{\cV^{2+2\delta}_{t,\eta}} \lesssim_{ \delta} N^{-1 + 2\delta} \varepsilon R^2\,,
	\end{align*}
	using that $j_1 + j_2 + j_3 = 2$.
	Thus
$
	\|\partial_{v_2}\mathcal{T}_{v_1,w}(v_2)[h_2]\|_{\cV^{2+\delta}_{t,x}} \lesssim_{ \delta}  N^{-1 + 2\delta} R^{2}  \| h_2\|_{\cV^{2+2\delta}_{t,\eta}} $
	which is \eqref{v2.contrae.p3}.
\end{proof}

Differentiability of the function $v_2(v_1, w)$ with estimates \eqref{de.v2.de.v1.p3}, \eqref{de.v2.de.w.p3} follows similarly.

\subsection{Case $p=2$}

 For any $\delta > 0$, $R >0$ and $\gamma \in (0, \gamma_0)$ let
\begin{equation}\label{scelta.param.p2.1}
\rho_1 := R \varepsilon^{\frac 12}\,, \quad \rho_2 := \tc_2 \gamma^{-1} R^3 \varepsilon^{\frac 12}\,, \quad \rho_3 := \tc_3 \gamma^{-2} \varepsilon \sqrt{\varepsilon} R^{3} N^b \,, \quad b := 3 + 6\delta\,, \quad N := \varepsilon^{-\frac{1}{\beta}}
\,.
\end{equation}
 where $\varepsilon>0$ is defined by \eqref{omega.ep}, namely $\omega^2 = 1 -\varepsilon$, $\tc_2, \tc_3 >0$ and $\beta>1$. Since equations \eqref{v1.eq}--\eqref{w.eq} for $p=2$ are degenerate, in the sense that $\Pi_{V_2} (v_1 + v_2)^2 = 0$, we perform the translation
\begin{equation}\label{proviamo.traslando}
w = \mathcal{L}_\omega^{-1}(v_1 + v_2)^2 + \tilde w\,.
\end{equation}
We then rewrite \eqref{v2.eq}, \eqref{w.eq} in terms of $v_1, v_2, \tilde w$, and since $\Pi_V(v_1 + v_2)^2 = 0$ by Lemma \ref{lem.vsquare}, we obtain 
\begin{gather}
\label{new.v2.2}
-\varepsilon A v_2 = \Pi_{V_2}\left(2(v_1 + v_2)\left(\mathcal{L}_\omega^{-1}(v_1 + v_2)^2 + \tilde w\right) + \left(\mathcal{L}_\omega^{-1}(v_1 + v_2)^2 + \tilde w\right)^2 \right)\,,\\
\label{new.w.2}
\mathcal{L}_\omega \tilde w = \Pi_{W} \left( 2(v_1 + v_2) \left(\mathcal{L}_\omega^{-1}(v_1 + v_2)^2 + \tilde w\right) + \left(\mathcal{L}_\omega^{-1}(v_1 + v_2)^2 + \tilde w \right)^2 \right)\,.
\end{gather}
\begin{proposition}[Solution of $v_2$ equation for $p=2$]\label{prop.v2.p2}
	For any $\delta \in (0, \frac{1}{10})$, $R>0$ and $\gamma \in (0, \gamma_0)$, let $\rho_1, \rho_2, \rho_3, N$ be as in \eqref{scelta.param.p2.1}. There exist $\tc_{2}:= \tc_2(\delta)>0$, $\beta_{\delta}>1$, $\mathtt{b}:= \mathtt{b}(\delta)>0$, $\epsilon_{\delta, R}>0$ and $C_{1,\delta}, C_{2,\delta}>0$ such that, for any $\beta > \beta_\delta$ and any  $\varepsilon>0$ such that $\omega \in \Omega_\gamma$ and
	\begin{equation}\label{little.gamma.p2}
	N^{-\mathtt{b}} \gamma^{-1}  \leq \epsilon_{\delta, R}\,, \quad N:= \varepsilon^{-\frac 1 \beta}\,,
	\end{equation}
	there exists a $C^1$ function ${v}_2: \mathcal{D}_{\rho_1} \times {\mathcal{D}}^W_{\rho_3} \rightarrow V_2 \cap \cV^{2+2\delta}_{t,x}$, $(v_1, \tilde w) \mapsto {v}_2(v_1, \tilde w)$, where $\mathcal{D}_{\rho_1}$ and $\mathcal{D}^W_{\rho_3}$ are defined as in \eqref{v1.Domain} and \eqref{w.Domain},
	satisfying 	$v_2(0, 0) = 0$ and
	\begin{gather}
	\label{in.che.bolla.p2}
	\|v_2(v_1, \tilde w)\|_{\cV^{2+2\delta}_{t,x}} \leq \rho_2\,,\\
	\label{de.v2.de.v1.p2}
	\|\partial_{v_1}v_2(v_1, \tilde w)\|_{\mathcal{B}(V_1 \cap \cV^1_{t,x}, \cV^{2+2\delta}_{t,x})}\leq C_{1,\delta} \gamma^{-1} R^2\,,\\
	\label{de.v2.de.w.p2}
	\|\partial_{\tilde w}v_2(v_1, \tilde w)\|_{\mathcal{B}(W \cap H^{\frac{1}{2}+\delta}_t \cH^{\frac{3}{2}+\delta}_x,\cV^{2+2\delta}_{t,x})}\leq C_{2,\delta} \varepsilon^{-\frac 12} N^{-1 +2\delta} R\,,
	\end{gather}
	such that $v_2(v_1, \tilde w)$ solves equation \eqref{new.v2.2}.
\end{proposition}

We now prove Proposition \ref{prop.v2.p2}. For any $(v_1, \tilde w) \in \mathcal{D}_{\rho_1} \times \mathcal{D}^W_{\rho_3}$ we look for a solution of \eqref{new.v2.2} as a fixed point of the map which to $v_2$ associates
\begin{equation}\label{J.v2}
\mathcal{T}_{v_1, \tilde w}(v_2):=-\varepsilon^{-1}A^{-1} \Pi_{V_2} \left(2(v_1 + v_2)\left(\mathcal{L}_\omega^{-1}(v_1 + v_2)^2 + \tilde w\right) + \left(\mathcal{L}_\omega^{-1}(v_1 + v_2)^2 + \tilde w\right)^2 \right)
\end{equation}
with $A = -\Delta_{\mathbb{S}^3}^{ss} + \Id$ as in \eqref{def.A}. We shall use the following technical lemma:
\begin{lemma}\label{lem.v2.leq.v1}
	Let $\rho_1$ and $\rho_2$ as in \eqref{scelta.param.p2.1}. There exists $\epsilon_{R, \tc_2}>0$ such that, if $N^{-1-2\delta} \gamma^{-1} < \epsilon_{R, \tc_2}$, 
	then for any $s \in [0, 2 + 2\delta]$, any $v_1 \in \mathcal{D}_{\rho_1}$ and $v_2 \in \mathcal{D}^{V_2}_{\rho_2}$ one has
	\begin{equation}\label{hamtaro}
	\| v_1\|_{\cV^s_{t,x}} \leq N^{\max\{0, s-1\}} \rho_1\,, \quad  \|v_2\|_{\cV^{s}_{t,x}} \leq N^{\max\{0, s -1\}} \rho_1\,.
	\end{equation}
\end{lemma}
\begin{proof}
	The estimate on $\| v_1\|_{\cV^s_{t,x}}$ follows from \eqref{lem.s.sprime}.
	For any $s \in [0, 2 +2\delta]$ and $v_2 \in \mathcal{D}^{V_2}_{\rho_2}$, by \eqref{lem.s.sprime} and \eqref{scelta.param.p2.1} one has
	$$
	\|v_2\|_{\cV^{s}_{t,x}} \leq N^{-2-2\delta+ s} \|v_2\|_{\cV^{2 +2\delta}_{t,x}} \leq N^{-2-2\delta+ s} \rho_2
	= N^{-2 -2\delta + s} \tc_2 \gamma^{-1}  R^2 \rho_1 \leq N^{\max\{0, s-1\}} \rho_1\,,
	$$
	since $-2-2\delta + s \leq \max\{0, s-1\} - 1 - 2\delta$ and $N^{-1-2\delta} \gamma^{-1} < (\tc_2 R^2)^{-1}=:\epsilon_{R, \tc_2}$.
\end{proof}
The next Lemma is based on the Strichartz-type estimates of Proposition \ref{lem.strich.Lom}:
\begin{lemma}[Contraction]\label{lem.v2-2.contraz}
	There exist $\mathtt{b}_\delta \in (0, 1)$, $\epsilon_{\delta, R}>0$ and $C_\delta>0$ such that, if \eqref{little.gamma.p2} holds with $\mathtt{b} \geq \mathtt{b}_\delta$, then for any $v_1 \in \mathcal{D}_{\rho_1}$ and $\tilde w \in \mathcal{D}^W_{\rho_3}$,  $\mathcal{T}_{v_1, \tilde w}$ defined in \eqref{J.v2} maps $\mathcal{D}^{V_2}_{\rho_2}$ into itself, with
	\begin{equation}\label{contrae.p2}
	\left\| \mathcal{T}_{v_1, \tilde w} (v_2)  - \mathcal{T}_{v_1, \tilde w} (v'_2) \right\|_{\cV^{2+2\delta}_{t,x}} \leq C_{ \delta} \gamma^{-1} R^2 N^{-\frac 4 3} \|v_2 - v'_2\|_{\cV^{2+2\delta}_{t,x}} \quad \forall v_2, v'_2 \in \mathcal{D}^{V_2}_{\rho_2}\,.
	\end{equation}
	As a consequence, for any $ (v_1, \tilde w) \in \mathcal{D}_{\rho_1} \times \mathcal{D}^W_{\rho_3}$ there exists a unique $v_2(v_1, \tilde w)$ solving \eqref{new.v2.2}, satisfying $v_2(0,0) = 0$.
\end{lemma}
\begin{proof}
	We expand $	\mathcal{T}_{v_1,\tilde w}(v_2)$ in \eqref{J.v2} as $\mathcal{T}_{v_1,\tilde w}(v_2)= \mathcal{T}_1 + \mathcal{T}_2 + \mathcal{T}_3 + \mathcal{T}_4 + \mathcal{T}_5$, with
	$$
	\begin{gathered}
	\mathcal{T}_1 := -2\varepsilon^{-1}A^{-1}\Pi_{V_{2}}\left((v_1 + v_2) \mathcal{L}_\omega^{-1} (v_1 + v_2)^2\right) \,, \quad
	\mathcal{T}_2 := -2\varepsilon^{-1}A^{-1}\Pi_{V_{2}}\left((v_1 + v_2)  \tilde w\right)\,,\\
	\mathcal{T}_3 := -\varepsilon^{-1}A^{-1}\Pi_{V_{2}} \left( \left( \mathcal{L}_\omega^{-1} (v_1 + v_2)^2\right)^2\right)\,, \quad
	\mathcal{T}_4 := -\varepsilon^{-1}A^{-1}\Pi_{V_{2}} \tilde w ^2\,, \\
	\mathcal{T}_5 := -2 \varepsilon^{-1}A^{-1}\Pi_{V_{2}} \left( \tilde w  \mathcal{L}_\omega^{-1} (v_1 + v_2)^2 \right)\,,
	\end{gathered}
	$$
	and we estimate each term separately. 
	$\mathcal{T}_1$ is estimated using Item 2 of Proposition \ref{lem.strich.Lom}, which gives
	\begin{equation}\label{T.1}
	\begin{aligned}
	\| \mathcal{T}_1\|_{\cV^{2+2\delta}_{t,x}} &\leq 2 \varepsilon^{-1} \left\| \Pi_{V_{2}}\left((v_1 + v_2) \mathcal{L}_\omega^{-1} (v_1 + v_2)^2\right)\right\|_{\cV^{2\delta}_{t,x}}\\
	&= 2\varepsilon^{-1} \sup_{h_2 \in V_2 \cap \cV^{-2\delta}_{t,x} \atop \|h_2\|_{\cV^{-2\delta}_{t,x} } \leq 1} \left| \int_{\mathbb{T}} \int_{0}^\pi (v_1 + v_2) h_2 \mathcal{L}_\omega^{-1}(v_1 + v_2)^2 \sin^2(x) \,\dbar x \dbar t\right|\\
	&\lesssim_{\delta} \varepsilon^{-1} \gamma^{-1} \|v_1 + v_2\|_{\cV^{\frac 2 3 + 2 \delta}_{t,x}}^3  \lesssim_{\delta} \varepsilon^{-1} \gamma^{-1} \big(\|v_1\|_{\cV^{\frac 2 3 + 2 \delta}_{t,x}} + \|v_2\|_{\cV^{\frac 2 3 + 2 \delta}_{t,x}} \big)^3\,.
	\end{aligned}
	\end{equation}
	$\mathcal{T}_2$ is estimated using \eqref{lem.s.sprime} and Lemma \ref{algebraKer}: one has
	\begin{equation}\label{T.2}
	\begin{aligned}
	\|\mathcal{T}_2\|_{\cV^{2+2\delta}_{t,x}}& \leq \varepsilon^{-1}\left\|\Pi_{V_2} (v_1 + v_2)  \tilde w\right \|_{\cV^{2\delta}_{t,x}}
	\lesssim_{\delta} \varepsilon^{-1} N^{-\frac 3 2 + \delta} 
	\left\|\Pi_{V_2} (v_1 + v_2)  \tilde w\right \|_{\cV^{\frac 3 2 + \delta}_{t,x}} \\
	&\lesssim_{\delta}  \varepsilon^{-1} N^{-\frac 3 2 + \delta} 
	\Big(\|v_1 \|_{\cV^{\frac 3 2 + \delta}_{t,x}} +
	\| v_2\|_{\cV^{\frac 3 2 + \delta}_{t,x}} \Big) \| \tilde w\|_{H^{\frac 1 2 + \delta}_t \cH^{\frac 3 2 + \delta}_{x}}\,.
	\end{aligned}
	\end{equation}
	$\mathcal{T}_3$ and $\mathcal{T}_5$ are estimated using \eqref{algebra.prop}, \eqref{lem.s.sprime}, \eqref{hs.linf}, \eqref{cor.12delta} and Lemma \ref{lemma.sono.pochissimi}, which give 
	\begin{equation}\label{T.3}
	\begin{aligned}
	\|\mathcal{T}_3\|_{\cV^{2+2\delta}_{t,x}} 
	& \lesssim_{\delta} 
	\varepsilon^{-1}\left\|\Pi_{V_{2}} \Big( \left( \mathcal{L}_\omega^{-1} (v_1 + v_2)^2\right)^2\Big)\right\|_{\cV^{2\delta}_{t,x}} 
	\lesssim_{\delta} \varepsilon^{-1} 
	N^{-\frac 3 2 + \delta} \left\|  \left( \mathcal{L}_\omega^{-1} (v_1 + v_2)^2\right)^2 \right\|_{H^{\frac 1 2 + \delta}_t \cH^{\frac 3 2 + \delta}_{x}}\\
	&\lesssim_{\delta} \gamma^{-2} \varepsilon^{-1} N^{-\frac 3 2 + \delta} \big(\|v_1\|_{\cV^{2+2\delta}_{t,x}} +\| v_2\|_{\cV^{2+2\delta}_{t,x}}\big)^4\,,
	\end{aligned}
	\end{equation}
	and
	\begin{equation}\label{T.5}
	\begin{aligned}
	\|\mathcal{T}_5\|_{\cV^{2+2\delta}_{t,x}} &\lesssim_{\delta} \varepsilon^{-1} N^{-\frac 3 2 + \delta} \left\|\Pi_{V_2} \left(\tilde{w} \mathcal{L}_\omega^{-1}(v_1 + v_2)^2\right) \right\|_{\cV^{\frac 3 2 + \delta}_{t,x}}\\
	&\lesssim_{\delta} \gamma^{-1} \varepsilon^{-1}  N^{-\frac 3 2 + \delta} \|\tilde w\|_{H^{\frac 1 2 + \delta}_t \cH^{\frac 3 2 + \delta}_x} 
	\big(\|v_1\|_{\cV^{2+2\delta}_{t,x}} + \|v_2\|_{\cV^{2+2\delta}_{t,x}}\big)^2\,.
	\end{aligned}
	\end{equation}
	Finally, $\mathcal{T}_4$ is estimated using algebra property \eqref{algebra.prop}:
	\begin{equation}\label{T.4}
	\|\mathcal{T}_4\|_{\cV^{2+2\delta}_{t,x}} \lesssim_{\delta}  \varepsilon^{-1}\|\Pi_{V_{2}} \tilde w ^2\|_{\cV^{2\delta}_{t,x}} \lesssim_{\delta} \varepsilon^{-1} N^{-\frac 3 2 + \delta} \|\tilde w\|^2_{H^{\frac 1 2 +\delta}_t \cH^{\frac 32+ \delta}_{x}}\,.
	\end{equation}
	Then by Lemma \ref{lem.v2.leq.v1}, recalling the definitions of $\rho_1, \rho_2, \rho_3$ as in \eqref{scelta.param.p2.1} and combining \eqref{T.1}, \eqref{T.2}, \eqref{T.3}, \eqref{T.4}, \eqref{T.5}, there exists a positive constant $C_\delta$ such that
	$$
	\begin{aligned}
	\| \mathcal{T}_{v_1, \tilde w} (v_2) \|_{\cV^{2+2\delta}_{t,x}} &\leq C_\delta \big( \gamma^{-1} \sqrt{\varepsilon} R^3 + \tc_3 \gamma^{-2} \varepsilon N^{-1 + 2\delta + b} R^4 + \gamma^{-2} \varepsilon N^{\frac 5 2 + 9\delta} R^4\\
	&\quad + \tc_3^2 \gamma^{-4} R^{6} N^{-\frac 3 2 + \delta + 2b} \varepsilon^{2}+ \tc_3 \gamma^{-3}\varepsilon \sqrt{\varepsilon} R^5 N^{\frac 1 2 + 5\delta + b}\big)\\
	&\leq 2 C_\delta \gamma^{-1} \sqrt{\varepsilon} R^3 \leq \gamma^{-1} \tc_2 \sqrt{\varepsilon} R^3 = \rho_2\,,
	\end{aligned}
	$$
	provided \eqref{little.gamma.p2} holds for some $\mathtt{b} \in (0,1)$ and $\epsilon_{R, \delta}$ small enough and $\tc_2 \geq 2 C_\delta$. 
	We now prove that $\mathcal{T}_{v_1, \tilde w}$ is a contraction. We actually prove that $\partial_{v_2} \mathcal{T}_{v_1, \tilde w} \in \mathcal{B}\big(V_2 \cap \cV^{2+2\delta}_{t,x}, \cV^{2+2\delta}_{t,x}\big)$. Indeed, one has
	$\partial_{v_2} \mathcal{T}_{v_1, \tilde w} [h_2] = D_1[h_2] + D_2[h_2] + D_3 [h_2]+ D_4[h_2]$\,,
	with
	$$
	\begin{aligned}
	D_1[h_2] &:=-2 \varepsilon^{-1} A^{-1} \Pi_{V_2}\left(h_2 \mathcal{L}_\omega^{-1}(v_1 + v_2)^2\right)\,,\\
	D_2[h_2] &:= -2 \varepsilon^{-1} A^{-1} \Pi_{V_2}\left( h_2 \tilde w\right)\,,\\
	D_3[h_2] &:= -4 \varepsilon^{-1} A^{-1} \Pi_{V_2}\left((v_1 + v_2)  \mathcal{L}_\omega^{-1}\left((v_1 + v_2) h_2\right)\right)\,,
	\\
	D_4[h_2] &:=- 4 \varepsilon^{-1} A^{-1} \Pi_{V_2}\left( \left(\mathcal{L}_\omega^{-1}(v_1 + v_2)^2 + \tilde w\right) \mathcal{L}_\omega^{-1}\left((v_1 + v_2) h_2\right)\right)\,.
	\end{aligned}
	$$
	We proceed estimating separately all terms. By Item 2 of Proposition \ref{lem.strich.Lom}, by Lemma \ref{lem.v2.leq.v1} and using the definitions of the parameters $\rho_1, \rho_2, \rho_3$, one has
	\begin{align}
	\nonumber
	\|D_1[h_2]\|_{\cV^{2+2\delta}_{t,x}} &\lesssim_\delta \varepsilon^{-1} \left\|\Pi_{V_2} \left(h_2 \mathcal{L}_\omega^{-1}(v_1 + v_2)^2\right)\right\|_{\cV^{2 \delta}_{t,x}}\\
	\nonumber
	&\lesssim_{\delta} \gamma^{-1} \varepsilon^{-1} \|h_2\|_{\cV^{\frac 2 3 + 2\delta}_{t,x}} \|v_1 + v_2\|_{\cV^{\frac 2 3 +2\delta}_{t,x}}^2 \lesssim_{\delta} \gamma^{-1} \varepsilon^{-1} N^{-\frac 4 3} \|h_2\|_{\cV^{2 + 2\delta}_{t,x}} \|v_1 + v_2\|_{\cV^{\frac 2 3 +2\delta}_{t,x}}^2\\
	& \lesssim_{\delta} \gamma^{-1}\varepsilon^{-1} \rho_1^2 N^{-\frac 4 3} \|h_2\|_{\cV^{2+2\delta}_{t,x}} \lesssim_{\delta} \gamma^{-1} R^2 N^{-\frac 4 3} \|h_2\|_{\cV^{2+2\delta}_{t,x}}\,. \label{D1}
	\end{align}
	The estimate of $D_3[h_2]$ is the same, and gives
	\begin{equation}\label{D3}
	\| D_3[h_2]\|_{\cV^{2+2\delta}_{t,x}} \lesssim_{\delta } \gamma^{-1}  R^2 N^{-\frac 4 3} \|h_2\|_{\cV^{2+2\delta}_{t,x}}\,.
	\end{equation}
	The estimate of $D_2[h_2]$ is analogous to the estimate of $\mathcal{T}_2$, and yields
	\begin{equation}\label{D2}
	\begin{aligned}
	\| D_2[h_2]\|_{\cV^{2+2\delta}_{t,x}} &\lesssim_{\delta } \varepsilon^{-1} N^{-\frac 3 2 + \delta} 
	\|h_2\|_{\cV^{\frac 3 2 + \delta}_{t,x}} \| \tilde w\|_{H^{\frac 1 2 + \delta}_t \cH^{\frac 3 2 + \delta}_{x}}\\
	&\lesssim_{\delta } \varepsilon^{-1} N^{-2} \|h_2\|_{\cV^{2+2\delta}_{t,x}} \rho_3 \lesssim_{ \delta} \tc_3 \gamma^{-2} R^{3} \sqrt{\varepsilon} N^{-2 + b} \|h_2\|_{\cV^{2+2\delta}_{t,x}} \,.
	\end{aligned}
	\end{equation}
	We finally estimate $D_4[h_2]$ using algebra property \eqref{algebra.prop} and Lemma \ref{lemma.sono.pochissimi}. One gets
	\begin{align*}
&	\|D_4[h_2]\|_{\cV^{2+2\delta}_{t,x}} \lesssim_{\delta} N^{-\frac 3 2 + \delta} \varepsilon^{-1} \left\| \Pi_{V_2}\left( \left(\mathcal{L}_\omega^{-1}(v_1 + v_2)^2 + \tilde w\right) \mathcal{L}_\omega^{-1}\left((v_1 + v_2) h_2\right)\right) \right\|_{\cV^{\frac 3 2 + \delta}_{t,x}}\\
	&\lesssim_{\delta } N^{-\frac 3 2 + \delta} \varepsilon^{-1} \big(\left\|\mathcal{L}_\omega^{-1}(v_1 + v_2)^2\right\|_{H^{\frac 1 2 + \delta}_t \cH^{\frac 3 2 + \delta}_x} + \|\tilde w\|_{H^{\frac 1 2 + \delta}_t \cH^{\frac 3 2 + \delta}_x}\big) \left\|\mathcal{L}_\omega^{-1}\left((v_1 + v_2) h_2\right)	\right\|_{H^{\frac 1 2 + \delta}_t \cH^{\frac 3 2 + \delta}_x}\\
	&\lesssim_{\delta} N^{-\frac 3 2 + \delta} \gamma^{-2} \varepsilon^{-1} \big(\|v_1 + v_2\|^2_{\cV^{2+2\delta}_{t,x}}  + \|\tilde w\|_{H^{\frac 1 2 + \delta} \cH^{\frac 3 2 + \delta}_x}\big) \|v_1 + v_2\|_{\cV^{2+2\delta}_{t,x}} \|h_2\|_{\cV^{2+2\delta}_{t,x}}\,.
	\end{align*}
	Then using Lemma \ref{lem.v2.leq.v1} and the definitions of parameters $\rho_1, \rho_2, \rho_3$ as in \eqref{scelta.param.p2.1} one has
	\begin{equation}\label{D4}
	\begin{aligned}
	\|D_4[h_2]\|_{\cV^{2+2\delta}_{t,x}} &\lesssim_{\delta} N^{-\frac 3 2 + \delta} \gamma^{-2} \varepsilon^{-1} \left( \varepsilon R^2 N^{2(1+2\delta)} + \tc_3 \gamma^{-2} R^3 \varepsilon^{\frac 3 2} N^{b} \right) R \sqrt{\varepsilon} N^{1+2\delta} \|h_2\|_{\cV^{2+2\delta}_{t,x}}\\
	&\lesssim_{\delta} \gamma^{-2} R^3 N^{3(1+2\delta) - \frac 3 2 + \delta}  \sqrt{\varepsilon}  \|h_2\|_{\cV^{2+2\delta}_{t,x}}\,.
	\end{aligned}
	\end{equation}
	Thus, combining \eqref{D1}, \eqref{D2}, \eqref{D3}, \eqref{D4} and assuming \eqref{little.gamma.p2}, one gets
	$$
	\|\partial_{v_2} \mathcal{T}_{v_1, \tilde w}(v_2)[h_2] \|_{\cV^{2+2\delta}_{t,x}} \lesssim_{\delta} \gamma^{-1} R^2 N^{-\frac 4 3} \|h_2\|_{\cV^{2+2\delta}_{t,x}}\,,
$$
	which implies \eqref{contrae.p2} and 
	 that $\mathcal{T}_{v_1, \tilde w}$ is a contraction. Finally, since  $\mathcal{T}_{0, 0} = \varepsilon^{-1} A^{-1} \Pi_{V_2} \big(2 v_2 \mathcal{L}_\omega^{-1}(v_2^2) + \left(\mathcal{L}_\omega^{-1} (v_2^2)\right)^2 \big)$ vanishes at $v_2 =0$, we also have $v_2(0,0) = 0$.
\end{proof}

Finally, with analogous arguments to the ones in the proof of Lemma \ref{lem.v2-2.contraz}, one obtains differentiability of $v_2(v_1, \tilde w)$ with respect to $v_1$ and $\tilde w$ with estimates \eqref{de.v2.de.v1.p2}, \eqref{de.v2.de.w.p2}.

\section{Solution of the range equation}\label{sec.w}

In this section we solve the range equation \eqref{w.eq} in the algebra spaces $H^{\frac 1 2 + \delta}_t \cH^{\frac 3 2 + \delta}_z$, where $v_2 = v_2(v_1, w)$ is the solution of \eqref{v2.eq}, namely we find $w$ such that
\begin{equation}\label{w2.eq.v2}
	\mathcal{L}_\omega w - \Pi_{W} (v_1 + v_2(v_1, w) +w)^p= 0 \,.
\end{equation}

\subsection{Cases $p=5$ and $p=3$}

\begin{proposition}[Solution of the range equation for $p=5$]\label{prop.w}
	For any $\delta \in (0, \frac{1}{8})$, $\gamma \in (0, \gamma_0)$ and $R>0$, let $\rho_1, \rho_2, \rho_3, N$ as in \eqref{parametri.p5}.
	There exist $\beta_\delta>0$, $\tc_3 := \tc_3(\delta)>0$, $\zeta:=\zeta(\delta)>0$, $\epsilon_{\delta, R}>0$ and $C_\delta>0$ such that, for any $\beta > \beta_\delta$ and any $\varepsilon>0$ such that $\omega \in \Omega_\gamma$ and \eqref{little.gamma} holds, there exists a $C^1$ function $w: \mathcal{D}_{\rho_1} \rightarrow W \cap H^{\frac{1}{2} +\delta}_t \cH^{\frac 3 2 + \delta}_x$, $v_1 \mapsto w(v_1)$, satisfying $w(0) = 0$,
	\begin{align}
	& \|w\|_{H^{\frac 1 2 + \delta}_t \cH^{\frac 3 2 + \delta}_x} \leq \rho_3\,,\\
	& \label{de.w.de.v1}
	\|d_{v_1}w(v_1)\|_{\mathcal{B}(V_1 \cap \cV^1_{t,x},H^{\frac{1}{2}+\delta}_t \cH^{\frac{3}{2}+\delta}_x)}\leq C_\delta \gamma^{-1} N^{5 + 10\delta} \varepsilon R^4\,,
	\end{align}
	which solves equation \eqref{w2.eq.v2}.
\end{proposition}

We are going to prove that the map
	\begin{equation}\label{tw}
		\mathcal{T}_{v_1}:
		w\mapsto  \mathcal{L}_{\omega}^{-1}\Pi_{W}\left({(v_1+v_2(v_1,w)+w)^5}\right)
	\end{equation} 
	is a contraction,
	with $v_2(v_1, w)$ as in Proposition \ref{prop.v2}.
\begin{lemma}[Contraction]\label{prop.tw.contraction}
	Assume the smallness condition \eqref{little.gamma} holds with $\zeta \geq \frac 9 2 + 9 \delta$. The map  $\mathcal{T}_{v_1}$ in \eqref{tw} maps $\mathcal{D}^W_{\rho_3}$ in itself, and there exists $C_\delta>0$ such that
	\begin{equation}\label{de.w.1}
	\| \mathcal{T}_{v_1} (w) - \mathcal{T}_{v_1} (w') \|_{H^{\frac 1 2 + \delta}_t \cH^{\frac 3 2 +\delta}_x} \leq C_\delta \gamma^{-1} N^{\frac{9}{2}+13\delta}R^8 \| w - w'\|_{H^{\frac 12 +\delta}_t \cH^{\frac 3 2 +\delta}_x} \quad \forall w, w' \in \mathcal{D}^W_{\rho_3}\,.
	\end{equation}
	As a consequence, for any 
	$ v_1 \in \mathcal{D}_{\rho_1}$ there exists a unique $w(v_1) \in \mathcal{D}^W_{\rho_3}$ solving \eqref{w2.eq.v2}, and such that $w(0) = 0$.
\end{lemma}
\begin{proof}
	For brevity, we denote $v_2(v_1, w) = v_2$. By \eqref{tw}, Lemma \ref{lemma.sono.pochissimi}, algebra property \eqref{algebra.prop}, \eqref{lem.s.sprime} and Lemma \ref{lemma.come.vuoi}, we have
	\begin{align}
	\nonumber
	\big\| \mathcal{T}_{v_1} (w) \big\|_{\HtHx}
	&\lesssim \gamma^{-1} \sum_{j_1 + j_2 + j_3 =5} \big\| \Pi_{W} \big(v_1^{j_1} v_2^{j_2} w^{j_3}\big)\big\|_{\HtHx}\\
	\nonumber
	&\lesssim_\delta \gamma^{-1} \sum_{j_1 + j_2 + j_3 =5} 
	\big(N^{1+2\delta}\|v_1\|_{\cV^1_{t,x}}\big)^{j_1}\|v_2\|^{j_2}_{\cV^{2+2\delta}_{t,x}}\|w\|^{j_3}_{H^{\frac{1}{2}+\delta}_t \cH^{\frac{3}{2}+\delta}_x}\,,\\
	\label{caffe.tantissimo}
	& \lesssim_\delta \gamma^{-1}  \sum_{j_1 + j_2 + j_3 =5} (N^{1+2\delta}\rho_1)^{j_1}\rho_2^{j_2}\rho_3^{j_3}\,.
	\end{align}
	Then by the smallness assumption \eqref{little.gamma}, we have $\rho_2 \leq N^{1+2\delta} \rho_1$ and $\rho_3 \leq N^{1+2\delta} \rho_1$, thus
	by \eqref{caffe.tantissimo} there exists a constant $C_\delta>0$ such that
	$$
	\left\| \mathcal{T}_{v_1} (w) \right\|_{\HtHx} \leq C_{\delta} \gamma^{-1} N^{5 + 10\delta} \varepsilon^{\frac 5 4} R^{5} < \tc_{3} \gamma^{-1} N^{5 + 10\delta} \varepsilon^{\frac 5 4} R^{5} = \rho_3\,,
	$$
	provided $\tc_3 > C_\delta$. This proves that $\mathcal{T}_{v_1}$ maps $\mathcal{D}^W_{\rho_3}$ in itself. We now prove \eqref{de.w.1}.
	Let $h_3\in H_t^{\frac{1}{2}+\delta}\cH^{\frac{3}{2}+\delta}_x$, then  arguing as to obtain \eqref{caffe.tantissimo}, and using estimate \eqref{de.v2.de.w}, we obtain
	\begin{align}
	\nonumber
	\left\|\partial_{w}\mathcal{T}_{v_1}(w)[h_3]\right\|_{\HtHx}
	&=5\left\|\mathcal{L}_{\omega}^{-1}\Pi_{W}\left( {(v_1+v_2+w)^4} (h_3 +\partial_{w}v_2[h_3])\right)\right\|_{\HtHx}\\
	\nonumber
	&\lesssim \gamma^{-1} \sum\limits_{j_1+j_2+j_3=4}\big\| \Pi_{W}
	\big( {v_1^{j_1}v_2^{j_2}w^{j_3}} \left(h_3 +\partial_{w}v_2[h_3]\big)\right)
	\big\|_{\HtHx}\\
	\nonumber
	& \lesssim_\delta \gamma^{-1} \sum\limits_{j_1+j_2+j_3=4}  (N^{1+2\delta} \rho_1)^{j_1} \rho_2^{j_2} \rho_3^{j_3}\big( 1 + N^{\frac 1 2 + 5\delta} R^4\big) \|h_3\|_{\HtHx}\\
	\label{teiera}
	& \lesssim_\delta \gamma^{-1} \sum\limits_{j_1+j_2+j_3=4}  N^{\frac 9 2 + 13 \delta} R^8 \varepsilon \|h_3\|_{\HtHx}\,,
	\end{align}
	where we have used $\rho_2 \leq N^{1+2\delta} \rho_1$ and $\rho_3 \leq N^{1+2\delta} \rho_1$. Estimate \eqref{teiera} then gives
	\eqref{de.w.1}. Then the map $\mathcal{T}_{v_1}$ is a contraction, and since by  Proposition \ref{prop.v2} $v_2(0, 0) =0$, we have $w(0) = 0$.
\end{proof}
\begin{lemma}[Differentiability of $w(v_1)$]\label{lem.de.ff.w}
The function $w(v_1)$  is differentiable in $v_1$ and 
 \eqref{de.w.de.v1} holds.
\end{lemma}

\begin{proof}
Due to \eqref{teiera} and the smallness condition \eqref{little.gamma}, one has $\left[\Id - \partial_{w} \mathcal{T}_{v_1}(w)\right]^{-1} \in \mathcal{B} ( \HtHx )$, with norm bounded by $2$. Let us denote for brevity $v_2(v_1, w(v_1)) = v_2$ and $w(v_1) = w$. 
Then for any $h_1 \in V_1$, by Lemma \ref{lemma.sono.pochissimi}, 
\eqref{algebra.prop}, \eqref{parametri.p5}, smallness condition \eqref{little.gamma} and estimate \eqref{de.v2.de.w}, one has
\begin{equation*}
	\begin{aligned}
	\left\|\partial_{v_1}\mathcal{T}_{v_1}(w)[h_1]
	\right\|_{\HtHx} &= 5 \left\|\mathcal{L}_{\omega}^{-1}\Pi_W\big( (v_1+v_2+w)^4\left(h_1+\partial_{v_1}v_2[h_1]\right)\big) \right\|_{\HtHx}\\
	&\lesssim \gamma^{-1} \sum\limits_{j_1+j_2+j_3=4} \big\|\Pi_W
	\big( v_1^{j_1}v_2^{j_2}w^{j_3}\left(h_1+\partial_{v_1}v_2[h_1]\right)\big) 
	\big\|_{\HtHx}\\
	&\lesssim_{ \delta}  \gamma^{-1} \sum\limits_{j_1+j_2+j_3=4}  (N^{1+2\delta} \rho_1)^{j_1} \rho_2^{j_2} \rho_3^{j_3} \big( N^{1+2\delta}  + N^{10\delta} R^4\big) \|h_1\|_{\cV^1_{t,x}}\\
	&\lesssim_{\delta }  \gamma^{-1} N^{5+10\delta}\varepsilon R^4 \|h_1\|_{\cV^1_{t,x}}\,.
	\end{aligned}
\end{equation*}
This gives \eqref{de.w.de.v1}.
\end{proof}

If $p=3$ the proof of the existence of a solution $w$ of 
 \eqref{w2.eq.v2} follows arguing as in the case $p=5$.

\begin{proposition}[Solution of the range equation for $p=3$]\label{prop.w.p3}
	For any $\delta \in (0, \frac 1 8)$, $\gamma \in (0, \gamma_0)$ and $R>0$, let $\rho_1, \rho_2, \rho_3, N$ as in \eqref{parametri.p3}. There exist $\beta_\delta>0$, $\tc_3:= \tc_3(\delta)>0$, $\zeta:= \zeta(\delta)>0$, $\epsilon_{R, \delta}>0$ and $C_\delta>0$ such that, for any $\beta> \beta_\delta$ and any $\varepsilon>0$ such that $\omega \in \Omega_\gamma$ and \eqref{little.gamma.p3} holds, there exists a $C^1$ function $w: \mathcal{D}_{\rho_1} \rightarrow W \cap \HtHeta$, $v_1 \mapsto w(v_1)$, satisfying $w(0) = 0$,
	\begin{equation*}
	 \|w\|_{\HtHeta} \leq \rho_3\,, \quad 
	\|d_{v_1} w(v_1) \|_{\mathcal{B}(V_1 \cap \cV^1_{t,\eta}, \HtHeta)} \leq C_\delta  \varepsilon \gamma^{-1} N^{3+6\delta}R^2\,,
	\end{equation*}
	which solves \eqref{w2.eq.v2}.
\end{proposition}

\subsection{Case $p=2$}

	Here we solve equation \eqref{new.w.2} with $v_2 = v_2(v_1, \tilde w)$, which reads
	\begin{equation}\label{trenitalia}
	\begin{aligned}
	\mathcal{L}_\omega \tilde w &= \Pi_{W} \left( 2(v_1 + v_2(v_1, \tilde w)) \left(\mathcal{L}_\omega^{-1}(v_1 + v_2(v_1, \tilde w))^2 + \tilde w\right) 
	+  
	 \left(\mathcal{L}_\omega^{-1}(v_1 + v_2(v_1, \tilde w))^2 + \tilde w \right)^2  \right)
	\end{aligned}
	\end{equation}
	where $v_2(v_1, \tilde w)$ is the function in Proposition \ref{prop.v2.p2}.
 \begin{proposition}[Solution of the range equation for $p=2$]\label{prop.w.p2}
	For any $\delta \in (0, \frac 1 8)$, $\gamma \in (0, \gamma_0)$ and $R>0$, assume that $\rho_1, \rho_2, \rho_3, N$ are as in \eqref{scelta.param.p2.1} and satisfy the assumptions of Proposition \ref{prop.v2.p2}. There exist $\tc_3:=\tc_3(\delta)>0$, $\beta_\delta>0$, $\mathtt{b}:=\mathtt{b}(\delta) >0$, $\epsilon_{\delta, R}>0$ and $C_\delta>0$ such that, for any $\beta > \beta_\delta$ and any $\varepsilon>0$ such that $\omega \in \Omega_\gamma$ and \eqref{little.gamma.p2} holds,
	there exists a $C^1$ function $	\tilde w: \mathcal{D}_{\rho_1} \rightarrow W \cap H^{\frac{1}{2} +\delta}_t \cH^{\frac 3 2 + \delta}_x$, $v_1 \mapsto \tilde w(v_1)$, where $\mathcal{D}_{\rho_1}$ is defined as in \eqref{v1.Domain}, satisfying $\tilde w (0) = 0$, and
	\begin{align}
	\label{wtilde.rho3}
	& \|\tilde w(v_1)\|_{H^{\frac 1 2 + \delta}_t \cH^{\frac 3 2 + \delta}_x} \leq \rho_3\,,\\
	& \label{de.w.de.v1.p2}
	\|d_{v_1}\tilde w(v_1)\|_{\mathcal{B}(V_1 \cap \cV^1_{t,x},H^{\frac{1}{2}+\delta}_t \cH^{\frac{3}{2}+\delta}_x)}\leq C_{\delta} \gamma^{-2} N^{3(1+2\delta)}\varepsilon R^2\,,
	\end{align}
	which solves equation \eqref{trenitalia}.
\end{proposition}
We show that the map
\begin{equation}\label{J.w}
	\begin{aligned}
	\tilde w \mapsto  	\mathcal{T}_{v_1}(\tilde w) &:= \mathcal{L}_\omega^{-1}
	\Pi_{W} \big( 2(v_1 + v_2(v_1, \tilde w)) \left(\mathcal{L}_\omega^{-1}(v_1 + v_2(v_1, \tilde w))^2 + \tilde w\right) \big) \\
	& \ + \mathcal{L}_\omega^{-1} \Pi_{W} \left( \left(\mathcal{L}_\omega^{-1}(v_1 + v_2(v_1, \tilde w))^2 + \tilde w\right)^2 \right)
	\end{aligned} 
\end{equation}
is a contraction.
\begin{lemma}[Contraction]\label{lem.w.contraz}
	The function $\mathcal{T}_{v_1}$ defined  in \eqref{J.w} maps $\mathcal{D}^W_{\rho_3}$ into itself, with
	\begin{equation}\label{diff.w}
	\| \mathcal{T}_{v_1}(\tilde w)  -  \mathcal{T}_{v_1}(\tilde w')\|_{H^{\frac 1 2 + \delta}_t \cH^{\frac 3 2 + \delta}_x} \leq C_{\delta} \gamma^{-2} \sqrt{\varepsilon} N^{1+6\delta} R^3 \|\tilde w - \tilde w '\|_{H^{\frac 1 2 + \delta}_t \cH^{\frac 3 2 + \delta}_x} \quad \forall \tilde w , \tilde w ' \in \mathcal{D}^W_{\rho_3}\,
	\end{equation}
	for some $C_\delta>0$. As a consequence, for any $v_1 \in \mathcal{D}_{\rho_1}$ there exists a unique $\tilde w(v_1)$ solving \eqref{trenitalia} and such that $\tilde w (0) = 0$.
\end{lemma}
\begin{proof}
	We denote $v_2:= v_2(v_1, \tilde w)$.
	By \eqref{algebra.prop}, Lemma \ref{lemma.sono.pochissimi}, Lemma \ref{lemma.come.vuoi}, Lemma \ref{lem.v2.leq.v1},  \eqref{scelta.param.p2.1} and the smallness condition \eqref{little.gamma.p2}, one has
	\begin{gather*}
	\begin{aligned}
	\left\|\mathcal{L}_\omega^{-1}\left(\Pi_{W} \left( 2(v_1 + v_2) \mathcal{L}_\omega^{-1}(v_1 + v_2)^2 \right)\right)\right\|_{H^{\frac 1 2 + \delta}_t \cH^{\frac 32 +\delta}_x} &\lesssim_{\delta} \gamma^{-2} \| v_1 + v_2\|^3_{\cV^{2+2\delta}_{t,x}} \lesssim_{\delta} \gamma^{-2} \varepsilon \sqrt{\varepsilon} N^{3(1+2\delta)} R^3\,,
	\end{aligned}	\\
	\begin{aligned}
	\left\|\mathcal{L}_\omega^{-1}\Pi_{W} \left( 2 (v_1 + v_2) \tilde w \right)\right\|_{H^{\frac 1 2 + \delta}_t \cH^{\frac 3 2 + \delta}_x} &\lesssim_{\delta} \gamma^{-1} N^{1+2\delta} \rho_1 \rho_3 \lesssim_{\delta} \gamma^{-3} N^{4+8\delta} {\varepsilon}^2 R^4 \,,
	\end{aligned}\\
	\begin{aligned}
	\left\| \mathcal{L}_\omega^{-1} \Pi_{W} \left( \left(\mathcal{L}_\omega^{-1}(v_1 + v_2)^2 + \tilde w\right)^2 \right)\right\|_{H^{\frac 1 2 + \delta}_t \cH^{\frac 3 2 + \delta}_x} 
	&\lesssim_{\delta} \gamma^{-1} \big( \gamma^{-1}( N^{1 + 2\delta} \rho_1)^2 + \rho_3 \big)^2 \lesssim_{\delta} \gamma^{-3} \varepsilon^2 R^4 N^{4(1+2\delta)}\,.
	\end{aligned}
	\end{gather*}
	Therefore there exists a constant $C_\delta>0$ such that
	$$
	\| \mathcal{T}_{v_1} (\tilde w)\|_{H^{\frac 12 + \delta}_t \cH^{\frac 3 2 + \delta}_x} \leq C_{\delta} \gamma^{-2} \varepsilon \sqrt{\varepsilon} N^{3(1+2\delta)} R^3 < \rho_3\,,
	$$
	where the last inequality holds taking $\tc_3=\tc_3(\delta)> C_\delta$ in the definition of $\rho_3$ (see \eqref{scelta.param.p2.1}). We now  show that $\mathcal{T}_{v_1}$ is a contraction. We have that $\partial_{\tilde w} \mathcal{T}_{v_1}(\tilde w)[h_3] = F_1[h_3] + F_2[h_3] + F_3[h_3]$, with
	$$
	\begin{gathered}
	F_1[h_3] :=\mathcal{L}_\omega^{-1} \Pi_W\left[2 \partial_{\tilde w} v_2[h_3] \left(\mathcal{L}_\omega^{-1}(v_1 + v_2)^2 + \tilde w \right)\right]\,,\\
	F_2[h_3] :=\mathcal{L}_\omega^{-1} \Pi_W\left[2(v_1 + v_2) \left(2\mathcal{L}_\omega^{-1}\left((v_1 + v_2) \partial_{\tilde w} v_2[h_3] \right) + h_3 \right)\right]\,,\\
	F_3[h_3] :=2\mathcal{L}_\omega^{-1} \Pi_W \left[\left(\mathcal{L}_\omega^{-1}(v_1 + v_2))^2 + \tilde w\right)\left(2\mathcal{L}_\omega^{-1}\left( (v_1 + v_2) \partial_{\tilde w} v_2[h_3] \right) + h_3\right) \right]\, . 
	\end{gathered}
	$$
	By Lemma \ref{lemma.sono.pochissimi}, algebra property \eqref{algebra.prop}, Lemma \ref{lem.v2.leq.v1}, Lemma \ref{lemma.come.vuoi}, estimate \eqref{de.v2.de.w.p2}, \eqref{scelta.param.p2.1} and the smallness condition \eqref{little.gamma.p2}, one gets
	\begin{align*}
	\| F_1[h_3]\|_{H^{\frac 1 2 + \delta}_t \cH^{\frac 3 2 + \delta}_x}
	&\lesssim_{ \delta} \gamma^{-1} \left\|\partial_{\tilde w} v_2[h_3] \right\|_{\cV^{2 +2\delta}_{t,x}} \big( 
	\gamma^{-1}\left\|v_1 + v_2 \right\|^2_{\cV^{2 +2\delta}_{t,x}} + \|\tilde w\|_{H^{\frac 1 2 + \delta}_t \cH^{\frac 3 2 + \delta}_x} \big) \\
	&\lesssim_{\delta} \gamma^{-2} \varepsilon^{\frac 12} N^{1 +6\delta} R^3 \|h_3\|_{H^{\frac 1 2 + \delta}_t \cH^{\frac 3 2 + \delta}_x}\,.
	\end{align*}
	Similarly $F_2 $ and $ F_3 $ satisfy the estimates
	\begin{align*}
	& \| F_2[h_3]\|_{H^{\frac 1 2 + \delta}_t \cH^{\frac 3 2 + \delta}_x}
	\lesssim_{\delta} \gamma^{-2} N^{1+6\delta} R^3 \varepsilon^{\frac 12} \|h_3\|_{H^{\frac 1 2 + \delta}_t \cH^{\frac 3 2 + \delta}_x}\,, \\
	& \|	F_3[h_3]\|_{H^{\frac 1 2 + \delta}_t \cH^{\frac 3 2 + \delta}_x}
	\lesssim_{\delta} \gamma^{-3} N^{2 + 8\delta} R^4 \varepsilon \|h_3 \|_{H^{\frac 1 2 + \delta}_t \cH^{\frac 3 2 + \delta}_x}\,.
	\end{align*}
	Combining the estimates on $F_1, F_2, F_3$ and assuming that condition \eqref{little.gamma.p2} holds for suitable parameters $\mathtt{b}$ and $\epsilon_{R, \delta}$, one gets
$$
	\| \partial_{\tilde w} \mathcal{T}_{v_1}(\tilde w) [h_3]\|_{H^{\frac 1 2 + \delta}_t \cH^{\frac 3 2 + \delta}_x} \lesssim_{\delta} \gamma^{-2} \sqrt{\varepsilon} N^{1+6\delta} R^3 \|h_3\|_{H^{\frac 1 2 + \delta}_t \cH^{\frac 3 2 + \delta}_x}\,,
$$
	which gives \eqref{diff.w}. Thus $\mathcal{T}_{v_1}$ is a contraction.
\end{proof}

The proof of \eqref{de.w.de.v1.p2} follows by similar arguments, using Lemma \ref{lemma.sono.pochissimi}, algebra property \eqref{algebra.prop}, Lemma \ref{lem.v2.leq.v1}, estimate \eqref{de.v2.de.v1.p2}, \eqref{scelta.param.p2.1} and the smallness condition \eqref{little.gamma.p2}.

\section{Solution of the bifurcation equation}\label{sec.bif}

In this section we solve
\begin{equation}\label{v1.eq.v2.w}
\begin{gathered}
(\omega^2 - 1) A v_1 - \Pi_{V_1} \left( v_1 + \mathtt{v}_2(v_1) + w(v_1)\right)^{p} = 0\,,\\
\mathtt{v}_2(v_1) := \begin{cases}
	v_2(v_1, w(v_1)) & \quad \text{if } p=3,\ p=5\,,\\
	v_2(v_1, \tilde w(v_1)) & \quad \text{if } p=2\,,
\end{cases}
\end{gathered}
\end{equation}
 where
 \begin{itemize}
 \item If $p=5$, $v_2(v_1, w(v_1))$ is the solution of \eqref{v2.eq}, whose existence has been proved in Proposition \ref{prop.v2} (resp. in Proposition \ref{prop.v2.p3}  if $p=3$), and $w(v_1)$ is the solution of \eqref{w2.eq.v2}, whose existence has been proved in Proposition \ref{prop.w} (resp. in Proposition \ref{prop.w.p3} if $p=3$);
 \item If $p=2$, $v_2(v_1, \tilde w(v_1))$ is the solution of \eqref{new.v2.2} as in Proposition \ref{prop.v2.p2}, $\tilde w(v_1)$ is the solution of \eqref{new.w.2}, as in Proposition \ref{prop.w.p2}, and (cfr.  \eqref{proviamo.traslando})
 \begin{equation}\label{therealw}
 	w(v_1) := \mathcal{L}_\omega^{-1}(v_1 + \mathtt{v}_2(v_1))^2
 	+ \tilde w(v_1) \,.
 \end{equation}
 \end{itemize}
 In all this section, we  suppose that $\rho_1, \rho_2, \rho_3, N$ and $\varepsilon$ satisfy the hypotheses of Sections \ref{sec.v2} and \ref{sec.w}.

 \subsection{Restricted Euler-Lagrange functional}
 
 We start  observing that \eqref{v1.eq.v2.w} has a variational structure.

 \begin{lemma}\label{lemma.ker.variational}
Equation \eqref{v1.eq.v2.w} is the Euler Lagrange equation 
of the restricted action functional
 	\begin{equation}\label{action.restriction}
 	\breve{\Psi}: \mathcal{D}_{\rho_1} \rightarrow \R\,, \quad v_1 \mapsto \breve{\Psi}(v_1):= \Psi(v_1 + \mathtt{v}_2(v_1) + w(v_1))\,,
 	\end{equation}
 	where $\Psi$ is the action functional defined in \eqref{az.funct}.
 	In particular, $u = v_1 + \mathtt{v}_2(v_1) + w(v_1)$ solves equation \eqref{v1.eq.v2.w} if and only if $v_1$ is a critical point of $\breve{\Psi}\,.$
 \end{lemma}
 
 \begin{proof}
 	We prove the result  for $p=5$ for definiteness.  
	In view of \eqref{az.funct}-\eqref{big.G}, the fact that $ - \partial_{tt} v_1 = A v_1 $, 
	and  $\mathtt{v}_2:= \mathtt{v}_2(v_1) \in \cV^{2 + 2\delta}_{t,x} $, resp. $w:= w(v_1) \in 
	H^{\frac 1 2 + \delta}_t \cH^{\frac 3 2 + \delta}_x $, solves equation \eqref{v2.di.v1.e.w}, resp. \eqref{w2.eq.v2}, one has
 	$$
 	\begin{aligned}
 	& \breve{\Psi}(v_1) = \frac{\omega^2 - 1}{2} \|v_1\|_{\cV^1_{t,z}}^2 + \frac{1}{2} \int_{\mathbb{T}} \int_{0}^{\pi} \mathtt{v}_2(v_1)  \Pi_{V_2} (v_1 + \mathtt{v}_2(v_1) + w(v_1))^5 \sin^2(x) \dbar x \dbar t\\
 	&  \!\! +  \!\frac{1}{2} \int_{\mathbb{T}} \int_{0}^{\pi} w(v_1) 
	 \Pi_{W} (v_1 + \mathtt{v}_2(v_1) + w(v_1))^5 \sin^2(x) \dbar x \dbar t \! - 
	 \! \frac 1 6 \int_{\mathbb{T}} \int_0^{\pi} \!\! (v_1 + \mathtt{v}_2(v_1) + w(v_1))^6 \sin^2(x) \dbar x \dbar t\,,
 	\end{aligned}
 	$$
 	and by \eqref{algebra.prop} 
	the functions 
	$\mathtt{v}_2(v_1) \Pi_{V_2} (v_1 + \mathtt{v}_2(v_1) + w(v_1))^5$, $w(v_1) 
	\Pi_{W} (v_1 + \mathtt{v}_2(v_1) + w(v_1))^5 $ and $(v_1 + \mathtt{v}_2(v_1) + w(v_1))^6$ are  in $H^{\frac 1 2 + \delta}_t \cH^{\frac 3 2 + \delta}_x \subset L^2_{t,x} \subset L^1_{t,x}$. Thus $\breve \Psi$ is well posed and differentiable 
	since by Propositions \ref{prop.v2} and \ref{prop.w} also $d_{v_1} \mathtt{v}_2[h_1]$
	and  $d_{v_1} w[h_1]$ are in $\HtHx$ for any $h_1 \in V_1$.
	Differentiating 
	equations \eqref{v2.di.v1.e.w} and \eqref{w2.eq.v2}, we get 
 	\begin{align}
 	\nonumber
 	& d \breve{\Psi}(v_1)[h_1] = \langle \mathcal{L}_\omega v_1, h_1 \rangle_{L^2_{t,x}} + \frac 1 2 \int_{\mathbb{T}} \int_{0}^{\pi} d \mathtt{v}_2[h_1] \mathcal{L}_\omega \mathtt{v}_2 \sin^2(x) \dbar x \dbar t\\
 	\nonumber
 	& \ + \frac 1 2 \int_{\mathbb{T}} \int_{0}^{\pi} \mathtt{v}_2 \mathcal{L}_\omega  \left( d \mathtt{v}_2[h_1]\right) \sin^2(x) \dbar x \dbar t + \frac 1 2 \int_{\mathbb{T}} \int_{0}^{\pi} d w [h_1] \mathcal{L}_\omega w \sin^2(x) \dbar x \dbar t\\
 	\nonumber
 	& \ + \frac 1 2 \int_{\mathbb{T}} \int_{0}^{\pi} w \mathcal{L}_\omega\left( d w[h_1] \right) \sin^2(x) \dbar x \dbar t - \int_{\mathbb{T}} \int_0^{\pi} (v_1 + \mathtt{v}_2 + w)^5 (h_1 + d \mathtt{v}_2[h] + d w[h]) \sin^2(x) \dbar x \dbar t\\
 	\label{sono.soluzioni}
 	& \ = \langle \mathcal{L}_\omega v_1, h_1 \rangle_{L^2_{t,x}} - \int_{\mathbb{T}} \int_0^{\pi} \Pi_{V_1}(v_1 + \mathtt{v}_2 + w)^5 h_1  \sin^2(x) \dbar x \dbar t  \notag 
 	\end{align}
so that  $d \breve{\Psi}(v_1)[h_1] = 0$ if and only if $v_1$ solves \eqref{v1.eq.v2.w}.
 \end{proof}
 
The following result ensures  by Lemma \ref{lemma.ker.variational} 
the existence of one solution of equations \eqref{main.sph.sym}-\eqref{main.hopf.1}.

\begin{theorem}[Existence of one critical point]\label{teo.v1}
	Let $\delta \in (0, \frac{1}{100})$. Then the following holds:\\
	\underline{Case $p=5, p=3$}: There exist  $R_0 >1$, $\epsilon_{\delta, R}>0$, $\zeta := \zeta(\delta)>1$ and $\beta_\delta>1$ such that, if $R>R_0$, \eqref{little.gamma} holds, $\beta > \beta_\delta$, $\rho_1, \rho_2, \rho_3, N$ are as in \eqref{parametri.p5} if $p=5$ and in \eqref{parametri.p3} if $p=3$,
		and $\omega >1$, $\omega \in \Omega_\gamma$,
		then the functional $\breve\Psi$ defined in \eqref{action.restriction} admits a critical point $v_{1}^{(1)} \in \mathcal{D}_{\rho_1}$ with norm
		\begin{equation}\label{buccia.R}
		\|v_{1}^{(1)}\|_{\cV^1_{t,z}} \asymp
		\begin{cases}
		\varepsilon^{\frac 1 4} & \text{if } p=5\\
		\varepsilon^{\frac 1 2} & \text{if } p=3\,.
		\end{cases}
		\end{equation} 
		\underline{Case $p=2$}: There exist $R_0 >1$, $\epsilon_{ \delta, R}>0$, $\mathtt{b}:=\mathtt{b}(\delta)>1$ and $\beta_\delta>1$ such that, if $R>R_0$, \eqref{little.gamma.p2} holds, $\beta > \beta_\delta$, $\rho_1, \rho_2, \rho_3, N$ are as in \eqref{scelta.param.p2.1},  and $\omega<1$, $\omega \in \Omega_\gamma$, then the functional $\breve\Psi$ defined in \eqref{action.restriction} admits a critical point $v_{1}^{(1)} \in \mathcal{D}_{\rho_1}$ with norm $\|v_{1}^{(1)}\|_{\cV^1_{t,x}} \asymp \varepsilon^{\frac 1 2} $ as $\varepsilon \rightarrow 0$.
\end{theorem}
In order to prove Theorem \ref{teo.v1} we first provide a suitable decomposition of the functional $\breve{\Psi}$, using Lemma \ref{lemma.ker.variational}. We argue separately for the cases $p=5, p=3$ and for the degenerate case $p=2$.
 \begin{lemma}[$\breve{\Psi}$ for $p=5,$ $p=3$]\label{lem.azione.ridotta}
 	The functional  $\breve \Psi$  defined in \eqref{action.restriction} 
	has the form
 	\begin{equation}
 	\label{all.together.now}
 	\breve\Psi(v_1) = \frac{\varepsilon}{2} \|v_1\|^2_{\cV^1_{t,z}} - \mathcal{G}_{p+1}(v_1) + \mathcal{R}_{p+1}(v_1) \quad \forall  v_1 \in \mathcal{D}_{\rho_1}\,,
 	\end{equation}
 	where $\mathcal{G}_{p+1}$ is defined in \eqref{big.G}, $\mathcal{R}_{p+1}(0) = 0$ and
 	\begin{equation}
 	\label{big.R}
 	d \mathcal{R}_{p+1}(v_1)[v_1] =
 	\begin{cases}
 	-\int_{\mathbb{T}} \int_0^\pi \big( (v_1 + \mathtt{v}_2(v_1) + w(v_1))^{5} - v_1^5 \big) \sin^2(x) \dbar x \dbar t & \quad \text{if } p=5\\
 	-\int_{\mathbb{T}} \int_0^{\frac \pi 2} \big( (v_1 + \mathtt{v}_2(v_1) + w(v_1))^{3} - v_1^3 \big) \sin(2\eta ) d \eta \dbar t & \quad \text{if } p=3\,.
 	\end{cases}
 	\end{equation}
 	If $p=5$ the functions $\mathtt{v}_2(v_1)$ and $w(v_1)$ satisfy
 	\begin{gather}
 	\label{primavera}
 	\|\mathtt{v}_2(v_1)\|_{\cV^{2+2\delta}_{t,x}} \lesssim_{\delta} N^{10\delta} R^4 \|v_1\|_{\cV^1_{t,x}}\,, \quad \|w(v_1)\|_{\HtHx} \lesssim_{\delta} \varepsilon \gamma^{-1} N^{5+10\delta}  R^4 {\|v_1\|_{\cV^1_{t,z}}}\,,
 	\end{gather}
 	whereas if $p=3$
 	\begin{gather}
 	\label{piccolo.piccolo.p3}
 	\|  \mathtt{v}_2(v_1) \|_{\cV^{2+2\delta}_{t,\eta}} \lesssim_\delta N^{4\delta} R^2 {\|v_1\|_{\cV^{1}_{t,\eta}}}\,, \quad 
 	{\|w(v_1)\|_{\HtHeta}} \lesssim_{\delta} \varepsilon \gamma^{-1} N^{3 + 6\delta} R^2 {\|v_1\|_{\cV^{1}_{t,\eta}}}\,.
 	\end{gather}
 \end{lemma}
 Here the estimates \eqref{primavera}, \eqref{piccolo.piccolo.p3} follow by Propositions \ref{prop.v2}, \ref{prop.w}, \ref{prop.v2.p3}, \ref{prop.w.p3}.
 
The case $p=2$ is different 
since Lemma \ref{lem.vsquare} implies that  $ \int_{\mathbb{T}} \int_{0}^\pi v_1^3 \sin^2(x)\, \dbar x\, \dbar t = 0 $
 identically vanishes.  We perform a different decomposition.
 \begin{lemma}[$\breve{\Psi}$ in the case $p=2$]
 	For any $ v_1 \in \mathcal{D}_{\rho_1}$ one has
 	\begin{gather}
 	\label{all.together.now.p2}
 	\breve\Psi(v_1) = -\frac{\varepsilon}{2}  \|v_1\|^2_{\cV^1_{t,x}} - \breve{\mathcal{G}_{4}}(v_1) + \breve{\mathcal{R}}_4(v_1)\,,\\
 	\label{big.G2}
 	\breve{\mathcal{G}_{4}}(v_1) := \frac{1}{2}\int_{\mathbb{T}}\int_{0}^\pi v_1^2 \mathcal{L}_1^{-1} v_1^2\sin^2(x)\,\dbar x\dbar t\,,
 	\end{gather}
 	with $\breve{\mathcal{R}}_4(0) = 0$ and
 	\begin{equation}
 		\label{big.R2}
 	\begin{aligned}
 	 d \breve{\mathcal{R}}_4(v_1)[v_1] &= 2 \int_{\mathbb{T}} \int_0^\pi v_1^2 (\mathcal{L}_1^{-1} - \mathcal{L}_\omega^{-1}) v_1^2 \sin^2(x) \dbar x \dbar t\\
 	 &- \int_{\mathbb{T}} \int_0^\pi ( 2 v_1 \mathtt{v}_2(v_1) w(v_1) + w(v_1)^2 v_1) \sin^2(x) \dbar x \dbar t\\
 	 &- 2 \int_{\mathbb{T}} \int_0^\pi v_1^2 \left(w(v_1) - \mathcal{L}_\omega^{-1} v_1^2\right)\sin^2(x) \dbar x \dbar t\,.
 	\end{aligned}
 	\end{equation}
 \end{lemma}

We prove Theorem \ref{teo.v1} as an application of the following abstract result, which is a particular case of Theorem 2.3 of \cite{Berti-BolleNA}:

\begin{theorem}[Abstract mountain pass theorem]\label{mountain.bolla}
Let $E$  be a finite dimensional Hilbert space  equipped with scalar product $\langle \cdot, \cdot \rangle$ and norm $\| \cdot \|^2 = \langle \cdot, \cdot \rangle$. 
Let $\mathcal{I}: B_{\rho_1} \subset E \rightarrow \mathbb{R}$ a $C^1$ functional defined on the ball $B_{\rho_1}:=\{ v \in E\ :\ \|v\| < \rho_1\}$ for some $\rho_1>0$, of the form
 \begin{equation}\label{i.astratto}
 	\mathcal{I}(v) = \frac{\varepsilon}{2} \|v\|^2 - \mathcal{G}(v) + \mathcal{R}(v) \, , 
 \end{equation}
 where $ \mathcal{G} \in C^1(E,\R) $ 
 is a non zero 
 homogeneous functional of degree $p+1$, $p>1$, 
 and $\mathcal{R} \in C^1(B_{\rho_1},\R) $ satisfies $\mathcal{R} (0) = 0 $.
	Define
	\begin{equation}\label{def.m.alpha}
	m^+(\mathcal{G}) := \sup_{ v \neq 0 } \frac{\mathcal{G}(v)}{\|v\|^{p+1}}\,, \quad m^-(\mathcal{G}) := \inf_{ v \neq 0 } \frac{\mathcal{G}(v)}{\|v\|^{p+1}}
	\,,\quad m(\mathcal{G}) := 
\begin{cases} 
\ \ m^+(\mathcal{G})  \quad \text{if} \ m^+(\mathcal{G})>0 \\
 -m^{-}(\mathcal{G}) \quad \text{if} \ m^-(\mathcal{G})<0 \, ,  
\end{cases}
	\end{equation}
and suppose $ \varepsilon>0 $ (resp. $\varepsilon<0$) if $ m^+(\mathcal{G})>0$
(resp. $m^-(\mathcal{G})<0$).

 Then there exists a positive constant $\underline{C}$, depending on {$p$} only, such that, if
	\begin{equation}
	\label{alpha.m}
	\alpha(\mathcal{R}) := \sup_{v \in {B_{\rho_1} \setminus \{0\}}} \frac{| d\mathcal{R}(v)[v] |}{\|v\|^{p+1}} \leq \underline{C} m(\mathcal{G})\,, \quad \left(\frac{|\varepsilon|}{m(\mathcal{G})}\right)^{\frac{1}{p-1}} \leq \underline{C} \, \rho_1\,,
	\end{equation}
	the functional  $\mathcal{I}$ has a critical point $v \in B_{\rho_1}$ on a critical level
	\begin{equation}
	c := \frac{p-1}{2} m(\mathcal{G}) \left(\frac{|\varepsilon|}{(p+1)m(\mathcal{G})}\right)^{(p+1)/(p-1)} \left(1 + \mathcal{O}\left(\frac{\alpha(\mathcal{R})}{m(\mathcal{G})}\right)\right)\,.
	\end{equation}
	Moreover
	\begin{equation}\label{ecco.y}
	v = \left(1 + \mathcal{O}\left(\frac{\alpha(\mathcal{R})}{m(\mathcal{G})}\right)\right) \left(\frac{|\varepsilon|}{(p+1)m(\mathcal{G})}\right)^{1/(p-1)} y\,,
	\end{equation}
	for some $y \in E$ with $\|y\| = 1$ and $\mathcal{G}(y) = m^+(\mathcal{G}) + \mathcal{O}(\alpha(\mathcal{R}))$ (resp., $\mathcal{G}(y) = m^-(\mathcal{G}) + \mathcal{O}(\alpha(\mathcal{R}))$).
\end{theorem} 
We shall apply Theorem \ref{mountain.bolla} to $\breve{\Psi}$ in \eqref{action.restriction} with $E = (V_1,\langle \cdot, \cdot \rangle_{\cV^1_{t,z}})$. 

\subsection{Cases $p=5$ and $p=3$}
We first consider the case $p=5$.
Since the functional $\mathcal{G}_6 $ in \eqref{big.G} is positive, we  have $m(\mathcal{G}_6) = m^+(\mathcal{G}_6)$.
\begin{lemma}[Estimate of $m(\mathcal{G}_6)$]\label{lem.grad.G}
	 There exists $C>0$, independent of $N$, such that
	\begin{equation}\label{m.above}
	\frac{5}{48}\leq m(\mathcal{G}_6) \leq C\,.
	\end{equation}
\end{lemma}

\begin{proof}
	By Item 1 of Proposition \ref{lem.Strichartz.1} with $\delta = \frac 1 6$, one has
	$\mathcal{G}(v_1) \lesssim_{} \|v_1\|^6_{\cV^{1}_{t,x}}$,
	which gives $m(\mathcal{G}_6)\leq C$ for some $C >0$.
	We now estimate $m(\mathcal{G}_6)$ from below.
	Let $\check{v}_1(t,x) := \cos(t) e_0(x) = \cos(t)$,
	then $\check v_1 \in V_1$, and by \eqref{hs.ker} $\| \check v_1\|^2_{\cV^1_{t,x}} =1\,.$
	One has $\mathcal{G}_6(\check{v}_1) = {\frac{1}{6}} \int_{\mathbb{T}} \cos^6(t) \dbar t\, \int_{0}^\pi \sin^2(x)\, \dbar x = \frac{5}{48}$,
	since
	$\int_{\mathbb{T}} \cos^6(t) \dbar t = \frac 5 8$, and the estimate \eqref{m.above} is proved.
\end{proof}

We now prove properties of the functional $\mathcal{R}_6$ defined in Lemma \ref{lem.azione.ridotta}.

\begin{lemma}[Estimate of $\alpha(\mathcal{R}_6)$]\label{lem.R.grad}
	There exists $C_\delta>0$ such that
	$\alpha(\mathcal{R}_6) \leq C_{\delta} N^{-\frac{7}{6} + 9\delta} R^4$.
\end{lemma}
\begin{proof}
 	We set $\mathtt{v}_2 := \mathtt{v}_2(v_1)$ and $w:= w(v_1)$. By \eqref{big.R}, one has
	\begin{equation}\label{grad.grad}
	\begin{aligned}
	\left| d\mathcal{R}_6(v_1)[v_1] \right| & \lesssim \sum_{j_1 + j_2 = 5 \atop j_2 \geq 1}   \left|\int_\mathbb{T}\int_{0}^{\pi}  {v_1^{j_1+1} \mathtt{v}_2^{j_2}}\sin^2(x)\,\dbar x\dbar t\right|\\
	& + \sum_{j_1 + j_2 + j_3 = 5  \atop j_3 \geq 1} \left|\int_\mathbb{T}\int_{0}^{\pi} {v_1^{j_1+1} \mathtt{v}_2^{j_2} w^{j_3}} \sin^2(x)\,\dbar x\dbar t\right|\,.
	\end{aligned}
	\end{equation}
	Using \eqref{item1st1} for the first term one has, for any $j_1, j_2$ with $j_1 + j_2 = 5$ and $j_2 \geq 1$,
	\begin{equation}\label{stimette.1}
	\begin{aligned}
	\left|\int_\mathbb{T}\int_{0}^{\pi} {v_1^{j_1+1} \mathtt{v}_2^{j_2}}\sin^2(x)\,\dbar x\dbar t\right| &\lesssim_\delta \|v_1\|_{\cV^{\frac 5 6 + \delta}_{t,x}}^{j_1 + 1} \|\mathtt{v}_2\|_{\cV^{\frac 5 6 + \delta}_{t,x}}^{j_2} \stackrel{\eqref{lem.s.sprime}}{\lesssim_\delta} N^{-(\frac 7 6 + \delta)j_2}\|v_1\|_{\cV^1_{t,x}}^{j_1+1} \|\mathtt{v}_2\|_{\cV^{2+2\delta}_{t,x}}^{j_2}\\
	&\stackrel{\eqref{primavera}}{\lesssim_{ \delta}} N^{-(\frac 7 6 + \delta)j_2} (N^{10\delta} R^4)^{j_2} \|v_1\|^6_{\cV^1_{t,x}} \lesssim_{ \delta} N^{-\frac 7 6 + 9 \delta} R^4 \|v_1\|^6_{\cV^1_{t,x}} \,,
	\end{aligned}
	\end{equation}
	since $N^{-\frac 7 6 + 9 \delta} R^4 < 1$, due to the smallness condition in \eqref{little.gamma}. The second term is estimated using Lemma \ref{algebraKer} and recalling $j_1 + j_2 + j_3 = 5$, $j_3 \geq 1$, one obtains
	\begin{align}
	\nonumber
	\left|\int_\mathbb{T}\int_{0}^{\pi} \Pi_{V_1}({v_1^{j_1} \mathtt{v}_2^{j_2} w^{j_3}}){v_1}\sin^2(x)\,\dbar x \dbar t\right|
	&\lesssim_\delta \| \Pi_{V_1} (v_1^{j_1} \mathtt{v}_2^{j_2} w^{j_3})\|_{\cV^{0}_{t,x}} \|v_1\|_{\cV^0_{t,x}}\\
	\nonumber
	&\lesssim_\delta\|v_1\|^{j_1}_{\cV^{\frac 3 2 +\delta}_{t,x}} \|\mathtt{v}_2\|^{j_2}_{\cV^{\frac 3 2 +\delta}_{t,x}} \|w\|^{j_3}_{H^{\frac 1 2 + \delta}_{t} \cH^{\frac 3 2 +\delta}_x} \|v_1\|_{\cV^{1}_{t,x}}\\
	\nonumber
	&\stackrel{\eqref{lem.s.sprime}}{\lesssim_\delta} N^{j_1(\frac 1 2 + \delta)} \|v_1\|^{j_1+1}_{\cV^1_{t,x}} \|\mathtt{v}_2\|^{j_2}_{\cV^{2+2\delta}_{t,x}} \|w\|^{j_3}_{H^{\frac 1 2 + \delta}_{t} \cH^{\frac 3 2 +\delta}_x}\\
	\label{stimette.2}
	&\stackrel{\eqref{primavera}, \eqref{little.gamma}}{\lesssim_{ \delta}}  N^{4(\frac 1 2 + \delta)+ 5 + 10\delta} R^{4} \gamma^{-1} \varepsilon  \|v_1\|^{6}_{\cV^1_{t,x}}\,,
	\end{align}
which is ${\lesssim_{ \delta}} N^{-\frac 7 6 + 9 \delta} R^4 \|v_1\|^6_{\cV^1_{t,x}}$, using again the smallness condition \eqref{little.gamma}.
	Estimates \eqref{stimette.1} and \eqref{stimette.2} imply
	$
	\left| d\mathcal{R}_6(v_1)[ v_1] \right| \lesssim_{\delta} N^{-\frac{7}{6} + 9\delta} R^4 {\|v_1\|^6_{\cV^1_{t,x}}}\,,
	$
	which gives the thesis. 
\end{proof}

\begin{proof}[Proof of Theorem \ref{teo.v1} for $p=5$]
 We now verify that the two conditions  in \eqref{alpha.m} are satisfied.
 Let $\underline{C}$ be the positive constant defined in Theorem \ref{mountain.bolla}: by  \eqref{m.above} and Lemma \ref{lem.R.grad}
 \begin{equation}\label{trentatre.trentini}
 \frac{\alpha(\mathcal{R}_6)}{m(\mathcal{G}_6)} \leq \frac{C_\delta N^{-\frac{7}{6} + 9\delta} R^4}{\frac{5}{48}} \leq \underline{C}\,,
 \end{equation}
 provided $N^{\frac 7 6 - 9 \delta} \geq \frac{ 48 C_\delta R^4 }{5 \underline{C}}$. Furthermore, by \eqref{m.above}, one has
 $$
 \left(\frac{\varepsilon}{m(\mathcal{G}_6)}\right)^{\frac 1 4} \leq \varepsilon^{\frac 1 4} \left(\frac{48}{5}\right)^{\frac 1 4} \leq \underline{C} \rho_1 \stackrel{\eqref{parametri.p5}}{=} \underline{C} R \varepsilon^{\frac{1}{4}}\,,
 $$
 provided $R \geq (\frac{48}{5})^{\frac 1 4} \underline{C}^{-1}$. Then Theorem \ref{mountain.bolla} ensures the existence of a critical point $v^{(1)}_{1} \in  V_1$. Finally, \eqref{m.above}, \eqref{trentatre.trentini} and $N := \varepsilon^{-\frac 1 \beta}$ imply $\frac{\alpha(\mathcal{R}_6)}{m(\mathcal{G}_6)} = o(1)$ as $\varepsilon \rightarrow 0$. Then equation \eqref{ecco.y} of Theorem \ref{mountain.bolla} implies $\| v^{(1)}_{1}\|_{\cV^1_{t,x}} \asymp \varepsilon^{\frac 1 4}$ as $\varepsilon \rightarrow 0$.
\end{proof}

We now consider the case $p=3$.
Since $\mathcal{G}_4$ defined in \eqref{big.G} is positive, we have $m(\mathcal{G}_4) = m^+(\mathcal{G}_4)$.
\begin{lemma}[Estimate of $m(\mathcal{G}_4)$]\label{lem.grad.G.p3}
	There exist $C^+>0$ and $C^-_{\mu_1, \mu_2}>0$, independent of $N$, such that
	\begin{equation}\label{m.p3}
	C^-_{\mu_1, \mu_2} \leq m(\mathcal{G}_4) \leq C^+\,.
	\end{equation}
\end{lemma}

\begin{proof}
	By \eqref{int4estimate.0} and by Lemma \ref{lemma.come.vuoi}, for any $v_1 \in V_1$ one has
	$$
	\mathcal{G}_4(v_1) = \frac 1 4 \int_{\mathbb{T}} \int_{0}^{\frac \pi 2} v_1^4 \sin(2\eta) d \eta \dbar t  \lesssim \|v_1\|^4_{\cV^1_{t,\eta}}\,,
	$$ 
	which gives the upper bound in \eqref{m.p3}. The lower bound follows since there exists $\check{v}_1$ such that $\mathcal{G}_4(\check{v}_1) >0$. For example, $\check{v}_1(t,\eta) := \cos( \omega_0 t) \egen_0(\eta)$,
	with $\egen_0 = e_0^{(\mu_1, \mu_2)}$ as in \eqref{ejmu}.
\end{proof}

We now  estimate the functional $\mathcal{R}_4$ defined in Lemma \ref{lem.azione.ridotta}:

\begin{lemma}[Estimate of $\alpha(\mathcal{R}_4)$]\label{lem.R.grad.p3} There exists $C_\delta>0$ such that
	$\alpha(\mathcal{R}_4) \leq C_\delta N^{-\frac 5 4 +2\delta} R^2$.
\end{lemma}

\begin{proof}
	We set $\mathtt{v}_2 := \mathtt{v}_2(v_1)$ and $w:= w(v_1)$. By \eqref{big.R}, one has
	\begin{align*}
	\left|d\mathcal{R}_4(v_1)[v_1]\right| & \lesssim \sum_{j_1 + j_2 = 3 \atop j_2 \geq 1} \left|\int_\mathbb{T}\int_{0}^{\frac{\pi}{2}} {v_1^{j_1 + 1} \mathtt{v}_2^{j_2}}\sin(2\eta)\,d \eta \dbar t\right|
	+ \sum_{j_1 + j_2 + j_3 = 3 \atop j_3 \geq 1} \left|\int_\mathbb{T}\int_{0}^{\frac{\pi}{2}} {v_1^{j_1 + 1} \mathtt{v}_2^{j_2} w^{j_3}}\sin(2\eta)\,d \eta\dbar t\right|
	\end{align*}
	and we estimate the two terms separately. Using Lemma \ref{lem.lp.hs.easy}, Lemma \ref{lemma.come.vuoi}, \eqref{lem.s.sprime}, \eqref{piccolo.piccolo.p3} and \eqref{little.gamma.p3}, and recalling $j_1 +j_2 = 3$, $j_2 \geq 1$, for the first term one has
	\begin{equation}\label{stimette.1.p3}
	\begin{aligned}
	\Big|\int_\mathbb{T}\int_{0}^{\frac{\pi}{2}} {v_1^{j_1 + 1}  \mathtt{v}_2^{j_2}}\sin(2\eta)\,d \eta\dbar t\Big| \lesssim_\delta N^{(-\frac 5 4 +2 \delta)j_2} R^{2 j_2} \|v_1\|^4_{\cV^1_{t,\eta}} \lesssim_{ \delta} N^{-\frac 5 4 + 2\delta} R^2 \|v_1\|^4_{\cV^1_{t,\eta}}\,.
	\end{aligned}
	\end{equation}
	For the second term
	using Lemma \ref{algebraKer},\eqref{lem.s.sprime}, \eqref{piccolo.piccolo.p3}, \eqref{little.gamma.p3} and recalling $j_1 + j_2 + j_3 = 3$ and $j_3 \geq 1$, one obtains
	\begin{align}
	\label{stimette.2.p3}
	\Big|\int_\mathbb{T}\int_{0}^{\frac{\pi}{2}}  \Pi_{V_1}({v_1^{j_1}  \mathtt{v}_2^{j_2} w^{j_3}}) v_1\sin(2 \eta)\,d \eta \dbar t\Big| \lesssim_{ \delta}
 \varepsilon \gamma^{-1} R^2 N^{4 + 8\delta} \|v_1\|^{4}_{\cV^1_{t,\eta}} \lesssim_{ \delta}  N^{-\frac 5 4 + 2\delta} R^2 \|v_1\|^4_{\cV^1_{t,\eta}}\,.
	\end{align}
	Combining estimates \eqref{stimette.1.p3} and \eqref{stimette.2.p3}, one then gets Lemma \ref{lem.R.grad.p3}.
\end{proof}
\begin{proof}[Proof of Theorem \ref{teo.v1} for $p=3$]
	By \eqref{m.p3} and Lemma \ref{lem.R.grad.p3}, one has
	$$
	\frac{\alpha(\mathcal{R}_4)}{m(\mathcal{G}_4)} \leq \frac{C_\delta N^{-\frac 5 4 + 2\delta} R^2}{C^-_{\mu_1, \mu_2}} \leq \underline{C}\,,
	$$
	with $\underline{C}$ the constant whose existence is stated in Theorem \ref{teo.v1}, provided \eqref{little.gamma.p3} holds with $\epsilon_{R, \delta}$ small enough. By \eqref{m.p3} one observes that
	$$
	\left(\frac{\varepsilon}{m(\mathcal{G}_4)}\right)^{\frac 1 2} \leq \varepsilon^{\frac 1 2} (C^-_{\mu_1, \mu_2})^{-\frac 1 2} \leq \underline{C} \rho_1 \stackrel{\eqref{parametri.p3}}{=} \underline{C} R \varepsilon^{\frac 1 2}\,,
	$$
	provided $R \geq (C^-_{\mu_1, \mu_2})^{-\frac 1 2} \underline{C}^{-1}$. Then the existence of a critical point $v^{(1)}_{1} \in \cV^1_{t,\eta}$ with $\|v^{(1)}_{1}\|_{\cV^1_{t,\eta}} \asymp \varepsilon^{\frac 1 2}$ as $\varepsilon \rightarrow 0$ follows.
\end{proof}

\subsection{Case $p=2$}

 In the next lemma we show that $\breve{\mathcal{G}_{4}}$ in \eqref{big.G2}
 assumes also negative values. 
 Thus $m(\breve{\mathcal{G}_{4}}) = -m^-(\breve{\mathcal{G}_{4}})$.
 
\begin{lemma}[Estimate of $m(\breve{\mathcal{G}_{4}})$]\label{lem.m.G.p2}
	There exists $C>0$, independent of $N$, such that
	\begin{equation}\label{m.p2}
		\frac{5}{24} \leq m(\breve{\mathcal{G}_{4}}) \leq C\,.
	\end{equation}
\end{lemma}

\begin{proof}
	By Item 1 of Proposition \ref{lem.strich.Lom} with $\delta = \frac 1 2$, for any $v_1 \in V_1$ one has
	$$
	-\breve{\mathcal{G}_{4}}(v_1) \leq \frac 1 2 \left| \int_{\mathbb{T}} \int_0^\pi v_1^2 \mathcal{L}^{-1}_1 v_1^2 \sin^2(x) \dbar x \dbar t \right| \leq  C \|v_1\|^4_{\cV^1_{t,x}}\,
	$$
	for some $C>0$, which gives the second inequality in \eqref{m.p2}. We now prove the first inequality in \eqref{m.p2}.
	Let $\bar{v}_1(t,x):= \cos(t) e_0(x) = \cos(t)$, with $e_0 = 1$ by \eqref{def.ej}. Then $\bar{v}_1 \in V_1$ and $\|\bar{v}_1\|_{\cV^1_{t,x}} = 1$. We now compute $\breve{\mathcal{G}_{4}}(\bar{v})$. Recalling \eqref{L.om.fourier}, one has
	\begin{align*}
	\int_{\mathbb{T}}\int_{0}^\pi \bar{v}_1^2 \mathcal{L}_1^{-1}\bar{v}_1^2 \sin^2(x)\,\dbar x \dbar t &= 	\int_{\mathbb{T}}\int_{0}^\pi \big(\tfrac 1 2 \cos(2t) + \tfrac 1 2 \big) 
	\big( \tfrac{1}{6} \cos(2t) - \tfrac{1}{2}\big)\sin^2(x)\,\dbar x \dbar t = -\tfrac{5}{12}\,,
	\end{align*}
	thus $m(\breve{\mathcal{G}_{4}}) = -\inf_{v_1 \in V_1 \setminus\{0\}} \frac{\breve{\mathcal{G}_{4}}(v_1)}{\|v_1\|_{\cV^1_{t,x}}^4} \geq  -{\breve{\mathcal{G}_{4}}(\bar{v})} > \frac{5}{24}$.
\end{proof}

We now estimate $\alpha(\breve{\mathcal{R}}_4)$ where $\breve{\mathcal{R}}_4$ is defined in  \eqref{big.R2}. We use the following auxiliary lemmas.

\begin{lemma}\label{cor.back.to.w}
	 There exists $C_\delta>0$ such that
	\begin{align}
	\label{di.v2.p2}
	& \|{\mathtt{v}}_2(v_1)\|_{\cV^{2+2\delta}_{t,x}} \lesssim_{\delta} \gamma^{-1} R^2 \|v_1\|_{\cV^1_{t,x}}\,,\\
	& \label{summarizing.12}
	\|w(v_1)\|_{H^{\frac 1 2 + \delta}_{t}\cH^{\frac 3 2 + \delta}_x} \leq C_{\delta} \gamma^{-1}  N^{2(1+2\delta)} \sqrt{\varepsilon} R \|v_1\|_{\cV^1_{t,x}}\,, \\
	& \label{lem.wtilde.order.2}
	\| \tilde w(v_1) \|_{H^{\frac 1 2 + \delta}_t \cH^{\frac 3 2 + \delta}_x} \leq C_{\delta} \gamma^{-2}  R N^{3(1+2\delta)} \sqrt{\varepsilon} \|v_1\|^2_{\cV^1_{t,x}}\,.
	\end{align}
\end{lemma}

\begin{proof}
	By chain rule we have $ d_{v_1} {\mathtt v}_2(v_1)[h_1]=(\partial_{v_1}v_2)(v_1, \tilde w(v_1))[h_1]+(\partial_w v_2)(v_1, \tilde w(v_1))\left[\partial_{v_1}\tilde w(v_1)[h_1]\right]$, hence \eqref{di.v2.p2} follows by \eqref{de.v2.de.v1.p2}, \eqref{de.v2.de.w.p2}, \eqref{de.w.de.v1.p2} and \eqref{little.gamma.p2} and recalling that $\mathtt{v}_2(0) = 0$.
	
	Concerning $w(v_1)$, we have $w(0) = 0$ since $v_2(0, 0) = 0$ and $\tilde w (0) = 0$.
	For any $ h_1 \in V_1$, by algebra property \eqref{algebra.prop}, by Lemma \ref{lemma.sono.pochissimi}, Lemma \ref{lem.v2.leq.v1}, \eqref{di.v2.p2}, \eqref{de.w.de.v1.p2}, \eqref{scelta.param.p2.1} and \eqref{little.gamma.p2} one has
	\begin{equation*}
	\begin{aligned}
	\|d_{v_1} w(v_1)[h_1]\|_{\HtHx} &\leq  \|2 \mathcal{L}_\omega^{-1} \left((v_1 + {\mathtt{v}}_2)(h_1 + d_{v_1}{\mathtt{v}}_2[h_1])\right)\|_{\HtHx} + \|d_{v_1} \tilde w[h_1]\|_{\HtHx}\\
	&\lesssim_{\delta} \gamma^{-1} N^{1+2\delta} \rho_1 (N^{1+2\delta} + \gamma^{-1} R^2 ) \|h_1\|_{\cV^{1}_{t,x}} + \gamma^{-2} N^{3(1+2\delta)}\varepsilon R^2 \|h_1\|_{\cV^{1}_{t,x}}\\
	&\lesssim_{\delta} \gamma^{-1} N^{2(1+2\delta)} R \sqrt{\varepsilon} \|h_1\|_{\cV^{1}_{t,x}} \,,
	\end{aligned}
	\end{equation*}
	which implies \eqref{summarizing.12}, since $w(0) = 0$.
	
	We now prove \eqref{lem.wtilde.order.2}. First we observe that, by  \eqref{trenitalia}, $\tilde w(v_1)$ satisfies
	$$
	\tilde w(v_1) = \mathcal{L}_\omega^{-1} \Pi_W (2 (v_1 + \mathtt{v}_2) w+ w^2) = \mathcal{L}_\omega^{-1} \Pi_W ((v_1 + \mathtt{v}_2 + w)^2 - (v_1 + \mathtt{v}_2)^2 )\,.
	$$
 Then by algebra property \eqref{algebra.prop} and by Lemma \ref{lemma.sono.pochissimi}  we have 
	\begin{align*}
	\| \tilde w (v_1)\|_{H^{\frac 1 2 + \delta}_t \cH^{\frac 3 2 + \delta}_x}
	&\lesssim_{\delta} \gamma^{-1} \sum_{j_1 + j_2 + j_3 = 2 \atop j_3 \geq 1} \left(N^{(1+2\delta)}\| v_1\|_{\cV^{1}_{t,x}} \right)^{j_1} \| {\mathtt{v}}_2(v_1)\|^{j_2}_{\cV^{2+2\delta}_{t,x}} \| w(v_1)\|^{j_3}_{H^{\frac 1 2 + \delta}_t \cH^{\frac 3 2 + \delta}_x}\,.
	\end{align*}
	Now \eqref{lem.wtilde.order.2} follows by \eqref{di.v2.p2} and \eqref{summarizing.12}.
\end{proof}
 
\begin{lemma}\label{lemma.lom-lun}
Assume $\omega \in \Omega_\gamma$.
	For any $w \in W \cap H^{r + 1}_t \cH^{s}_x $, $r, s \in \R$,  one has
	\begin{equation}\label{lom1sm}
	\left\|\left(\mathcal{L}_\omega^{-1} - \mathcal{L}_1^{-1} \right) w \right\|_{H^{r}_t \cH^{s}_x} \leq 2 \gamma^{-1} \varepsilon \|w\|_{H^{r + 1}_t \cH^{s}_x}\,.
	\end{equation}
\end{lemma}

\begin{proof}
By \eqref{L.om.fourier} we have 
	\begin{equation}\label{vittor}
	\left(\mathcal{L}_\omega^{-1} - \mathcal{L}_1^{-1}\right) w(t,x) = \sum_{\ell \neq \omega_j, \ell \neq 0} w_{\ell,j}\Big(\frac{1}{\omega^2 \ell^2 -\omega_j^2} - \frac{1}{ \ell^2 -\omega_j^2}\Big) \cos(\ell t) e_j(x)\,,
	\end{equation}
	and
	\begin{equation}\label{pisani}
	\Big|\frac{1}{\omega^2 \ell^2- \omega_j^2} - \frac{1}{ \ell^2 -\omega_j^2} \Big| 
	= \Big| \frac{\ell^2(1 -\omega^2) }{(\omega^2 \ell^2 -\omega_j^2)( \ell^2 -\omega_j^2)} \Big| \leq \frac{|\ell|^2 \varepsilon}{|\ell| \frac{\gamma}{2}} = \frac{2 |\ell| \varepsilon}{\gamma}\,,
	\end{equation}
	using \eqref{approx.bded}
	and $|\ell^2 - \omega_j^2| \geq |\ell + \omega_j| \geq |\ell|$.
	Combining \eqref{vittor} and \eqref{pisani}, one deduces \eqref{lom1sm}.
\end{proof}

We now exhibit an upper bound for $\alpha(\breve{\mathcal{R}}_4)$.

\begin{lemma}[Estimate of $\alpha(\breve{\mathcal{R}}_4)$ ]\label{lemma.aR.small.p2}
	There exists $C_\delta>0$ such that $\alpha(\breve{\mathcal{R}}_4) \leq C_{\delta} \gamma^{-2} N^{-\frac 3 2 -\delta} R^2$.
\end{lemma}

\begin{proof}
	 For brevity, we denote $\mathtt{v}_2 := \mathtt{v}_2(v_1)$, $w := w(v_1)$, $\tilde w := \tilde w(v_1)$. By \eqref{big.R2} and \eqref{therealw}, we have
	$ d \breve{\mathcal{R}}_4(v_1)[v_1] = A(v_1) + B(v_1) + C(v_1) + D(v_1)$,
	where
	$$
	\begin{aligned}
	&A(v_1)  := -2 \int_{\mathbb{T}} \int_0^\pi \Pi_{V_1}\left((v_1 + \mathtt{v}_2) \tilde w \right) v_1 \sin^2(x) \dbar x \dbar t \,, \\
	&B(v_1)
	=  \underset{:=B_1(v_1)}{\underbrace{-2 \int_{\mathbb{T}} \int_0^\pi v_1^2 \mathcal{L}_\omega^{-1} \left(2 v_1 \mathtt{v}_2 + \mathtt{v}_2^2\right)\sin^2(x) \dbar x \dbar t }}  \underset{:=B_2(v_1)}{\underbrace{-2 \int_{\mathbb{T}} \int_{0}^\pi v_1 \mathtt{v}_2 \mathcal{L}_\omega^{-1}(v_1 + \mathtt{v}_2)^2 \sin^2(x) \dbar x \dbar t}}  \,,\\
	& C(v_1) :=  - \int_{\mathbb{T}} \int_0^\pi w^2 v_1 \sin^2(x) \dbar x \dbar t \,,\\
	& D(v_1) :=  2 \int_{\mathbb{T}} \int_0^\pi v_1^2 \left(\mathcal{L}_1^{-1} - \mathcal{L}_\omega^{-1} \right)v_1^2 \,  \sin^2(x) \,\dbar x \dbar t\,.
	\end{aligned}
	$$
\\[1mm]	
{\bf Estimate of $ A(v_1)$.}
	By Cauchy-Schwarz inequality, 
	Lemma \ref{algebraKer},  Lemma \ref{cor.back.to.w}, \eqref{lem.s.sprime} and \eqref{little.gamma.p2}, we estimate
	\begin{align}
	\nonumber
	|A(v_1)|
	&\leq 2 \left\| \Pi_{V_1}( (v_1 + \mathtt{v}_2)  \tilde w) \right\|_{\cV^0_{t,x}} \|v_1\|_{\cV^0_{t,x}}\\
	\nonumber
	&\lesssim_{\delta} (N^{\frac 1 2 + \delta}\| v_1\|_{\cV^{1}_{t,x}} + N^{-\frac 1 2 - \delta}\|\mathtt{v}_2\|_{\cV^{2 + 2\delta}_{t,x}}) \|\tilde w\|_{H^{\frac 1 2 + \delta}_t \cH^{\frac 3 2 + \delta}_x} \|v_1\|_{\cV^1_{t,x}}\\
	&
	\label{stima.a}
	\lesssim_{\delta} \gamma^{-2} R N^{\frac 7 2 +7\delta} \sqrt{\varepsilon}  \|v_1\|^4_{\cV^1_{t,x}}\,.
	\end{align}
\noindent
{\bf Estimate of $ B(v_1) $.} We claim that 
	\begin{equation}\label{stima.b}
	|B(v_1)| \lesssim_{\delta} \gamma^{-2} N^{-\frac 3 2 -\delta} R^2 \|v_1\|^4_{\cV^1_{t,x}}\,.
	\end{equation}
	By Proposition \ref{lem.strich.Lom}, \eqref{lem.s.sprime}, \eqref{di.v2.p2} and \eqref{little.gamma.p2},
	\begin{align}
	\nonumber
	|B_1(v_1)|
	&\lesssim_{\delta} \gamma^{-1} \|v_1\|^2_{\cV^{1}_{t,x}}  N^{-\frac 3 2 - \delta}   \|\mathtt{v}_2\|_{\cV^{2+2\delta}_{t,x}} \big(\|v_1\|_{\cV^{1}_{t,x}} + N^{-\frac 3 2 -\delta} \|\mathtt{v}_2\|_{\cV^{2+2\delta}_{t,x}} \big)\\
	\label{stima.a2}
	&\lesssim_{\delta} \gamma^{-2} N^{-\frac 3 2 - \delta} R^2 \|v_1\|^4_{\cV^{1}_{t,x}}\,.
	\end{align}
	Similarly, one gets
	$|B_2(v_1)|
	\lesssim_{\delta} \gamma^{-2}  N^{-\frac 3 2 -\delta} R^2 \|v_1\|^4_{\cV^1_{t,x}}$.
		\\[1mm]	
{\bf Estimate of $ C(v_1) $.} Recalling \eqref{therealw} and using  \eqref{algebra.prop}, Lemma \ref{lemma.sono.pochissimi}, Lemma \ref{cor.back.to.w}, \eqref{scelta.param.p2.1} and \eqref{little.gamma.p2}, it results 
	\begin{equation}\label{stima.c}
	|C(v_1)|   
	\lesssim_{ \delta} \gamma^{-2} N^{4(1+2\delta)} \|v_1\|^5_{\cV^1_{t,x}}
\lesssim_{\delta} \gamma^{-2} \sqrt{\varepsilon} R N^{4(1+2\delta)} \|v_1\|^4_{\cV^1_{t,x}}\,.
\end{equation}
\noindent
{\bf Estimate of $ D(v_1) $.} 
	Using  \eqref{stri.l4}, Lemma \ref{lemma.lom-lun} and  \eqref{algebra.prop}, one has
	\begin{align}
	{|D(v_1)|}&{\lesssim \|v_1^2\|_{\cV^0_{t,x}} \left\| \left( \mathcal{L}_\omega^{-1} - \mathcal{L}_1^{-1} \right)v_1^2 \right\|_{\cV^0_{t,x}}} 
	{\lesssim_\delta\|v_1\|^2_{\cV^{\frac{3}{2}+\delta}_{t,x}} \varepsilon \gamma^{-1} \left\| v_1^2\right\|_{H^{1}_t \cH^{0}_x}} \lesssim_\delta \varepsilon \gamma^{-1} \|v_1\|^2_{\cV^{\frac 3 2 + \delta}_{t,x}}  \| v_1^2\|_{H^{1}_t\cH^{\frac{3}{2}+\delta}_x} \notag \\
	&\lesssim_\delta \varepsilon \gamma^{-1} \|v_1\|^2_{\cV^{\frac{3}{2}+\delta}_{t,x}}  \| v_1\|^2_{\cV^{\frac{5}{2}+\delta}_{t,x}}\lesssim_\delta \varepsilon \gamma^{-1}N^{2(\frac 3 2 +\delta) + 2(\frac 1 2 + \delta)}\|v_1\|^{4}_{\cV^1_{t,x}}  \lesssim_{ \delta} \gamma^{-1} \varepsilon N^{4+4\delta} \|v_1\|^{4}_{\cV^1_{t,x}}\,. \label{stima.d}
	\end{align}
	Combining estimates \eqref{stima.a}, \eqref{stima.b}, \eqref{stima.c}, \eqref{stima.d}, one gets $\alpha(\breve{\mathcal{R}}_4) \lesssim_{\delta} \gamma^{-3} N^{-\frac 3 2 -\delta} R^2 $. 
\end{proof}

\begin{proof}[Proof of Theorem \ref{teo.v1} for $p=2$]
		We verify  conditions \eqref{alpha.m}. 
		Let $\underline{C}$ be the positive constant defined in Theorem \ref{mountain.bolla}: by Lemma \ref{lem.m.G.p2} and Lemma \ref{lemma.aR.small.p2}, one has 
	$$
	\frac{\alpha(\breve{\mathcal{R}}_4)}{m(\breve{\mathcal{G}_{4}})} \leq \frac{C_{\delta} \gamma^{-2}  N^{-\frac 3 2 - \delta} R^2}{\frac{5}{24}} \leq \underline{C}\,,
	$$
	provided $C_{\delta} \gamma^{-2} N^{-\frac 3 2 - \delta} R^2 \leq \frac{5 \underline{C}}{24}\,,$ which is satisfied due to \eqref{little.gamma.p2}. Furthermore
	$$
	\Big(\frac{|\varepsilon|}{m(\breve{\mathcal{G}_{4}})}\Big)^{\frac 1 2} 
	\leq \Big(\frac{24 \varepsilon^{}}{5 }\Big)^{\frac 1 2} \leq \underline{C} \rho_1 
	\stackrel{\eqref{scelta.param.p2.1}} = \underline{C} R \varepsilon^{\frac 1 2}\,,
$$
	which is satisfied provided $R \geq R_0:=  \underline{C}^{-1} \left(\frac{24}{5}\right)^{\frac 1 2}$. Finally $\|v^{(1)}_{1}\|_{\cV^1_{t,x}} \asymp \varepsilon^{\frac 12}$ 
	by  \eqref{ecco.y} of Theorem \ref{mountain.bolla} and from the fact that $m(\breve{\mathcal{G}_{4}})$ is uniformly bounded in $N$, as stated in Lemma \ref{lem.m.G.p2}.
\end{proof}

\section{Multiplicity of solutions}\label{sec.tante}

In this section we prove multiplicity of solutions.

\begin{theorem}[Multiplicity of solutions with different minimal periods]\label{teo.molte.volte}
	For any $p= 2, 3, 5$ there exists a sequence of integers 
	$\{\tn_k\}_{k\in \N}$ with $\tn_0 := 1$ and $\tn_{k+1} > \tn_k$ 
	for any $k$ such that the following holds. For any $\delta \in (0, \frac{1}{100})$ and $k_* \in \N_*$  there exist $R_{k_*}>0$ and $\epsilon_{k_*,\delta, R}>0$ such that,  if $R, \varepsilon, N$ are as in the assumptions of Theorem \ref{teo.v1} and $R \geq R_{k_*}$, and if \eqref{little.gamma} holds with $\epsilon_{\delta, R} := \epsilon_{k_*, \delta, R}$ in the case $p=5$ (respectively \eqref{little.gamma.p3}
	and \eqref{little.gamma.p2} in the cases $p=3 $ and $ p=2$), 
	there exist $2\pi$-periodic distinct solutions $u^{(1)}, \dots, u^{({k_*})}$ of the form
	\begin{equation}\label{u.n.multi}
	u^{(k)} := v^{(k)}_1 + \mathtt{v}_2\big(v^{(k)}_1\big) + w\big(v_1^{(k)}\big) \quad \text{of} \quad \begin{cases}
	\eqref{main.sph.sym} & \text{ if}\quad  p=5,\ p=2 \, , \\
	\eqref{main.hopf.1} & \text{ if} \quad p=3 \,,
	\end{cases}
	\end{equation}
	with minimal period
	$$
	T_k \in \left\{\frac{2\pi}{\tn_{k}-1 }, \dots,  \frac{2\pi}{\tn_{k-1}}\right\}\,, \quad k = 1, \dots, k_*\,.
	$$
 	The following estimates hold:
	\begin{equation}
	\| v^{(k)}_1\|_{\cV^{1}_{t,z}} \asymp \varepsilon^{\frac {1}{\mathtt{q}}}\,, \quad \mathtt{q} := \begin{cases}
	4 & \text{ if} \quad  p=5\\
	2 & \text{ if} \quad p=2,\ p=3\,,
	\end{cases}
	\end{equation}
	and
	\begin{equation}\label{un.rho.per.ognuno}
	\begin{gathered}
		\|v_1^{(k)}\|_{\cV^1_{t,z}} \leq \rho_1\,, \quad \|\mathtt{v}_2(v_1^{(k)})\|_{\cV^{2+2\delta}_{t,z}} \leq \rho_2\,, \\
	\|w(v_1^{(k)})\|_{H^{\frac 12 + \delta}_t \cH^{\frac 3 2 + \delta}_z} \leq \begin{cases}
	\rho_3 &  \text{ if} \quad p=3, p=5\,,\\
	C_\delta \gamma^{-1}  N^{2(1+2\delta)} \varepsilon R^2 &\text{ if} \quad p=2
	\end{cases}
	\end{gathered}
	\end{equation}
	with $\rho_1, \rho_2, \rho_3$ defined in \eqref{parametri.p5} if $p=5$, respectively in \eqref{parametri.p3} if $p=3$ and in \eqref{scelta.param.p2.1} if $p=2$. 
\end{theorem}

The remaining part of this section is devoted to prove  Theorem \ref{teo.molte.volte}. Since the dependence of the spaces $V_1, V_2$ on the parameter $N$  
plays a significant role, in this section we  denote them respectively by $V_{\leq N}, V_{> N}$.
We  regard equations \eqref{main.sph.sym}, \eqref{main.hopf.1} on the space of $\frac{2\pi}{\tn}$ time periodic functions
$$
X_\tn := \Big\lbrace u(t,z) = \sum_{\ell \in \N} \sum_{j \in \N} u_{\ell, j} \cos(\tn \ell t) \egen_j(z) \Big\rbrace\,.
$$
We define the restrictions to $X_\tn$ of the kernel and range subspaces $V, W, V_{\leq N}, V_{>N}$ defined in \eqref{def.V}, \eqref{def.W}, \eqref{def.V1}, \eqref{def.V2}:
\begin{equation}
	 V_\tn :=V \cap X_\tn\,,\quad  W_{\tn} := W \cap X_{\tn}\,, \quad V_{\leq N, \tn} := V_{\leq N} \cap X_{\tn}\,, \quad V_{> N, \tn} := V_{>N} \cap X_{\tn}\,.
\end{equation}
We note that for any $\tn \in \N_*$ the space $X_\tn$ is left invariant both by the spatial operator $A$ defined in \eqref{def.A} as well as by $\mathcal{L}_\omega$ defined in \eqref{main.common}.

\begin{lemma}[Kernel on $\frac{2\pi}{\tn}$ periodic functions]\label{lem.no.Ln}
A function $v \in V_\tn $ if and only if
	\begin{equation}\label{vn}
 v =
	\begin{cases}
	\displaystyle
	\sum_{\ell \in \N_*} v_\ell \cos(\tn \ell t) e_{\tn \ell -1}(x) & \text{if} \quad p=5,\ p=2\,, \\
	 & \\
	 \displaystyle
\underset{\ell \in \N_* \atop \ell \geq \underline{\mu}\,,\ \tn \ell - \underline{\mu} \text{ is even}}{\sum} v_\ell \cos(\tn \ell t) e^{(\mu_1, \mu_2)}_{ \frac{\tn \ell - \underline{\mu}}{2}}(\eta) & \text{if} \quad p=3\,,
	\end{cases}
	\end{equation}
	where
	$\underline{\mu} := |\mu_1| + |\mu_2| + 1$.
\end{lemma}

\begin{proof}
	If $p=2$ or $p=5$, the thesis follows by \eqref{omgen} and by \eqref{def.V}.
	If $p=3$, it is sufficient to observe that, by \eqref{omgen}, $\omgen_j = 2 j + \underline{\mu}$, thus $\ell = 2j + \underline{\mu}$ for some $j \in \N_*$ if and only if $\ell - \underline{\mu}$ is an even positive number.
	Thus one has $ v \in V$ if and only if
	$$
	v (t, \eta) = \sum_{\ell \in \N \atop \ell - \underline{\mu} \text{ is even}\,, \ell \geq \underline{\mu}} v_{\ell} \cos(\ell t) \egen_{\frac
{\ell - \underline{\mu}}{2}}(\eta)\,,
	$$
	and \eqref{vn} follows restricting to the indexes $\ell$ such that $\ell = \tn \ell'$ for some $\ell' \in \N_*$.
\end{proof}
\begin{remark}
	In the case $p=3$, if $\underline{\mu}$ is odd, then $V_\tn = \{0\}$ for any $\tn$ even, and $V_\tn \neq \{0\}$ if and only if
	$\tn$ belongs to
	\begin{equation}
	\mathcal{Z}^{(\mu_1, \mu_2)} := \begin{cases}
	\N_* & \text{if} \quad \underline{\mu}\quad \text{is even}\\
	\text{odd integers}  & \text{if} \quad \underline{\mu} \quad  \text{is odd}\,.
	\end{cases}
	\end{equation}
\end{remark}

\begin{lemma}\label{rmk.Ln.per.n}
For any $s < s'$ and any $v \in V_\tn$,  one has $\|v\|_{\cV^s_{t,z}} \leq \|v\|_{\cV^{s'}_{t,z}} \tn^{s-s'}$.
\end{lemma}
\begin{proof}
One has $	\|v\|_{\cV^s_{t,z}}^2=
\sum\limits_{\ell \in \N} |v_\ell|^2\left(\tn \ell \right)^{2s} = \tn^{2s} \sum\limits_{\ell \in \N} |v_\ell|^2 \ell ^{2s}\leq \tn^{2s}\sum\limits_{\ell \in \N} |v_\ell|^2 \ell ^{2s'} = \tn^{2(s-s')} \|v\|_{\cV^{s'}_{t,z}}^2$.
\end{proof}

We look for $\frac{2\pi}{\tn}$ periodic solutions of \eqref{v1.eq}-\eqref{w.eq}. The Lyapunov-Schmidt decomposition defined in Sections \ref{sec.v2} and \ref{sec.w} preserves the spaces of $\frac{2\pi}{\tn}$ periodic functions:
\begin{lemma}\label{lem.f.invariant}
	Given $\tn \in \N_*,$ let $\delta, \rho_1, \rho_2, \rho_3, \varepsilon, N \geq \tn$ and $R$ be as in the assumptions of Propositions \ref{prop.v2}, \ref{prop.w} if $p=5$, resp. Propositions \ref{prop.v2.p3}, \ref{prop.w.p3} if $p=3$, and Propositions \ref{prop.v2.p2} and \ref{prop.w.p2} if $p=2$. For any $ v_1 \in V_{\leq N, \tn} \cap \mathcal{D}_{\rho_1}$ let $\mathtt{v}_2(v_1)$ and $w(v_1)$ be the solutions to \eqref{v2.eq} and \eqref{w.eq}; then $\mathtt{v}_2(v_1) \in V_{> N, \tn}$ and $w(v_1) \in W_\tn$.
\end{lemma}
\begin{proof}
 The functions $v_2$ and $w$ are respectively obtained as the fixed point of the contractions $\mathcal{T}_{v_1, w(v_1)}$ and $\mathcal{T}_{v_1}\,,$ defined in \eqref{tv2} and \eqref{tw} in the case $p=5$ (the case $p=3$, $p=2$ are analogous). Then the lemma follows observing that, for any $ v_1 \in V_{\leq N, \tn} \cap \mathcal{D}_{\rho_1}$ and $w \in W_\tn \cap {\mathcal{D}}^W_{\rho_3}$ the operator $\mathcal{T}_{v_1, w}$ maps $V_{> N, \tn}$ into itself, and that for any $  v_1 \in V_{\leq N, \tn} \cap \mathcal{D}_{\rho_1}$, the operator $\mathcal{T}_{v_1}$ maps $W_\tn$ into itself.
\end{proof}

In order to find $\frac{2\pi}{\tn}$ periodic solutions of \eqref{v1.eq.v2.w}, we look for critical points of
\begin{equation}\label{inn}
\breve{\Psi}_\tn := \breve{\Psi} |_{V_{\leq N, \tn} \cap \mathcal{D}_{\rho_1}} \,.
\end{equation}
We remind that $\breve{\Psi}$ has the expansion \eqref{all.together.now} in the cases $p=5,\ p=3$ and \eqref{all.together.now.p2} in the case $p=2$.
\begin{proposition}[Critical point with minimal period]\label{lem.ci.vuole.un.beta}
	For any $\tn \in \N_*$ if $p=2$ or $p=5$, resp. $\tn \in \mathcal{Z}^{(\mu_1, \mu_2)}$ if $p=3$, define
	\begin{equation}\label{alpha.n.m.n}
	\alpha_\tn(\mathcal{R}) := \alpha(\mathcal{R}|_{V_{\leq N, \tn} \cap \mathcal{D}_{\rho_1}})\,, \quad m_\tn(\mathcal{G}):= m(\mathcal{G} |_{V_{\leq N, \tn} \cap \mathcal{D}_{\rho_1}})\,,
	\end{equation}
	where
	\begin{equation}
	\mathcal{G} := \begin{cases}
	\mathcal{G}_{p+1} \text{ defined in } \eqref{big.G}& \text{if} \quad p= 5,\ p=3\\
	\breve{\mathcal{G}_{4}} \text{ defined in } \eqref{big.G2} & \text{if} \quad p=2\,,
	\end{cases}\quad
	\mathcal{R} := \begin{cases}
	\mathcal{R}_{p+1} \text{ as in } \eqref{all.together.now}& \text{if} \quad p= 5,\ p=3\\
	\breve{\mathcal{R}_{4}} \text{ as in } \eqref{all.together.now.p2} & \text{if} \quad p=2\,,
	\end{cases}
	\end{equation}
	and $m(\cdot)$ and $\alpha(\cdot)$ are defined in Theorem \ref{mountain.bolla}. Suppose that there exist $\beta \in (0, 1)$ and $\tm_0>0$ such that 
	for any $ \tm > \tm_0$
	\begin{equation}\label{solo.sup}
	\sup_{v_1 \in V_{\leq N, \tn \tm} \setminus\{0\}} \frac{{\cal G}(v_1)}{\| v_1\|_{\cV^1_{t,z}}^{q}} \leq \beta \sup_{v_1 \in V_{\leq N, \tn} \setminus\{0\}} \frac{{\cal G}( v_1)}{\| v_1\|_{\cV^1_{t,z}}^{q}}\,, \quad q :=
	\begin{cases}
	p+1 & \text{ if } p=3,\ p=5\\
	4 & \text{if } p=2\,.
	\end{cases}
	\end{equation}
	Then there exist positive constants $\underline{C}$ and $C_1:=C_1(\beta)$ such that, if
	\begin{equation}\label{check.me.n}
	\alpha_\tn(\mathcal{R}) \leq C_1 m_\tn(\mathcal{G})\,, \quad \left(\frac{|\varepsilon|}{m_\tn(\mathcal{G})}\right)^{\frac{1}{q-2}} \leq \underline{C} \rho_1\,,
	\end{equation}
	the functional $\breve{\Psi}$ has  a critical point $v^{(\tn)}_{1}$ with minimal period $T_\tn \in \{ \frac{2\pi}{\tm_0\tn}, \dots, \frac{2\pi}{\tn}\}$,  satisfying
	\begin{equation}
	 \| v_1^{(\tn)}\|_{\cV^{1}_{t,z}} \asymp \varepsilon^{\frac {1}{q-2}}\,.
	\end{equation}
	Furthermore $v^{(\tn)}_1$ is also a critical point of $\breve{\Psi}$.
\end{proposition}

\begin{proof}
	By Theorem \ref{mountain.bolla} and \eqref{check.me.n}, $\breve{\Psi}_\tn$ admits a critical point $v^{(\tn)}_{1} \in V_{\leq N, \tn}$ which is proportional to a point $y^{(\tn)}$ satisfying
	\begin{equation}\label{y.critico}
	\| y^{(\tn)}\|_{\cV^1_{t,z}} = 1\,, \quad \mathcal{G}(y^{(\tn)}) =m_\tn( \mathcal{G})+ r^{(\tn)}\,, \quad r^{(\tn)}= \mathcal{O}(\alpha_{\tn}(\mathcal{R}))\,.
	\end{equation}
	We note that, since $r^{(\tn)} = \mathcal{O}(\alpha_\tn(\mathcal{R})),$ provided $\frac{\alpha_\tn(\mathcal{R})}{m_\tn(\mathcal{G})} < C$ with $C = C(\beta)$ small enough, one has
	$(1 - \beta) m_\tn(\mathcal{G}) >  |r^{(\tn)}|$,
	which by \eqref{y.critico} gives
	$$
	\mathcal{G}(y^{(\tn)})= m_\tn( \mathcal{G})+ r^{(\tn)} \geq m_\tn(\mathcal{G}) - |r^{(\tn)}| > \beta m_\tn(\mathcal{G})\,.
	$$
	Combining the latter inequality with hypothesis \eqref{solo.sup}, one gets
	$$
	\frac{\mathcal{G}(y^{(\tn)})}{\|y^{(\tn)}\|^{q}_{\cV^1_{t,z}}} > \beta m_\tn(\mathcal{G}) \geq \sup_{ v_1 \in V_{\leq N, \tn \tm} \setminus\{0\} } 	\frac{\mathcal{G}(v_1)}{\|v_1\|^{q}_{\cV^1_{t,z}}}\,,
	$$
	thus for any $  \tm >\tm_0$ one has that $y^{(\tn)}$ belongs to $V_{\leq N, \tn} \subset V$ but $y^{(\tn)}$ does not belong to $V_{\leq N, \tm \tn}$, namely $y^{(\tn)}$ has minimal period $\geq \frac{2\pi}{\tm_0 \tn}$. Since $v^{(\tn)}_{1}$ and $y^{(\tn)}$ are proportional, the same holds for $v^{(\tn)}_{1}$. It remains to prove that $v^{(\tn)}_1$ is also a critical point for the functional $\breve{\Psi}$. To fix ideas, we prove the result for $p=5$. The cases $p=3, p=2$  follow analogously. By Lemma \ref{lemma.ker.variational}, a point $v_1 \in \mathcal{D}_{\rho_1}$ is critical for $\breve{\Psi}$ if and only if
	\begin{equation}\label{cosa.vogliamo}
	\int_{\mathbb{T}}\int_{0}^{\pi}  \left(A v_1 + \Pi_{V_{\leq N}} 
	{(v_1 +\mathtt{v}_2(v_1) + w(v_1))^5}\right) h \sin^2(x)\, \dbar x \dbar t= 0 \quad \forall h \in V_{\leq N}\,.
	\end{equation}
	Since $v^{(\tn)}_{1}$ is critical for $\breve{\Psi}_\tn$, one already has that \eqref{cosa.vogliamo} holds for $h \in V_{\leq N, \tn}$, thus it remains to prove it for $h \in V_{\leq N} \cap V_{\leq N, \tn}^\bot$. Then it is sufficient to observe that, by Lemma \ref{lem.f.invariant}, $\mathtt{v}_2(v_1^{(\tn)}) \in V_{>N, \tn}$ and $w(v^{(\tn)}_1) \in W_\tn$, since $v^{(\tn)}_1 \in V_{\leq N,\tn}$. Thus 
	$A v^{(\tn)}_1 + \Pi_{V_{\leq N}} {(v^{(\tn)}_1 +\mathtt{v}_2(v^{(\tn)}_1) + w(v^{(\tn)}_1))^5}$ belongs to $V_{\leq N, \tn}$, namely it is orthogonal to any $h \in V_{\leq N} \cap V_{\leq N,\tn}^{\bot},$ which gives the thesis.
\end{proof}
Theorem \ref{teo.molte.volte} follows from an iterative application of Lemma \ref{lem.f.invariant} and Proposition \ref{lem.ci.vuole.un.beta}. In the next sections we verify the assumptions \eqref{solo.sup} arguing separately for the cases $p=5$, $p=3$, $p=2$.

\subsection{Cases $p=5$ and $p=3$}

	We start with $p=5$ and we prove lower and upper bounds for $m_\tn(\mathcal{G}_6)$ defined  in \eqref{alpha.n.m.n}.
	\begin{lemma}[Estimate of $m_\tn(\mathcal{G}_6)$]\label{lem.kn}
		For any $\delta>0$ there exists $C_\delta>0$ and for any $\tn \in \N_*$  there exists $\kappa_\tn>0$ such that
		\begin{equation}\label{m.inn.big.1}
		\kappa_\tn \leq m_\tn(\mathcal{G}_6) \leq C_\delta \tn^{-1 + 6 \delta}\,.
		\end{equation}
	\end{lemma}
	\begin{proof}
	We take $v_\tn = \cos(\tn t) e_{\tn-1}(x),$ 
	then $\|v_\tn\|_{\cV^1_{t,x}} = \tn$.
	One has
	$$
	\mathcal{G}_6(v_\tn) = \frac 1 6 \int_{\mathbb{T}} \cos^6(\tn t)\,\dbar t \int_{0}^\pi e^6_{\tn-1}(x) \sin^2(x)\,\dbar x =: \alpha_\tn>0\,,
	$$
	and the lower bound in \eqref{m.inn.big.1} follows setting $\kappa_\tn := \tfrac{\alpha_\tn}{\tn^6}$. The upper bound follows because for any $v \in V_\tn$ by Proposition \ref{lem.Strichartz.1} and Lemma \ref{rmk.Ln.per.n} we have
	$\mathcal{G}_6(v) \lesssim_\delta \|v\|^6_{\cV^{\frac 5 6 + \delta}_{t,x}} \lesssim_\delta \tn^{-1 + 6 \delta} \|v\|^6_{\cV^1_{t,x}}$.
\end{proof}

\begin{proof}[Proof of Theorem \ref{teo.molte.volte} for $p=5$]
	By Lemma \ref{lem.kn} with $\delta = \frac{1}{12}$, there exist $C>0, \kappa_\tn >0$ such that 
	$$
	\sup_{v \in V_{\tn \tm} \setminus\{0\}} \frac{\mathcal{G}_6( v)}{\| v\|^6_{\cV^1_{t,\eta}}} \leq  \frac{C}{ (\tn \tm)^{\frac 1 2}} \leq \frac{1}{2}  \kappa_\tn \leq \frac 1 2 m_\tn(\mathcal{G}_6)\,,
	$$
	provided $\tm \geq \tm_0(\tn) := \frac{1}{\tn}\big(\frac{C}{2 \kappa_\tn}\big)^{2}$. 
	This proves that for any $\tn, \tm \in \N$ there exists $\tm_0 = \tm_0(\tn) \in \N$ such that, if $\tm > \tm_0(\tn)$ and $N \geq \tn$, one has
	$$
	\sup_{v \in V_{\tm \tn} \setminus\{0\}} \frac{{\cal G}_6( v)}{\| v\|_{\cV^1_{t,\eta}}^{6}} \leq \frac 1 2\sup_{ v_1 \in V_{\leq N,\tn} \setminus\{0\}} \frac{{\cal G}_6(v_1)}{\| v_1\|_{\cV^1_{t,\eta}}^{6}}\,,
$$
	namely \eqref{solo.sup} is satisfied.
	We then define $\tn_0 := 1$, $\tn_{k+1} := \tm_0(\tn_{k}) \tn_k + 1$ and we apply Proposition \ref{lem.ci.vuole.un.beta} with $\tn = \tn_k$ for any $k=1, \dots, k_*$. In particular, assumptions \eqref{check.me.n} hold for any $\tn_k$, observing that $\alpha_{\tn_k}(\mathcal{R}_6) \leq \alpha(\mathcal{R}_6) \lesssim_\delta N^{-\frac{7}{6} + 9\delta} R^4$ by Lemma \ref{lem.R.grad} and $m_{\tn_k}(\mathcal{G}_6) \geq \kappa_\tn$ by Lemma \ref{lem.kn}, and assuming $R \geq R_{k_*} = \max_{k} (\kappa_{\tn_k}\underline{C})^{-1} $ and $N\lesssim_{ \delta} \inf_{k}(R^4\kappa_{\tn_k})^{\frac{1}{\frac 7 6 - 9\delta}}$, which is ensured by \eqref{little.gamma}.
	Thus by Proposition \ref{lem.ci.vuole.un.beta} the functional $\breve{\Psi}$ admits a critical point $v^{(k)}_1 := v_1^{(\tn_k)}$ with minimal period $T_{\tn_k} \in \{\frac{2\pi}{\tm_0(\tn_k) \tn_k}, \dots, \frac{2\pi}{\tn_k}\}$. Finally, by Lemma \ref{lem.f.invariant} $u^{(k)} := v_1^{(k)} + \mathtt{v}_2\big(v_1^{(k)}\big) + w\big(v_1^{(k)}\big)$ has the same minimal period $T_{\tn_k}$.
\end{proof}

\begin{remark}\label{rmk.si.puo.dare.di.piu}
	With careful estimates  on $m(\mathcal{G}_6)$ one can obtain $\tn_{k + 1} = 3 \tn_k + 1$ for any $k$.
\end{remark}

For $p=3$, defining $m_\tn(\mathcal{G}_4)$ as in \eqref{alpha.n.m.n}, we prove the following:
\begin{lemma}[Estimate of $m_\tn(\mathcal{G}_4)$]\label{lem.kn.p3}
	There exists $C>0$ and for any $\tn \in \mathcal{Z}^{(\mu_1, \mu_2)}$ there exists $\kappa_\tn := \kappa_\tn(\mu_1, \mu_2)>0$ such that
	\begin{equation}\label{m.inn.big}
	\kappa_\tn \leq m_\tn(\mathcal{G}_4) \leq C \tn^{-1}\,.
	\end{equation}
\end{lemma}
\begin{proof}
	We take $v_\tn(t, \eta) =\cos(\tn \underline{\mu}t) e_{\frac{(\tn-1)\underline{\mu}}{2}}(\eta)$.
	Then
	$\|v_\tn\|^4_{\cV^1_{t,\eta}} = \tn^4 \underline{\mu}^4$ and
	$$\mathcal{G}_4(v_\tn) = \frac{1}{4}\int_{\mathbb{T}} \cos^4
	\left(\tn \underline\mu t\right) \dbar t  \int_{0}^{\frac{\pi}{2}}\big(e^{(\mu_1, \mu_2)}_{\frac{\tn \ell - \underline{\mu}}{2}}(\eta)\big)^4 \sin(2\eta)\,d\eta := \alpha_\tn >0\, . $$
Then the lower bound in \eqref{m.inn.big} follows with
	$\kappa_{\tn}(\mu_1, \mu_2) := \frac{\alpha_\tn}{\tn^4 \underline{\mu}^4}$. For the upper bound we observe that, 	by Lemma \ref{lem.lp.hs.easy} and Lemma \ref{lemma.come.vuoi}, for any $ \tn \in \mathcal{Z}^{(\mu_1, \mu_2)}$ and $v \in V_{\tn}$ there exists $C>0$ such that
	$$
	\mathcal{G}_4( v) = \frac{1}{4}
	 \int_{\mathbb{T}}\int_{0}^{\frac{\pi}{2}} v^4(t,\eta) \sin(2\eta)\,d \eta \dbar t  \leq \frac{ C}{4} \| v\|^4_{\cV^{\frac {3}{4}}_{t, \eta}} \leq \frac{C}{4} \tn^{-1} \|v\|^4_{\cV^1_{t,\eta}}\,,
	$$
	by Lemma \ref{rmk.Ln.per.n}.
\end{proof}

\begin{proof}[Proof of Theorem \ref{teo.molte.volte} for $p=3$]
	By Lemma \ref{lem.kn.p3} there exist $C>0$ and $\kappa_\tn := \kappa_\tn(\mu_1, \mu_2)$ such that 
	$$
	\sup_{v \in V_{\tn \tm} \setminus\{0\}} \frac{\mathcal{G}_4( v)}{\| v\|^4_{\cV^1_{t,\eta}}} \leq  \frac{C}{4 \tn \tm} \leq \frac{1}{2}  \kappa_\tn \leq \frac 1 2 m_\tn(\mathcal{G}_4)\,,
	$$
	provided $\tm \geq \tm_0(\tn) := \lfloor\frac{C}{2 \tn \kappa_\tn}\rfloor + 1$. Then for any $\tn, \tm \in \N$ there exists $\tm_0 = \tm_0(\tn) \in \N$ such that, if $\tm > \tm_0$ and $N \geq \tn$, one has
	$$
	\sup_{v \in V_{\tm \tn} \setminus\{0\}} \frac{{\cal G}_4( v)}{\| v\|_{\cV^1_{t,\eta}}^{4}} \leq \frac 1 2\sup_{ v_1 \in V_{\leq N,\tn} \setminus\{0\}} \frac{{\cal G}_4(v_1)}{\| v_1\|_{\cV^1_{t,\eta}}^{4}}\,,
	$$
	namely for any $\tn$  there exists $\tm_0(\tn) \in \N_*$ such that \eqref{solo.sup} is satisfied for any $\tm > \tm_0(\tn)$.
	We then define $\tn_0 := 1$, $\tn_{k+1} := \tm_0(\tn_{k-1}) \tn_k + 1$ and Theorem \ref{teo.molte.volte} follows by Proposition \ref{lem.ci.vuole.un.beta}, with $\tn = \tn_k$ for any $k=1, \dots, k_*$.
\end{proof}

\subsection{Case $p=2$}

\begin{lemma}[Estimate of $m_\tn(\breve{\mathcal{G}}_4)$]\label{lem.Gn.below.p2}
	For any $\delta >0$ there exist 
	$C_\delta, \kappa>0$ and $\underline{\tn}>0$ such that for any 
	$\tn \geq \underline{\tn}$ and  $N \geq \tn$, 
	\begin{equation}\label{mn}
	\frac{\kappa}{\tn^4} \leq {m}_\tn(\breve{\mathcal{G}}_4) := - \inf_{ v_1 \in V_{\leq N,\tn} \setminus\{0\}} \frac{\breve{\mathcal{G}}_4(v_1)}{\| v_1\|_{\cV^1_{t,x}}^4} \leq \frac{C_\delta}{\tn^{2-4\delta}}\,.
	\end{equation}
\end{lemma}
\begin{proof}
	Let $\bar{v}_\tn := \cos(\tn t) e_{\tn-1}(x)$. We compute $\breve{\mathcal{G}_4}(\bar{v}_\tn)$. By \eqref{big.G2}, \eqref{productrule}, using \eqref{L.om.fourier} and Lemma \ref{teo.eigencouples}, one has
	\begin{align*}
	\breve{\mathcal{G}}_4(\bar{v}_\tn)
	&= \frac 1 8 \Big( \frac{1}{4 \tn}\sum_{k=0}^{\tn -1} \frac{1}{2 \tn + 2k+1} +  \frac{1}{8 \tn}\sum_{k=0}^{\tn -1} \frac{1}{2 \tn  - 2k-1} - \sum_{k=0}^{\tn -1} \frac{2}{(2k+1)^2} \Big)\\
	&
	=: \frac 1 8 \left(S_1(\tn) + S_2(\tn) + S_3(\tn) \right)\,.
	\end{align*}
	One has
	$
	S_1(\tn) + S_2(\tn) + S_3(\tn) \leq \frac{1}{8\tn} + \frac{\ln(\tn)}{\tn} - 2 \leq -1 $ if  
	$ \tn\geq \underline{\tn}$ large enough.
	Thus we conclude that
	$$
	{m}_\tn(\breve{\mathcal{G}}_4) = - \inf_{v_1 \in V_{\leq N,\tn} \setminus\{0\}} \frac{\breve{\mathcal{G}}_4(v_1)}{\| v_1\|_{\cV^1_{t,x}}^4} \geq - \frac{	\breve{\mathcal{G}}_4(\bar{v}_\tn)}{\|\bar{v}_\tn\|_{\cV^1_{t,x}}^4} \gtrsim \frac{1}{\tn^4}\,.
	$$
	This proves the lower bound in \eqref{mn}. To prove the upper bound we observe that by Proposition \ref{lem.strich.Lom} and Lemma \ref{rmk.Ln.per.n}, for any 
	$ v \in V_\tn \setminus \{0\} $, 
	$$
	-{\breve{\mathcal{G}}_4(v)} \leq {\left|\breve{\mathcal{G}}_4(v)\right|} ={\frac 1 2\left| \int_{\mathbb{T}} \int_{0}^\pi v^2 \mathcal{L}_1^{-1} v^2 \sin^2(x) \dbar x \dbar t \right|} \lesssim_\delta {\| v\|^4_{\cV^{\frac 1 2 +\delta}_{t,x}}} \lesssim_{\delta} {\tn^{-2-4\delta} \| v\|^4_{\cV^{1}_{t,x}}}\,.
	$$
	Then the second estimate in \eqref{mn} follows since
	$- \inf_{v \in V_{\tn} \setminus\{0\}} \frac{\breve{\mathcal{G}}_4(v)}{\| v\|_{\cV^1_{t,x}}^4} = \sup_{v \in V_\tn \setminus\{0\}} \Big(-\frac{\breve{\mathcal{G}}_4(v)}{\|v\|_{\cV^1_{t,x}}^4} \Big)\,.$
\end{proof}

\begin{proof}[Proof of Theorem \ref{teo.molte.volte} for $p=2$]
	
	By Lemma \ref{lem.Gn.below.p2} with $\delta = \frac{1}{8}$ there exists $\underline{\tn} \in \N$ such that for any $\tn\geq \underline{\tn}$ there exist $C>0$ and $\kappa_\tn>0$ such that 
	$$
	\sup_{v \in V_{\tn \tm} \setminus\{0\}} \frac{\breve{\mathcal{G}}_4( v)}{\| v\|^4_{\cV^1_{t,\eta}}} \leq  \frac{C}{(\tn \tm)^{\frac 3 2}} \leq  \frac 1 2 \kappa_\tn \leq \frac 1 2 m_\tn(\breve{\mathcal{G}}_4)\,,
	$$
	provided $\tm \geq \tm_0(\tn)$, with $\tm_0(\tn)$ such that
	$\frac{C}{ \tn^{\frac 3 2} \tm_0(\tn)^{\frac 3 2}} \leq \frac 1 2\kappa_\tn$. Then for any $\tn \geq \underline{\tn}$ and $\tm \in \N$ there exists $\tm_0 = \tm_0(\tn) \in \N$ such that, if $\tm > \tm_0$ and $N \geq \tn$, one has
	\begin{equation}\label{eq.beta.inn.p2}
	\sup_{v \in V_{\tm \tn} \setminus\{0\}} \frac{{\cal G}_4( v)}{\| v\|_{\cV^1_{t,\eta}}^{4}} \leq \frac 1 2\sup_{ v \in V_{\leq N,\tn} \setminus\{0\}} \frac{{\cal G}_4(v)}{\| v\|_{\cV^1_{t,\eta}}^{4}}\,,
	\end{equation}
	namely for any $\tn\geq \underline{\tn}$  there exists $\tm_0(\tn) \in \N_*$ such that \eqref{solo.sup} is satisfied with $\beta = \frac 1 2$ for any $\tm > \tm_0(\tn)$.
	We then define $\tn_0 := 1$, $\tn_1 := \tm_0(\underline{\tn}) + \underline{\tn} + 1$, $\tn_{k+1} := \tm_0(\tn_{k-1}) \tn_k + 1$ and Theorem \ref{teo.molte.volte} follows by Proposition \ref{lem.ci.vuole.un.beta} with $\tn = \tn_k$ for any $k=1, \dots, k_*$.
\end{proof}

\section{Strong solutions}\label{sec.classical}

In this section we prove higher regularity of the solutions found in Theorem \ref{teo.molte.volte}.

\begin{theorem}[Regularity]\label{teo.strong}
	 Let $R, \varepsilon, N$ as in the assumptions of Theorem \ref{teo.molte.volte} and for any $k_* \in \N_*$ let $\{u^{(k)}\}_{k=1}^{k_*}$ be the functions in \eqref{u.n.multi}.
	 Then for any $r > \frac 1 2$, $s > \frac 3 2$, there exist $\epsilon_{r, s, k_*, R}>0$, $\zeta_{r, s}>0$, $\mathtt{A}_{r, s}>0$ and $\mathtt{B}_{r, s}>0$ such that:
	 
	 \noindent
	 \underline{Cases $p=3, 5$:} If $ \gamma^{-1} \varepsilon N^{\zeta_{r, s}} \leq \epsilon_{r, s, k_*, R}$, each solution
	 $ u^{(k)} $ in \eqref{u.n.multi}  belongs to $  H^r_t \cH^s_z $ and 
	 \begin{equation}\label{norme.alte}
	 \begin{gathered}
	 \|v_1^{(k)}\|_{\cV^{r+s}_{t,z}} \leq C_{1, R, r, s, k_*} \varepsilon^{\frac{1}{p-1}} N^{r+s-1}\,, \quad \|\mathtt{v}_2(v_1^{(k)})\|_{\cV^{r+s}_{t,z}} \leq C_{2, R, r,s, k_*} \varepsilon^{\frac{1}{p-1}} N^{\mathtt{A}_{r, s}}\,, \\
	 \|w(v_1^{(k)})\|_{H^{r}_t \cH^{s}_z} \leq C_{3, R, r,s, k_*} \gamma^{-1} \varepsilon^{\frac{p}{p-1}} N^{\mathtt{B}_{r, s}}\,,
	 \end{gathered}
	 \end{equation}
	 for some positive constants $C_{l, R, r, s, k_*}$, $l=1, 2, 3$. 
	 
	 \noindent
	 \underline{Case $p=2$:} If $ \gamma^{-2} \varepsilon N^{\zeta_{r,s}} \leq \epsilon_{r, s, k_*, R}$, each solution
	 $ u^{(k)} $ in \eqref{u.n.multi}  belongs to  $ H^r_t \cH^s_x$ and 
	 \begin{equation}
	 \begin{gathered}
	 \|v_1^{(k)}\|_{\cV^{r+s}_{t,x}} \leq C_{1, R, r, s, k_*} \varepsilon^{\frac{1}{2}} N^{r+s-1}\,, \quad \|\mathtt{v}_2(v_1^{(k)})\|_{\cV^{r+s}_{t,x}}\leq C_{2, R, r, s, k_*} \varepsilon^{\frac{1}{2}} N^{\mathtt{A}_{r,s}}\,, \\
	 \|w(v_1^{(k)})\|_{H^{r}_t \cH^{s}_x} \leq C_{3, R, r, s, k_*} \varepsilon^{} N^{\mathtt{B}_{r,s}}
	 \end{gathered}
	 \end{equation}
	 for some positive constants $C_{l, R, r, s, k_*}$, $l=1, 2, 3$. 
\end{theorem}

Theorem \ref{teo.strong} immediately implies Theorems \ref{teo.spherical} and \ref{teo.hopf}. Let us prove Theorem \ref{teo.spherical}. 
Theorem \ref{teo.hopf} follows in analogous way. 

\begin{proof}[Proof of Theorem \ref{teo.spherical}]
We prove for $ p = 5 $,  the case $ p = 2 $ follows similarly.
	Let $n, r,s$ and $\mathtt{d}$ as in the assumptions of Theorem \ref{teo.spherical}. For any $k = 1, \dots, k_* = n$ define
	$\phi^{(k)}_\varepsilon(t,x) := u^{(k)}(\omega_\varepsilon t,x)\,,$ 
	$v^{(k)}_\varepsilon := v_1^{(k)} + 
	\mathtt{v}_2(v_1^{(k)}) $,  for any $  k = 1, \dots, n\,.$
	Then, recalling that $N := \varepsilon^{-\frac 1 \beta}$ (see \eqref{parametri.p5}), it is sufficient to choose $\beta := \beta (r, s, \mathtt{d})>1$ such that
	\begin{equation}\label{mal.di.pancia}
	\varepsilon^{\frac{1}{4}} N^{\max\{r + s -1, \mathtt{A}_{r, s}\}} := \varepsilon^{\frac 1 4 - \frac{\max\{r + s -1, \mathtt{A}_{r, s}\}}{\beta_{r,s, \mathtt{d}}}} \leq \varepsilon^{\frac 1 4 - \mathtt{d}}\,,
	\end{equation}
	and the upper bound in \eqref{nero}, \eqref{bianco} follows from \eqref{norme.alte}.  As a consequence of \eqref{mal.di.pancia}, \eqref{norme.alte} and \eqref{buccia.R}, one has
	$\varepsilon^{\frac{1}{4}} \lesssim \|v_\varepsilon^{(k)}\|_{H^r_t \cH^{s}_z} \lesssim \varepsilon^{\frac{1}{4}} N^{\max\{r + s -1, \mathtt{A}_{r, s}\}} \lesssim \varepsilon^{\frac 1 4 - \mathtt{d}}\,,$ proving \eqref{nero} and \eqref{bianco}.
\end{proof}

Theorem \ref{teo.strong}
 is a consequence of the iterative application of the two following lemmata.
	
		\begin{lemma}[Regularity bootstrap for $\mathtt{v}_2$]\label{v2.piumeglio.1}
		Let  $r>\frac{1}{2}$ and $s>\frac{3}{2}$.  
		Assume that 	
		$v_1\in \mathcal{D}_{\rho_1}$,  $\mathtt{v}_2:=\mathtt{v}_2(v_1, w(v_1))$ and $w:= w(v_1)$ satisfy
\begin{align} \label{v1.v2.itera}
& \|v_1\|_{\cV^{s}_{t,z}} \leq \widetilde{\rho}_1\,, \quad \|\mathtt{v}_2\|_{\cV^{s}_{t,z}} \leq \widetilde{\rho}_2\,, \\
& \label{w.small.1}
\|w\|_{H^r_t\mathcal{H}^{s}_z}\leq \max \lbrace \widetilde{\rho}_1, \widetilde{\rho}_2\rbrace.
\end{align}				
Then $ \mathtt{v}_2 $ belongs to $ \cV^{s+2}_{t,z} $ and 
		\begin{equation}\label{tesi.1.1}
		\|\mathtt{v}_2\|_{\cV^{s+2}_{t,z}} \lesssim_{r, s}
		\begin{cases}
		\varepsilon^{-1} \max \lbrace \widetilde{\rho}_1,\,  \widetilde{\rho}_2 \rbrace^p & \text{if} \quad p=3 \, , \, p=5\,,\\
		\varepsilon^{-1} \max \lbrace \widetilde{\rho}_1,\,  \widetilde{\rho}_2 \rbrace\|w\|_{H^r_t\cH^s_x} & \text{if} \quad p=2\, .
		\end{cases}
		\end{equation}
	\end{lemma}
	
\begin{proof}
For cases $ p = 5, 3 $, since  $\mathtt{v}_2$ solves \eqref{v2.di.v1.e.w}, resp. \eqref{v2.di.v1.e.w.p3},
\eqref{Delta.smooth}, Lemma \ref{algebraKer}, 
\eqref{v1.v2.itera},  we get 
	$$
	\|\mathtt{v}_2\|_{\cV^{s+2}_{t,z}}\lesssim_{s,r} \varepsilon^{-1}\sum\limits_{j_1+j_2+j_3=p} \widetilde{\rho}_1^{j_1} \widetilde{\rho}_2^{j_2} \|w\|^{j_3}_{H^r_t\cH^s_x} \stackrel{\eqref{w.small.1}} 
	{\lesssim_{s,r}} \varepsilon^{-1} \max \lbrace \widetilde{\rho}_1,\, \widetilde{\rho}_2 \rbrace^p\,,
	$$
	which gives \eqref{tesi.1.1}. 
	If $p=2$ then 
	$
	\mathtt{v}_2=\varepsilon^{-1} A^{-1}\Pi_{V_2}\left( (2(v_1+\mathtt{v}_2)+w) w\right)
	$ 
	(cf. \eqref{new.v2.2}) and the estimate \eqref{tesi.1.1} follows similarly.	
\end{proof}

\begin{lemma}[Regularity bootstrap for $w$]\label{w.piumeglio.1}
Let  $\lambda>\frac{1}{2}$, $\mu>\frac{3}{2}$. 
	There exists $K_{\lambda, \mu, p}>0$, depending only on $p$ and on the algebra constant $C_{\lambda, \mu}$  in \eqref{algebra.prop}, such that, if $v_1\in \mathcal{D}_{\rho_1}$ and $\mathtt{v}_2(v_1)$ defined in \eqref{v1.eq.v2.w} satisfy
	\begin{equation}\label{cose.piccole.1}
	{\gamma}^{-1}\max \lbrace \|v_1\|_{\cV^{\lambda+\mu}_{t,z}},\,\|\mathtt{v}_2(v_1)\|_{\cV^{\lambda + \mu}_{t,z}} \rbrace^{p-1} < K_{\lambda,\mu,p}\,,
	\end{equation}
	then
	\begin{enumerate}
	\item If $p=5$, resp. $p=3$, then the solution $ w(v_1) $ 
	 of \eqref{w2.eq.v2} 
	found in Proposition \ref{prop.w} 
	(resp. Proposition \ref{prop.w.p3}) belongs to 
	$ H^\lambda_t\cH^\mu_z$ and 
	\begin{equation}\label{senso.1}
	\|w(v_1)\|_{H^\lambda_t\cH^\mu_z}\leq {K^{-1}_{\lambda,\mu,p}} {\gamma}^{-1} \max \lbrace \|v_1\|_{\cV^{\lambda+\mu}_{t,z}},\,\|\mathtt{v}_2(v_1)\|_{\cV^{\lambda + \mu}_{t,z}} \rbrace^{p} \, . 
	\end{equation}
	\item  If $p=2$ then the function 
	$w(v_1) = \tilde{w}(v_1) +
	\mathcal{L}_\omega^{-1}\left(v_1 + \mathtt{v}_2(v_1)\right)^2 $, where  $ \tilde{w}(v_1)$ is the solution of \eqref{trenitalia},
	found  in Proposition \ref{prop.w.p2} belongs to 
	$ H^\lambda_t\cH^\mu_x$ and 
	\begin{equation}\label{senso.p2.1}
	\|{w}(v_1)\|_{H^\lambda_t\cH^\mu_x} \leq \frac{1}{2 K_{\lambda, \mu, 2}} \gamma^{-1} \max \lbrace \|v_1\|_{\cV^{\lambda + \mu}_{t,x}},\,\|\mathtt{v}_2(v_1)\|_{\cV^{\lambda + \mu}_{t,x}}\rbrace^2\,.
	\end{equation}	
	\end{enumerate}
\end{lemma}

\begin{proof}
For brevity we denote $\mathtt{v}_2 := \mathtt{v}_2(v_1)$ and $w := w(v_1)$. Let $p=3,5$.
\\[1mm]
{\sc Step 1.} {\it The sequence 
$ (\bar{w}_{k})_{k \in \N} $ defined 	by 
 $\bar{w}_0 :=0$ and $\bar{w}_{k+1}:=\mathcal{L}^{-1}_{\omega}\Pi_W((v_1+\mathtt{v}_2+\bar{w}_k)^p)$ has limit $ \lim\limits_{k \rightarrow \infty}\bar{w}_k = w $ in $ H^{\frac{1}{2}+\delta}_t\cH^{\frac{3}{2}+\delta}_z $.}
 
 In fact,  arguing as in Lemma \ref{prop.tw.contraction},   the map
	$
	\bar{w}\mapsto \mathcal{L}^{-1}_{\omega}\Pi_W((v_1+\mathtt{v}_2+\bar{w})^p) $
	is a contraction on $\mathcal{D}_{\rho_3}^{W}$. Hence it admits a unique fixed point 
	$ \hat{w} =\lim\limits_{k\rightarrow \infty}\bar{w}_k $ in $ H^{\frac{1}{2}+\delta}_t\cH^{\frac{3}{2}+\delta}_z $ satisfying the equation 
	$\hat{w}=\mathcal{L}^{-1}_{\omega}\Pi_W(v_1+\mathtt{v}_2+\hat{w})^p)$, which is \eqref{w2.eq.v2}.
	It implies  that  $ w = \hat{w} = \lim\limits_{k\rightarrow \infty}\bar{w}_k $.
\\[1mm]
{\sc Step 2.}	
{\it For any  $ k \in \N $ each  $ \bar{w}_k $ satisfies \eqref{senso.1}.}  
 
We proceed by induction.  Clearly  $ \bar w_0 = 0 $ satisfies \eqref{senso.1}.
	Now assume that $\bar{w}_{k-1}$ satisfies \eqref{senso.1}.
	By Lemma \ref{lemma.sono.pochissimi}  and
	 \eqref{algebra.prop}, there exists $C_{\lambda,\mu,p}>0$ such that
\begin{equation}\label{defCin}
	\big\| \mathcal{L}^{-1}_{\omega}\Pi_W (v_1+\mathtt{v}_2+\bar{w})^p
	\big\|_{H^\lambda_t\cH^\mu_z} \leq 
	{\gamma}^{-1} C_{\lambda, \mu, p} \max\left\{ \|v_1\|_{\cV^{\lambda+ \mu}_{t,z}},\  \|\mathtt{v}_2\|_{\cV^{\lambda+ \mu}_{t,z}}, \|w\|_{H^\lambda_t\cH^\mu_z}  \right\}^p
\end{equation}
	Then take $K_{\lambda, \mu ,p} := {C_{\lambda, \mu, p}}^{-1}$ in \eqref{cose.piccole.1}. By \eqref{defCin}, the fact that  
	$\bar{w}_{k-1}$ satisfies \eqref{senso.1}, 
	$$
	\begin{aligned}
	\|\bar{w}_k\|_{H^\lambda_t\cH^\mu_z}&\leq {\gamma}^{-1} C_{\lambda, \mu, p} \max\left\{ \|v_1\|_{\cV^{\lambda + \mu}_{t,z}},\  \|\mathtt{v}_2\|_{\cV^{\lambda+ \mu}_{t,z}},\ \|\bar{w}_{k-1}\|_{H^\lambda_t\cH^\mu_z}\right\}^p\\
	&\leq {\gamma}^{-1} C_{\lambda, \mu, p} \max\left\{ \|v_1\|_{\cV^{\lambda+ \mu}_{t,z}},\  \|\mathtt{v}_2\|_{\cV^{\lambda+ \mu}_{t,z}},\  {K^{-1}_{\lambda, \mu, p}} \gamma^{-1} \max\{\|v_1\|_{\cV^{\lambda+ \mu}_{t,z}},\,\|\mathtt{v}_2\|_{\cV^{\lambda+ \mu}_{t,z}}\}^{p} \right\}^p \\
	& \stackrel{\eqref{cose.piccole.1}} 
	\leq {\gamma}^{-1}  {K^{-1}_{\lambda, \mu, p}}  \max\{\|v_1\|_{\cV^{\lambda+ \mu}_{t,z}},\,\|\mathtt{v}_2\|_{\cV^{\lambda+ \mu}_{t,z}}\}^p\,.
	\end{aligned}
	$$
{\sc Step 3. }{\it Proof of \eqref{senso.1}.}
By Step 2, the bounded sequence $\{\bar{w}_k\}_{k\in 	\mathbb{N}}$ converges up to subsequences to a weak limit  $\bar{w} \in H^\lambda_t\cH^\mu_z$ satisfying 
$\|\bar{w}\|_{H^\lambda_t\cH^\mu_z}\leq  K^{-1}_{\lambda, \mu, p} \gamma^{-1} \max \lbrace \|v_1\|_{\cV^{\lambda+ \mu}_{t,z}},\,\|\mathtt{v}_2\|_{\cV^{\lambda+ \mu}_{t,z}} \rbrace^p$.
	Since $H^{\lambda}_t \cH^{\mu}_z$ is compactly embedded into $H^{\frac{1}{2} + \delta}_t \cH^{\frac 3 2 + \delta}_z$ for $\delta>0$ small enough, and using Step 1, 
	we deduce that $\bar w = w $. This proves that 
$ w $ satisfies  \eqref{senso.1}.	

	We now consider the case $p=2$.
	
	\noindent
	{\sc Step 1.} {\it The sequence 
		$ (\check{w}_{k})_{k \in \N} $ defined 	by 
		$\check{w}_0 :=0$ and 
		\begin{equation}\label{contrae.di.nuovo}
		\begin{aligned}
		\check{w}_{k+1}  
		:= \mathcal{T}(v_1, \mathtt{v}_2, \check{w}_{k})
		& :=\mathcal{L}_\omega^{-1}\Pi_{W} \left( 2(v_1 + \mathtt{v}_2) \left(\mathcal{L}_\omega^{-1}(v_1 + \mathtt{v}_2)^2 + \check w_{k}\right) \right)\\
		& \ + \mathcal{L}_\omega^{-1} \Pi_{W} \left( \left(\mathcal{L}_\omega^{-1}(v_1 + \mathtt{v}_2)^2 + \check w_{k}\right)^2 \right)\,,
		\end{aligned}
		\end{equation}
		has limit $ \lim\limits_{k \rightarrow \infty}\check{w}_k = \tilde  w $ in $ H^{\frac{1}{2}+\delta}_t\cH^{\frac{3}{2}+\delta}_x $.}
	
	Arguing as in Lemma \ref{lem.w.contraz}, the map $\check{w} \mapsto \mathcal{T}(v_1, \mathtt{v}_2, \check{w})$ is a contraction on $\mathcal{D}^W_{\rho_3}$, thus it admits a unique fixed point in $ \mathcal{D}^W_{\rho_3}$ which solves \eqref{trenitalia} and therefore it coincides with $\tilde w$.
	
	\noindent
	{\sc Step 2.} \emph{For any $k \in \N$ the function $\check{w}_k$ satisfies}
	\begin{equation}\label{wtilde.k}
	\|\check{w}_k\|_{H^\lambda_t\cH^\mu_x}\leq \frac 1 4 K_{\lambda,\mu,2}^{-2} {\gamma}^{-2} \max \lbrace \|v_1\|_{\cV^{\lambda + \mu}_{t,x}},\,\|\mathtt{v}_2(v_1)\|_{\cV^{\lambda + \mu}_{t,x}}\rbrace^3\,.
	\end{equation}
	We proceed by induction. Clearly $\check{w}_0 $ satisfies  \eqref{wtilde.k}. Then suppose $\check{w}_k$ satisfies \eqref{wtilde.k}.
	Let $C_{\lambda, \mu, 2}$ the algebra constant in \eqref{algebra.prop} and take $K_{\lambda, \mu, 2} := \frac{1}{32 C_{\lambda, \mu, 2}}$. By Lemma \ref{lemma.sono.pochissimi}, 
	 \eqref{algebra.prop} and \eqref{v1.v2.itera}
	\begin{equation}\label{cacio.e.pepe}
	\|\mathcal{L}_{\omega}^{-1}(v_1+\mathtt{v}_2)^2 \|_{H^\lambda_t\cH^\mu_x}\leq 8 \gamma^{-1}C_{\lambda, \mu,2}\max\lbrace \|v_1\|_{\cV^{\lambda+ \mu}_{t,x}},\, \| {\mathtt v}_2\|_{\cV^{\lambda+ \mu}_{t,x}} \rbrace^2\,.
	\end{equation}
	For any $ k \in \N $ we define 
	$\underline{w}_k := \mathcal{L}_\omega^{-1}(v_1 + \mathtt{v}_2)^2 + \check{w}_k$.  
	By \eqref{contrae.di.nuovo}, Lemma \ref{lemma.sono.pochissimi}, \eqref{algebra.prop} and \eqref{lem.s.sprime} one has
	\begin{equation}\label{temporale}
	\|\check{w}_{k+1}\|_{H^\lambda_t\cH^\mu_x}\leq  2\gamma^{-1}C_{\lambda, \mu, 2} \|\underline{w}_k\|_{H^\lambda_t\cH^\mu_x}\left(4 \max\lbrace \|v_1\|_{\cV^{\lambda+ \mu}_{t,x}},\, \| {\mathtt v}_2\|_{\cV^{\lambda+ \mu}_{t,x}} \rbrace + \|\underline{w}_k\|_{H^\lambda_t\cH^\mu_x}\right)\, . 
	\end{equation}
	By \eqref{cacio.e.pepe}, the inductive assumption,   
	and assumption \eqref{cose.piccole.1}  and $K_{\lambda, \mu, 2} = \frac{1}{32 C_{\lambda, \mu, 2}}$ we have 
	\begin{align}
	\nonumber
	\| \underline{w}_{k}\|_{H^\lambda_t \cH^\mu_x} &\leq 8 \gamma^{-1} C_{\lambda, \mu, 2} \max\lbrace \|v_1\|_{\cV^{\lambda+ \mu}_{t,x}},\, \|{\mathtt v}_2\|_{\cV^{\lambda+ \mu}_{t,x}} \rbrace^2 + \frac 1 4 K_{\lambda, \mu, 2}^{-2} \gamma^{-2} \max\lbrace \|v_1\|_{\cV^{\lambda+ \mu}_{t,x}},\, \|{\mathtt v}_2\|_{\cV^{\lambda+ \mu}_{t,x}} \rbrace^3\\
	\label{arcobaleno}
	&\leq 16 \gamma^{-1} C_{\lambda, \mu, 2} \max\lbrace \|v_1\|_{\cV^{\lambda+ \mu}_{t,x}},\, \|{\mathtt v}_2\|_{\cV^{\lambda+ \mu}_{t,x}} \rbrace^2  \\
	\label{ombrello}
	&= \frac 1 2 \max\lbrace \|v_1\|_{\cV^{\lambda+ \mu}_{t,x}},\, \|{\mathtt v}_2\|_{\cV^{\lambda+ \mu}_{t,x}} \rbrace\,.
	\end{align}
	Then, by \eqref{temporale}, \eqref{arcobaleno}, \eqref{ombrello}, one gets $\| \check{w}_{k+1}\|_{H^\lambda_t \cH^\mu_x} \leq \frac{1}{4} K^{-2}_{\lambda, \mu, 2} \gamma^{-2} \max\lbrace \|v_1\|_{\cV^{\lambda+ \mu}_{t,x}},\, \|v_2\|_{\cV^{\lambda+ \mu}_{t,x}} \rbrace^3$,
	proving the claim.
\\[1mm]
	By Steps 1 and 2 we conclude, as 
	in the cases $ p = 3,5 $, that 
$\tilde{w}$ satisfies  \eqref{wtilde.k}. 
Finally  \eqref{cacio.e.pepe},  \eqref{wtilde.k} for $\tilde w$ and assumption \eqref{cose.piccole.1}  and $K_{\lambda, \mu, 2} = \frac{1}{32 C_{\lambda, \mu, 2}}$ implies \eqref{senso.p2.1}.

\end{proof}

We start proving Theorem \ref{teo.strong} in the cases $p=5$, $p=3$.
Given $\delta \in (0, \frac{1}{100})$, $r_0:= \frac 1 2 + \delta$, and $s_0 := \frac 3 2 + \delta$, define for all $l \geq 0$ the quantities
\begin{gather}
	\label{s.k}
	\sigma_0 := s_0:= \tfrac 3 2 + \delta\,, \quad
	\sigma_{l+1}:=\sigma_l +\tfrac{3}{2}-\delta \,,\\
	\label{a.0.b.0}
	\alpha_0 := \begin{cases}
	10 \delta & \text{ if} \quad p=5\\
	4 \delta & \text{ if} \quad p=3\,,
	\end{cases} \quad a_0 := \alpha_0\,, \\
	\label{al.k}
	\alpha_{l+1}:=p\ \max \lbrace \sigma_{l}-1,\ a_l\rbrace\,, \quad
	a_{l+1}:=\alpha_{l+1}-\tfrac{1}{2}-\delta \,,\\
	\label{b.k}
	b_0:= \begin{cases}
	5 + 11 \delta & \text{ if} \quad p=5\\
	3 + 7\delta & \text{ if}\quad p=3\,,
	\end{cases} \quad
	b_{l+1}:=p \max\lbrace \sigma_{l}+1,\ \alpha_{l+1} \rbrace \,,\\
	\label{c.k}
	\zeta_{l}:=\max \lbrace b_{l}-\max\lbrace \sigma_{l}-1,\, a_{l}\rbrace,\, (p-1)\max \lbrace \sigma_{l}+1,\,  \alpha_{l+1}\rbrace \rbrace\,.
\end{gather}
	\begin{lemma}[Iterative regularity bootstrap]
	Let $p=3, 5$ and $v_1^{(1)}  \in \mathcal{D}_{\rho_1} $ as in Theorem \ref{teo.v1}. For any $l \geq 0$ and $\delta \in (0, \frac{1}{100})$ there exists $\epsilon_{R, \delta, l}>0$ such that, if $\varepsilon, R, N$ are as in Theorem \ref{teo.v1} and
	$0<\gamma^{-1} \varepsilon  N^{\zeta_l} \leq \epsilon_{R, \delta, l}$, then for any $\lambda_l\geq r_0$ and $\mu_l \geq s_0$ such that $\lambda_l + \mu_l = \sigma_l + 2$,
	the function 	$u = v_1^{(1)} + \mathtt{v}_2(v_1^{(1)}) + w(v_1^{(1)})$ 
	 belongs to $ H^{\lambda_l}_t \cH^{\mu_l}_{z}$ and 
	\begin{equation}\label{un.strong}
	\begin{gathered}
	\|v_1^{(1)}\|_{\cV^{\sigma_l + 2}_{t,z}} \leq R \varepsilon^{\frac{1}{p-1}} N^{\sigma_{l} +1}\,, \quad \|\mathtt{v}_2(v_1^{(1)})\|_{\cV^{\sigma_l + 2}_{t,z}}\lesssim_{R, \delta, l} \varepsilon^{\frac{1}{p-1}} N^{\alpha_{l+1}}\,, \\ \|w(v_1^{(1)})\|_{H^{\lambda_l}_t \cH^{\mu_l}_z} \lesssim_{R, \delta, l} \gamma^{-1} \varepsilon^{\frac{p}{p-1}} N^{b_{l+1}}\,.
	\end{gathered}
	\end{equation}
	\end{lemma}

	\begin{proof}
	By \eqref{lem.s.sprime}, $v_1 \in \mathcal{D}_{\rho_1}$ and \eqref{parametri.p5}, \eqref{parametri.p3}, we have the first estimate in \eqref{un.strong} for any $l$.  We  denote $v_1 : = v_1^{(1)}$, $\mathtt{v}_2 := \mathtt{v}_2(v_1^{(1)})$, $w:= w(v_1^{(1)})$. The proof of the second and third inequalities in \eqref{un.strong} proceeds by induction.
	\\[1mm]
	{\sc Initialization.}
	If $l=0$, by \eqref{s.k}, Theorem \ref{teo.v1}, \eqref{lem.s.sprime}  the definition of $\rho_1, \rho_2, \rho_3$ 	
	 in \eqref{parametri.p5}, \eqref{parametri.p3}, \eqref{a.0.b.0},  \eqref{b.k}, for $N$ large enough we have
	\begin{equation}\label{soglie.j0}
	\begin{gathered}
	\|v_1\|_{\cV^{\sigma_0}_{t,z}} \leq N^{\sigma_0 -1}\rho_1 = 
N^{\sigma_0 -1} 	R \varepsilon^{\frac{1}{p-1}} \,, \quad \|\mathtt{v}_2\|_{\cV^{\sigma_0}_{t,z}} \leq \rho_2 = \tc_{2}(\delta) R^{p} \varepsilon^{\frac{1}{p-1}} N^{a_0}\,,\\
	\|w\|_{H^{r_0}_t \cH^{\sigma_0}_z} \leq \rho_3 \leq \gamma^{-1} R^p \varepsilon^{\frac{p}{p-1}} N^{b_0}\,.
	\end{gathered}
	\end{equation}
	{\sc Claim 1 :  $\|\mathtt{v}_2\|_{\cV^{\sigma_0 + 2}_{t,z}} \lesssim_{R, \delta} \varepsilon^{\frac{1}{p-1}} N^{\alpha_1}$. } We apply Lemma \ref{v2.piumeglio.1} with $s \leadsto \sigma_0$, $r \leadsto r_0$, $\widetilde{\rho}_1 \leadsto N^{\sigma_0 - 1}\rho_1$, $\widetilde{\rho}_2 \leadsto \rho_2$. By \eqref{soglie.j0}, taking $\gamma^{-1}\varepsilon N^{\zeta_0} \lesssim_{R, \delta} 1$ with $\zeta_0$ defined in \eqref{c.k}, one deduces \eqref{v1.v2.itera}, \eqref{w.small.1}.
	Thus the claim follows by \eqref{tesi.1.1}, recalling that $\alpha_1 \geq p (\sigma_0 -1)$ by \eqref{al.k}. 
\\[1mm]	
{\sc Claim 2 :}  {\it 
For any  $\lambda_0 + \mu_0 = \sigma_0+ 2$, with $\lambda_0 \geq r_0$ and $\mu_0 \geq s_0$, the function $w$ satisfies \eqref{un.strong} for $l=0$.} 
	 We apply Lemma \ref{w.piumeglio.1}. 
	The assumption \eqref{cose.piccole.1} is satisfied since, by \eqref{lem.s.sprime}, \eqref{parametri.p5}, \eqref{parametri.p3} and Claim 1, one has
	$$
	\gamma^{-1} \max \lbrace \|v_{1}\|_{\cV^{\sigma_0+2}_{t,z}},\|\mathtt{v}_2\|_{\cV^{\sigma_0+2}_{t,z}} \rbrace^{p-1}\lesssim_{R} \gamma^{-1} \varepsilon N^{(p-1)\max\lbrace \sigma_0+1, \alpha_1\rbrace} \stackrel{\eqref{c.k}}{\lesssim_{R}} \gamma^{-1} \varepsilon N^{\zeta_0}
	\leq  \! \!\inf_{\lambda_0 \in [r_0, \sigma_0 + 2] \atop\mu \in [s_0, \sigma_0 +2]} K_{\lambda_0, \mu_0, p} 
	$$ 
provided $\gamma^{-1} \varepsilon N^{\zeta_0}$ is small enough.
	Thus by \eqref{senso.1}, \eqref{lem.s.sprime}, $v_1 \in \mathcal{D}_{\rho_1}$ and Claim 1 one gets
	$$
	\|w\|_{H^{\lambda_0}_t \cH^{\mu_0}_z} \lesssim_{R, \delta} \gamma^{-1} \max\{\varepsilon^{\frac{1}{p-1}} N^{\sigma_0 +1} , \varepsilon^{\frac{1}{p-1}} N^{\alpha_1}\}^{p} \stackrel{\eqref{b.k}}{\lesssim_{R, \delta}} \gamma^{-1} \varepsilon^{\frac{p}{p-1}} N^{b_1}\,,
	$$ which is the second estimate in
 \eqref{un.strong} for $l=0$.
\\[1mm]
{\sc Induction.}	
	 We now assume that \eqref{un.strong} holds for $l-1$ and we prove it for $l$. 
	 
	 \noindent
	{\sc Claim $l$1 : $ \|\mathtt{v}_2\|_{\cV^{\sigma_{l}+2}_{t,z}}\lesssim_{R, \delta, l} \varepsilon^{\frac{1}{p-1}}N^{\alpha_{l+1}}$.} 	Assuming $ \gamma^{-1} \varepsilon N^{\zeta_l} \lesssim_{R, \delta, l} 1$, 
	choosing $\lambda_{l-1}=r_0$ and $\mu_{l-1}=\sigma_{l-1} + 2 - \lambda_{l-1} =\sigma_{l}$, by induction hypothesis and using \eqref{lem.s.sprime}, \eqref{s.k}, \eqref{al.k} we have
	\begin{equation}\label{inductive}
	\begin{gathered}
	\|v_1\|_{\cV^{\sigma_{l}}_{t,z}} \leq R \varepsilon^{\frac{1}{p-1}}N^{\sigma_{l} - 1}\,, \quad
	\|\mathtt{v}_2\|_{\cV^{\sigma_{l}}_{t,z}} \lesssim_{R, \delta, l} \varepsilon^{\frac{1}{p-1}}N^{a_l}\,,\\
	\|w\|_{H^{r_0}_t\cH^{\sigma_l}_z}\lesssim_{R, \delta, l} \gamma^{-1}\varepsilon^{\frac{p}{p-1}}N^{b_l}\,.
	\end{gathered}
	\end{equation}
	We apply  Lemma \ref{v2.piumeglio.1} with $s = \sigma_{l}$, $\widetilde{\rho}_1 = R \varepsilon^{\frac{1}{p-1}}N^{\sigma_l-1}$ and $\widetilde{\rho}_2 \gtrsim_{R, \delta, l} \varepsilon^{\frac{1}{p-1}}N^{a_l}$. Then by \eqref{inductive}  the assumptions \eqref{v1.v2.itera} and \eqref{w.small.1} are satisfied, taking $\gamma^{-1}\varepsilon N^{\zeta_l} \lesssim_{R, \delta, l} 1$, with $\zeta_l$ given by \eqref{c.k}.
	Thus \eqref{tesi.1.1} and \eqref{al.k} imply the claim.
	\\[1mm]
	{\sc Claim $l$2 :}
	\emph{For any  $\lambda_l + \mu_l = \sigma_l+ 2$ with $\lambda_l \geq r_0$ and $\mu_l \geq s_0$, the function $w$ satisfies \eqref{un.strong}.}
	We  apply Lemma \ref{w.piumeglio.1} with $\lambda=\lambda_l$ and $\mu=\mu_l$. Indeed, since $\gamma^{-1}\varepsilon N^{\zeta_l} \lesssim_{R, \delta, l} 1$ by \eqref{c.k}, and using \eqref{parametri.p5}, \eqref{parametri.p3}, and Claim $l1$, assumptions \eqref{cose.piccole.1} of Lemma \ref{w.piumeglio.1} is satisfied, and thus by  \eqref{senso.1}   $\|w\|_{H^{\lambda_l}_t\cH^{\mu_l}_z}\lesssim_{R, \delta, l} \gamma^{-1} \varepsilon^{\frac{p}{p-1}}N^{b_{l+1}}$. This concludes the inductive step.
	\end{proof}
	
The proof of Theorem \ref{teo.strong} for $p=2$ follows with similar arguments:
given $\delta \in (0, \frac{1}{100})$, $r_0 \geq \frac 1 2 + \delta$, define for any $l \geq 0$ the quantities
\begin{gather}
\label{s.k.p2}
\sigma_0 := s_0 := \tfrac 3 2 + \delta \,, \quad \sigma_{l+1}:=\sigma_l +\tfrac{3}{2}-\delta\,,\\
\label{a.0.b.0.p2}
 \alpha_0 := 0\,, \quad a_0 := 0\,,  \quad  b_0:= 2+5\delta\,,\\
\label{al.k.p2}
\alpha_{l+1}:= b_l+\max \lbrace \sigma_{l}-1,\ a_l\rbrace + 1\,, \quad
a_{l+1}:=\alpha_{l+1}-\tfrac{1}{2}-\delta\,, \quad 
b_{l+1}:=2 \max\lbrace \sigma_{l}+1,\ \alpha_{l+1} \rbrace\,,\\
\label{c.k.p2}
\zeta_{l}:=2\max \lbrace b_l-\max\lbrace \sigma_{l}-1,\, a_{l}\rbrace,\, \max \lbrace \sigma_l+1,\,  \alpha_{l+1}\rbrace \rbrace\,.
\end{gather}
Then one proves the following:
	\begin{lemma}[Iterative regularity bootstrap]
	Let $p=2$ and  $v_1^{(1)} \in \mathcal{D}_{\rho_1}  $ as in Theorem \ref{teo.v1}. For any $l \geq 0$ and $\delta \in (0, \frac{1}{100})$ there exists $\epsilon_{R, \delta, l}>0$ such that, if $\varepsilon, R, N$ are as in Theorem \ref{teo.v1} and
	$0<\gamma^{-2} \varepsilon  N^{\zeta_l} \leq \epsilon_{R, \delta, l}$, then for any $\lambda_l\geq r_0$ and $\mu_l \geq s_0$ such that $\lambda_l + \mu_l = \sigma_l + 2$ the function $u = v_1^{(1)} + \mathtt{v}_2(v_1^{(1)}) + w(v_1^{(1)})$ belongs to $  H^{\lambda_l}_t \cH^{\mu_l}_{x}$ and 
	\begin{equation}\label{un.strong.p2}
	\begin{gathered}
	\|v_1^{(1)}\|_{\cV^{\sigma_l + 2}_{t,x}} \leq R \varepsilon^{\frac{1}{2}} N^{\sigma_{l} +1}\,, \quad \|\mathtt{v}_2(v_1^{(1)})\|_{\cV^{\sigma_l + 2}_{t,x}}\lesssim_{R, \delta, l} \varepsilon^{\frac{1}{2}} N^{\alpha_{l+1}}\,, \\ \|w(v_1^{(1)})\|_{H^{\lambda_l}_t \cH^{\mu_l}_x} \lesssim_{R, \delta, l} \gamma^{-1} \varepsilon N^{b_{l+1}}\,.
	\end{gathered}
	\end{equation}
\end{lemma}

\footnotesize 

\end{document}